\documentclass[11pt,reqno]{article}

\usepackage{amsthm,amsmath,amssymb}
\usepackage[T1]{fontenc}
\usepackage[utf8]{inputenc}
\usepackage[english]{babel}
\usepackage{graphicx}
\usepackage[dvipsnames]{xcolor}
\usepackage[pagebackref,breaklinks,colorlinks=true,linkcolor=MidnightBlue,citecolor=MidnightBlue]{hyperref}
\usepackage[shortlabels]{enumitem}
\usepackage{tikz}
\usepackage{hyperref}
\usepackage{caption}
\usepackage{mathtools}
\usepackage{nicefrac}
\usepackage{cleveref}
\usepackage{crossreftools}

\usepackage{fl}
\usepackage[textwidth=2cm, textsize=footnotesize]{todonotes}

\theoremstyle{definition}
\newtheorem{definition}{Definition}[section]
\newtheorem{example}[definition]{Example}
\newtheorem{remark}[definition]{Remark}

\theoremstyle{plain}
\newtheorem{theorem}[definition]{Theorem}
\newtheorem*{theorem*}{Theorem}
\newtheorem{lemma}[definition]{Lemma}
\newtheorem{prop}[definition]{Proposition}
\newtheorem{corollary}[definition]{Corollary}

\numberwithin{equation}{section}

\theoremstyle{plain}
\newtheorem{algorithm}{Algorithm}
\newtheorem*{algorithm*}{Algorithm}

\newtheorem{manualalginner}{Algorithm}
\newenvironment{manualalg}[1]{%
  \renewcommand\themanualalginner{#1}%
  \manualalginner
}{\endmanualalginner}

\AtBeginDocument{ \def\MR#1{} }     % Remove MR entry in BibTeX

\newcommand{\xb}{{\bar x}}
\newcommand{\xt}{{\tilde x}}

\newcommand{\yt}{{\tilde y}}
\newcommand{\yb}{{\bar y}}
\newcommand{\yh}{{\hat y}}
\newcommand{\ft}{{\tilde f}}

\newcommand{\ib}{{\bar\imath}}
\newcommand{\jb}{{\bar\jmath}}
\newcommand{\kb}{{\bar k}}
\newcommand{\lb}{{\bar\ell}}
\newcommand{\mb}{{\bar m}}

\newcommand{\sbar}{{\bar s}}

\DeclareMathOperator{\KL}{KL}

\newcommand{\X}{\mathsf{X}}
\newcommand{\Y}{\mathsf{Y}}
\newcommand{\Z}{\mathsf{Z}}

\providecommand{\scalh}[2]{\langle{#1},{#2}\rangle}

\newcommand{\ps}[1]{\langle #1 \rangle}

\setlist[enumerate,1]{font=\normalfont}
\renewcommand{\S}{\mathfrak{S}}
\DeclareMathOperator{\cexp}{c-exp}

\title{Gradient descent with a general cost}
\author{
  Flavien Léger\thanks{Corresponding/first author. INRIA Paris, \textsc{France} (\texttt{flavien.leger@inria.fr}).}, \ \ 
  Pierre-Cyril Aubin-Frankowski\thanks{  INRIA and Département d’Informatique, École Normale Supérieure, PSL Research University (\texttt{pierre-cyril.aubin@inria.fr})}
}

\date{\today}

\begin{document}
\maketitle

We present a new class of gradient-type optimization methods that extends vanilla gradient descent, mirror descent, Riemannian gradient descent, and natural gradient descent. 
Our approach involves constructing a surrogate for the objective function in a systematic manner, based on a chosen \emph{cost function}. This surrogate is then minimized using an alternating minimization scheme. 
Using optimal transport theory we establish convergence rates based on generalized notions of smoothness and convexity. 
We provide local versions of these two notions when the cost satisfies a condition known as nonnegative cross-curvature.
In particular our framework provides the first global rates for natural gradient descent and the standard Newton's method.

\setcounter{tocdepth}{2}
\tableofcontents

% Introduction

\section{Introduction}

The present work lies at the intersection of three recurring questions in optimization. 
\begin{enumerate}[(i)]
    \item How to extend classical methods and their convergence theory beyond the Euclidean setting? 
    
    \item How to integrate different families of algorithms into a unified framework?
    
    \item How to systematically approximate optimization problems? 
\end{enumerate}
As a starting point for (i), let us focus on one of the simplest and most important algorithms in optimization, the gradient descent method
\begin{equation} \label{eq:intro:gd}
    x_{n+1}-x_n = -\frac{1}{L}\nabla f(x_n).
\end{equation}
Here $f$ is an objective function that we want to minimize and $L>0$ is a parameter. The fact that $L$ is the only parameter to choose in \eqref{eq:intro:gd} accounts for the simplicity of gradient descent and gives us a practical general-purpose algorithm. However in many situations we may want to tailor the algorithm more closely to $f$, or there may be additional structure or geometry to exploit. In such cases it is natural to use a variant of \eqref{eq:intro:gd}.

Let us bring up three popular generalizations of gradient descent: mirror descent, natural gradient descent, and Riemannian gradient descent. These iterations can be written respectively as 
\begin{align}
    \nabla u(x_{n+1})-\nabla u(x_n) &= -\nabla f(x_n), \label{eq:intro:md}\\
    x_{n+1}-x_n &= -\nabla^2u(x_n)^{-1}\nabla f(x_n),\label{eq:intro:ngd} \\
    -\nabla_x \,d^2(x_n,x_{n+1})/2 &=-\frac1L\nabla f(x_n). \label{eq:intro:rd}
\end{align}
Each of these extensions already provides much more room to tailor the algorithm to the objective function: through the choice of a convex function $u(x)$ for \eqref{eq:intro:md} and \eqref{eq:intro:ngd}, and of a Riemannian distance $d(x,y)$ for \eqref{eq:intro:rd}.
In this work we go one step further and introduce an extension of \eqref{eq:intro:gd} where in some sense the Euclidean norm $\norm{x-y}^2$ is replaced with a general cost function $c(x,y)$ of two variables.

Let us now focus on the point (ii) mentioned at the beginning.
Often in optimization we are faced with many similar methods and it is of great interest to know whether certain methods are equivalent and can be grouped together. As an example, estimate sequences \cite[Chapter 2.2.1]{Nesterov2018} offer an abstract setting to analyze many different algorithms. So does the ``approximate duality gap technique'' introduced by \cite{Diakonikolas2019}. A similar objective was pursued in \cite{dAspremont2021}, where the authors study accelerated first-order methods using ``systematic templates''. Our method brings together \eqref{eq:intro:md}--\eqref{eq:intro:rd} and other nonlinear versions of gradient descent under a common framework, and provides a unified convergence theory.

Concerning point (iii), problem approximation involving \emph{surrogate} functions can be done through \emph{majorization}--\emph{minimization algorithms}, see the many examples in \cite{lange2016mm}. It is for instance the underlying rationale behind the Expectation--Maximization (EM) algorithm in statistics. Majorization often happens by adding an extra variable, a procedure sometimes known as augmentation \cite{deLeeuw1994}. Majorization also occurs in all the descent lemmas of the optimization community, though the role of the extra variable is not often stressed, to the exception of accelerated settings \cite[p.4]{Drusvyatskiy2018}. It is key to many approximation procedures such as Tikhonov's regularization and proximal methods \cite{Beck2009}. \cite[Section 2]{Mairal2015} presents many such upper-bounds, but there the way to build the surrogates is problem-specific and many assumptions boil down to smoothness and strong convexity. There is a need to provide a systematic way of building surrogates in the non-Euclidean case and, most of all, have generic assumptions to prove that a given algorithm will successfully minimize the surrogate.

\paragraph{Short summary of the contents of this paper.} 
Let $f\colon X\to\R$ be an objective function that we would like to minimize over $X\subset \Rd$. While the function $f$ and its domain $X$ are given to us and shouldn't be modified, we can choose freely a second set $Y\subset\Rd$ and a function $c(x,y)$ ($x\in X$, $y\in Y)$. With the help of the $c$-transform $f^c(y)=\sup_{x \in X} f(x)-c(x,y)$ (Definition~\ref{def:c-transform}), the inequality 
\begin{equation} \label{eq:intro:surrogate}
    f(x)\le\phi(x,y)\coloneqq c(x,y)+f^c(y)
\end{equation}
automatically holds for all $x\in X$, $y\in Y$. We interpret \eqref{eq:intro:surrogate} as majorizing $f$ by the \emph{surrogate} function $\phi$. Our main algorithm (Algorithm~\ref{algo}) is then defined as an alternating minimization of the surrogate:

\begin{algorithm*}[Gradient descent with a general cost]
    \begin{align}
    y_{n+1} &= \argmin_{y\in Y} c(x_n,y) + f^c(y), \label{algo:intro:step1}\\
    x_{n+1} &= \argmin_{x\in X} c(x,y_{n+1}) + f^c(y_{n+1}). \label{algo:intro:step2}
\end{align}
\end{algorithm*}

In optimal transport, $f$ is said to be \emph{$c$-concave} (Definition~\ref{def:c-concave}) if $f(x)=\inf_{y} c(x,y)+f^c(y)$, i.e.\ if there is no gap between $f$ and the lower envelope of the surrogate $\phi$. In optimization terms $c$-concavity is a \emph{smoothness} property (Section~\ref{sec:examples:mirror-descent}). Then \Cref{prop:alg_grad} shows that when $f$ is $c$-concave, \eqref{algo:intro:step1}--\eqref{algo:intro:step2} can be written in the more explicit form
\begin{align}
    -&\nabla_x c(x_n, y_{n+1})=-\nabla f(x_n), \label{algo:intro:step1_grad}\\
    &\nabla_x c(x_{n+1}, y_{n+1}) =0. \label{algo:intro:step2_grad}
\end{align}    
Next we introduce a new property that generalizes standard convexity: we say that $f$ is $c$-cross-convex (Definition~\ref{def:sc}) if
\begin{equation*}\label{eq:intro:sc}
    f(x) -f(x_n) \geq  c(x,y_{n+1})-c(x,y_n)+ c(x_n,y_n)-c(x_n,y_{n+1}).
\end{equation*}
We establish convergence rates for \eqref{algo:intro:step1_grad}--\eqref{algo:intro:step2_grad}, or \eqref{algo:intro:step1}--\eqref{algo:intro:step2}, in Theorem~\ref{thm:cv_rates}, stated partially here:

\begin{theorem*}
    Suppose that $f$ is $c$-concave and $c$-cross-convex. Then
    \begin{equation}\label{eq:intro:descent_sublinear}
        f(x_n)-f(x_*)\le \frac{c(x_*,y_0)-c(x_0,y_0)}{n},
   \end{equation}   
   where $x_*=\argmin_Xf$.
\end{theorem*}

We also prove linear rates under a stronger version of cross-convexity. 
To obtain these rates we develop a new convergence theory for alternating minimization of a general function $\phi(x,y)$, 
\begin{equation} \label{eq:intro-am}
    \begin{aligned}
        y_{n+1} &= \argmin_{y\in Y}\phi(x_n,y)\\
        x_{n+1} &= \argmin_{x\in X} \phi(x,y_{n+1}),
    \end{aligned}
\end{equation}
since  \eqref{algo:intro:step1}--\eqref{algo:intro:step2} is of this form.
In 1984, Csiszár and Tusnády introduced the five-point property (Definition~\ref{def:fpp}) 
\begin{equation} \label{eq:intro-fpp}
    \phi(x,y_{n+1}) + \phi(x_n,y_n) \leq \phi(x,y) + \phi(x,y_n).
\end{equation}
When $\phi(x,y)=c(x,y)+f^c(y)$ the five-point property is connected to both cross-convexity and $c$-concavity of $f$.
We show that \eqref{eq:intro-fpp} yields sublinear rates for \eqref{eq:intro-am} (\Cref{thm:am-rates}).
We also obtain linear rates under a strong five-point property. 

In general \eqref{eq:intro-fpp} is nonlocal. When $\phi$ has \emph{nonnegative cross-curvature} (Definition~\ref{def:cross-curvature}), we provide the following ``semi-local'' criterion for \eqref{eq:intro-fpp}: if $F(x)\coloneqq \inf_y\phi(x,y)$ is convex on certain paths called $c$-segments (Definition~\ref{def:c-segments}) then $\phi$ satisfies \eqref{eq:intro-fpp} (\Cref{thm:sc-fpp-am}).

This situation is applicable to many schemes of interest, studied in Section~\ref{sec:examples}: gradient descent, mirror descent, natural gradient descent, Newton's method, POCS, Sinkhorn and EM. In particular for the natural gradient descent iteration 
\begin{equation}\label{eq:intro-ngd}
    x_{n+1}-x_n=-\nabla^2u(x_n)^{-1}\nabla f(x_n),
\end{equation}
(and Newton's method for which $u=f$), $c$-concavity takes the form 
\[
    \nabla^2f \le \nabla^3u\big(\nabla^2u(x)^{-1}\nabla f,-,-\big) + \nabla^2u,
\]
and cross-convexity is the convexity of the function $f\circ\nabla u^*$ (\Cref{lemma:ngd-local-characterizations}). These two properties imply the first global convergence rates for \eqref{eq:intro-ngd} (\Cref{thm:sec-ngd-cv}) and Newton's method (\Cref{thm:sec-Newton-cv}).

\paragraph{Main contributions.}
\begin{enumerate}[(a)]
    \item We show that the \emph{five-point property} of Csiszár and Tusnády gives convergence rates for alternating minimization.

    \item We present a new approach to obtain majorizing surrogates for an objective function, based on a cost function $c$ and optimal transport theory. We then introduce a new optimization method expressed as an alternating minimization of this surrogate. This method extends and unifies vanilla gradient descent, mirror descent, Riemannian gradient descent, natural gradient descent, and Newton's method.
    
    \item We show that the Euclidean notion of \emph{smoothness} generalizes to a notion known as $c$-concavity in optimal transport. Furthermore we introduce cross-convexity, a notion that generalizes classical convexity.
    
    \item In general $c$-concavity and cross-convexity are nonlocal properties that may prove hard to obtain. Under a condition on $c$ known as nonnegative cross-curvature, a local characterization of $c$-concavity is known in optimal transport. Under the same condition we propose a semi-local criterion for cross-convexity.
    
     \item We establish sublinear and linear convergence rates under $c$-concavity and cross-convexity.

    \item Our framework provides a unified convergence theory for gradient descent, mirror descent, Riemannian gradient descent, natural gradient descent, Newton's method, forward-backward methods, and alternating minimization methods such as projections onto convex sets, Sinkhorn and EM. In particular we obtain the first global rates for natural gradient descent and Newton's method.

\end{enumerate}

The remainder of the paper is organized as follows. In Section~\ref{sec:am} we present new results on alternating minimization, with a review in Section~\ref{sec:background:mtw} of the necessary optimal transport material to obtain the semi-local criterion in Section~\ref{sec:am-sufficient-fpp}. In \Cref{sec:gd}, we introduce our gradient descent algorithm, as well as a more general forward--backward version in Section~\ref{sec:fb}. Finally several examples are studied in detail in \Cref{sec:examples}.

% Alternating minimization

\section{Alternating minimization}\label{sec:am}

In this section we present new convergence rates for the alternating minimization method (Algorithm~\ref{algo:am} below). The theory developed in this section paves the way to the convergence rates for our gradient descent algorithm in Section~\ref{sec:gd}, but is also of independent interest.

We first revive in Section~\ref{sec:am-rates} an inequality introduced in 1984 by Csiszár and Tusnády \cite{CsiszarTusnady1984} under the name ``five-point property'', and show that it directly implies a sublinear convergence rate for the values of the minimized function (Theorem~\ref{thm:am-rates}). 
We review in Section~\ref{sec:background:mtw} the optimal transport material that we then use in Section~\ref{sec:am-sufficient-fpp} to design a semi-local criterion for the five-point property.

Our starting point is a function $\phi\colon X\times Y\to\R$, where $X$ and $Y$ are two (open) subsets of $\Rd$, or more generally two $d$-dimensional manifolds. In Section~\ref{sec:am-rates}, $X$ and $Y$ can even be taken to be any general sets. We suppose that $\phi$ is bounded below on $X\times Y$ and are interested in the minimization problem
\begin{equation}\label{eq:prob_joint}
    \min_{x\in X,y\in Y} \phi(x,y).
\end{equation}

The \emph{alternating minimization} procedure consists in successively minimizing $\phi$ over one variable while keeping the other one fixed. It is also known as (block) coordinate descent or nonlinear Gauss--Seidel. In this section we will always assume the following:
\begin{enumerate}[series=ass,label=\normalfont\textbf{(A\arabic*)}]
    \item For each $x\in X$, the function $y\mapsto \phi(x,y)$ attains its minimum at a unique point in  $Y$, and for each $y\in Y$, the function $x\mapsto \phi(x,y)$ attains its minimum at a unique point in $X$. \label{ass:phi-min-min}
\end{enumerate}
Under \ref{ass:phi-min-min} the alternating minimization algorithm can be written as follows.

\begin{algorithm}[Alternating minimization] \label{algo:am}
    Initialize $x_0\in X$ and alternate the following two steps,
\begin{align}
    y_{n+1} &= \argmin_{y\in Y}  \phi(x_n,y),\label{eq:am:step1}\\
    x_{n+1} &= \argmin_{x\in X} \phi(x,y_{n+1}).\label{eq:am:step2}
\end{align}
\end{algorithm}

\subsection{Convergence rates under the five-point property} \label{sec:am-rates}

In the context of alternating minimization, Csiszár and Tusnády introduced in~\cite{CsiszarTusnady1984} the \emph{five-point property}. We recall it below and also define a strong version of it; the five-point property will be related to sublinear rates while its stronger form will be related to linear rates (see Theorem~\ref{thm:am-rates}). 

In this section $X$ and $Y$ can be general sets. Let us then define two maps that correspond to the $y$- and $x$-updates of Algorithm~\ref{algo:am},
\begin{equation}\label{eq:def-T-S}
    \begin{aligned} 
        T_\phi(x) = \argmin_{y\in Y} \phi(x,y),\\
        S_\phi(y) = \argmin_{x\in X} \phi(x,y).
    \end{aligned}        
\end{equation}
These two maps are well-defined under \ref{ass:phi-min-min}.

\begin{definition}[Five-point property] \label{def:fpp}
    We say that $\phi$ satisfies the five-point property if for all $x\in X,y,y_0\in Y$ we have
    \begin{equation} \label{fpp} \tag{FP}
        \phi(x,y_1) + \phi(x_0,y_0) \leq \phi(x,y) + \phi(x,y_0),
    \end{equation}
    where $x_0\coloneqq S_\phi(y_0)$ and $y_1\coloneqq T_\phi(x_0)$. Additionally let $\lambda > 0$. 
    We say that $\phi$ satisfies the $\lambda$-strong five-point property if for all $x\in X,y,y_0\in Y$ we have
    \begin{equation}\label{sfpp} \tag{$\lambda$-FP}
        \phi(x,y_1) + (1-\lambda)\phi(x_0,y_0) \leq \phi(x,y) + (1-\lambda)\phi(x,y_0),
    \end{equation}
    with $x_0\coloneqq S_\phi(y_0)$ and $y_1\coloneqq T_\phi(x_0)$. 
\end{definition}

A few observations are in order. 
(i) The name ``five-point property'' comes from the fact that five different points appear in \eqref{fpp}: $x,y,y_0,x_0,y_1$. However note that only the first three of these points can be chosen freely while $x_0$ and $y_1$ are both defined from $y_0$. 
(ii) In the strong version \eqref{sfpp} we let $\lambda>0$ but the inequality is only ever used for $0<\lambda<1$; this is similar to what happens for classical $\lambda$-strongly convex functions. Indeed if $\lambda\geq 1$ and $(x_*,y_*)$ is a global minimizer of $\phi$, then \eqref{sfpp} implies $\phi(x_*,y_1)\leq \phi(x_*,y_*) - (\lambda-1)(\phi(x_*,y_0)-\phi(x_0,y_0))\leq \phi(x_*,y_*)$. This says that the algorithm always converges in two steps. 
(iii) The five-point property is a property of $\phi$ only and doesn't use any structure of $X$ and $Y$, be it metric or topological. This is why in this subsection $X$ and $Y$ can be taken to be general sets.
(iv) Assumption \ref{ass:phi-min-min} requires the sets $\argmin_Y \phi(x,\cdot)$ and $\argmin_X \phi(\cdot,y)$ to be nonempty and to reduce to a singleton. Out of these two properties only the nonemptiness is important: in many cases the requirement of reducing to a single element could be removed and we would take for instance $y_{n+1}\in\argmin_y\phi(x_n,y)$ in place of \eqref{eq:am:step1}. This is because the focus of this work is on the values $\phi(x_n,y_n)$ rather than the iterates $(x_n,y_n)$ themselves.

\begin{remark}[Csiszár--Tusnády and Byrne]\label{rmk:csiszar}
    The five-point property \eqref{fpp} corresponds to the one reported by Byrne in his review~\cite{Byrne2012} (with a change of index due to a difference between our notation and his). Oddly enough it is slightly different from the one originally introduced by Csiszár and Tusnády in~\cite{CsiszarTusnady1984}, which is
    \begin{equation}\label{eq:5pp-CT}
        \phi(x,y_1) + \phi(x_0,y_1) \leq \phi(x,y) + \phi(x,y_0).
    \end{equation}
    (The term $\phi(x_0,y_1)$ replaces $\phi(x_0,y_0)$.) It is clear that \eqref{eq:5pp-CT} is weaker than \eqref{fpp} since $\phi(x_0,y_1) \leq \phi(x_0,y_0)$. When $\phi$ is a Kullback--Leibler divergence, Csiszár and Tusnády show that \eqref{eq:5pp-CT} holds by obtaining a ``three-point property'' and a ``four-point property'' which when combined jointly imply \eqref{fpp} and therefore their five-point property \eqref{eq:5pp-CT}. Thus \eqref{eq:5pp-CT} is not paramount to the analysis of \cite{CsiszarTusnady1984}.
 \end{remark}

We are now ready to state our first result, which shows that a simple rewriting of the five-point properties \eqref{fpp} and \eqref{sfpp} implies convergence rates on the values $\phi(x_n,y_n)$. We set
\[
    \phi_* = \inf_{x\in X,y\in Y}\phi(x,y),
\]
and assume it to be finite and attained at $(x_*,y_*)\in X\times Y$.

\begin{theorem}[Convergence rates for alternating minimization] \label{thm:am-rates}    
    Suppose that $\phi$ satisfies \ref{ass:phi-min-min} and consider Algorithm~\ref{algo:am}. Then, the following statements hold.
    \begin{enumerate}[(i)]
        
        \item For all $n\geq 0$, $\phi(x_{n+1},y_{n+1}) \leq \phi(x_n,y_{n+1}) \leq \phi(x_n,y_n)$.
        
        \item Suppose that $\phi$ satisfies~\eqref{fpp}. Then for any $x\in X,y\in Y$ and any $n\geq 1$,
        \[
            \phi(x_n,y_n) \leq \phi(x,y) + \frac{\phi(x,y_0) - \phi(x_0,y_0)}{n}.
        \] 
        In particular $\phi(x_n,y_n)-\phi_* = O(1/n)$. 
        
        \item Suppose that $\phi$ satisfies \eqref{sfpp} for some $\lambda\in(0,1)$. Then for any $x\in X,y\in Y$ and any $n\geq 1$, 
        \[
            \phi(x_n,y_n) \leq \phi(x,y) + \frac{\lambda[\phi(x,y_0)-\phi(x_0,y_0)]}{\Lambda^n-1},
        \]
        where $\Lambda\coloneqq(1-\lambda)^{-1}>1$. In particular $\phi(x_n,y_n)-\phi_* = O((1-\lambda)^n)$. 
        \end{enumerate}     
\end{theorem}

\begin{proof} 
    (i): This ``descent property'' immediately follows from the two minimization steps of Algorithm~\ref{algo:am}.

    (ii): After rearranging terms, the five-point property can be written as 
    \[
        \phi(x_{n+1},y_{n+1}) \leq \phi(x,y) + [\phi(x,y_n)-\phi(x_n,y_n)] - [\phi(x,y_{n+1})-\phi(x_{n+1},y_{n+1})].
    \]
    The $x$-update given by~\eqref{eq:am:step2} tells us that the quantities inside brackets are nonnegative. 
    Summing from $0$ to $n-1$ and using (i) we obtain 
    \[
        n \phi(x_n, y_n) \leq n \phi(x,y) + [\phi(x,y_0)-\phi(x_0,y_0)] - [\phi(x,y_n)-\phi(x_n,y_n)],
    \]
    and the right-hand side can be bounded by $n \phi(x,y) + [\phi(x,y_0)-\phi(x_0,y_0)]$. Dividing by $n$ we obtain the desired result.

    (iii): Similarly to (ii), \eqref{sfpp} can be written as 
    \[
        \phi(x_{n+1},y_{n+1}) \leq \phi(x,y) + (1-\lambda)[\phi(x,y_n)-\phi(x_n,y_n)] - [\phi(x,y_{n+1})-\phi(x_{n+1},y_{n+1})].
    \]
    Dividing both sides by $(1-\lambda)^{n+1}$ and summing from $0$ to $n-1$, we use (i) in the left-hand side and note that the right-hand side contains a telescopic sum. We are left with 
    \[
        \Big(\sum_{k=0}^{n-1} \Lambda^{k+1} \Big) \phi(x_n,y_n) \leq \Big(\sum_{k=0}^{n-1} \Lambda^{k+1} \Big) \phi(x,y) + [\phi(x,y_0)-\phi(x_0,y_0)],
    \]
    with $\Lambda=(1-\lambda)^{-1}$. An explicit computation of the geometric sum leads to the desired inequality.
\end{proof}

\paragraph{Related work.} In~\cite{CsiszarTusnady1984} Csiszár and Tusnády show that when \eqref{eq:5pp-CT} holds, without any further assumption on $X$, $Y$, or $\phi$, the alternating minimization method converges to its infimum, in the sense that $\phi(x_n,y_n)\to \phi_*$ as $n\to \infty$. But generally alternating minimization is studied in a more structured context where $X$ and $Y$ are Euclidean spaces and $\phi$ is convex. In this setting, and for more than two variables, i.e.\ cyclic block coordinate descent, \cite{Schechter1962,Auslender1971,Grippof1999} studied the global convergence of the iterates. Beck and Tetruashvili~\cite[Theorem 5.2]{Beck2013} were the first to show a sublinear rate of convergence, assuming in addition $\phi$ to be $L$-smooth. For functions of the form $\phi(x,y)=c(x,y)+h(y)+g(x)$, with $c,h,g$ convex and only $c$ $L$-smooth, Beck~\cite[Theorem 3.7]{Beck2015} also gave a sublinear rate. We refer to \cite{Wright2015} for a review on coordinate descent algorithms. Removing the assumption of convexity, the convergence of proximal iterates of alternating minimization has been studied by \cite{Attouch2007-sj,Attouch2010,Bolte2013} assuming that a Kurdyka-Łojasiewicz inequality holds. Examples of applications of such models to game theory and PDE domain decomposition were given in \cite[Sections 1.3 and 5]{Attouch2008alternating}. The specific case where the cost $c(x,y)$ is a Bregman divergence was considered in great detail in \cite{BausckeCombettesNoll2006}. There the authors gave a rigorous foundation to the problem and studied convergence of the iterates, without giving rates. We will return to Bregman alternating minimization in \Cref{sec:alt-bregman-prox}.

\begin{remark}[Other writings of the five-point property]\label{rmk:fpp_interpretation}
    Setting $F(x):=\inf_{y\in Y} \phi(x,y)$, \eqref{fpp} gives us a lower bound on $F$, for all $x \in X$, $y,y_0\in Y$,
   \begin{equation}\label{eq:bounds_f_phi}
       \phi(S_\phi(y_0),y_{0})+\phi(x,T_\phi \circ S_\phi (y_0)) -\phi(x,y_0)  \leq F(x) \le \phi(x,y),
   \end{equation}
   In a majorization-minimization style similar to \cite[Section 5]{Byrne2012}, setting $G_n(x)=\phi(x,y_n)-F(x)$, \eqref{sfpp} can be written as follows, for all $x\in X$, $n\in \N$, 
   \begin{equation}\label{eq:bounds_maj_min}
       F(x_n)+G_n(x_n) + \frac{1}{1-\lambda}G_{n+1}(x) -G_{n}(x)\leq F(x) \le F(x)+G_{n+1}(x),% \forall x,n.
   \end{equation} 
   There is no gap at $x_n$ between the bounds. Moreover, for $\lambda=0$, \eqref{fpp} directly implies that the gap ($\phi(x_*,y_n)-\phi(x_n,y_n)$) at $x_*$  diminishes as $n$ increases. This can be further discussed through estimate sequences \cite[Section 2.2.1]{Nesterov2018}, roughly interpretable as convex combinations of upper and lower bounds \cite{Baes2009estimate,Drusvyatskiy2018}. Indeed setting $\phi_n(x)=\phi(x,y_{n+1})$, and 
    \begin{equation}\label{eq:rate_estimate_seq}
        \alpha_n=\max\left(\lambda \frac{\phi(x_*,y_{n+1})-\phi(x_{n+1},y_{n+1})}{(1-\lambda)(\phi(x_*,y_{n})-\phi(x_{n},y_{n}))}, \frac{\phi(x_{n+1},y_{n+1})-f_*}{\phi(x_*,y_{n})-\phi(x_{n},y_{n})}\right),
    \end{equation}
    we have that $\alpha_n\in[0,1]$ and $(\phi_n,\lambda_n)_{n\in\N}$ is a weak estimate sequence at $x_*$ for  $\lambda_{n}=\Pi_{i=0}^{n-1}(1- \alpha_i)$, which entails $f(x_n)-f_*=O(\lambda_n)$. We detail this point in \Cref{lem:estimate-seq} in the Appendix, and refer to \cite[p6]{zhang2018estimate} for the definition of weak estimate sequences as a relaxation of \cite[p84]{Nesterov2018}.

    A related way of interpreting \eqref{fpp} is to see it defines a Lyapunov potential function $V_n$ for alternating minimization (see \cite{Bansal2019} for a review of use cases, \cite[Section 4.4, p55]{dAspremont2021} for the connection with estimate sequences). As a matter of fact, for any given $(x,y)\in X\times Y$, set
    \begin{equation}\label{eq:lyapunov}
        V_n(x',y')\coloneqq n(\phi(x',y')-\phi(x,y))+\phi(x,y')-f_*.
    \end{equation}
    By \eqref{fpp}, we have that $V_{n+1}(x_{n+1},y_{n+1})\le V_{n}(x_{n},y_{n})$, which induces the sublinear rate $n(\phi(x_{n},y_{n})-\phi(x,y))\le V_0(x,y_{0})$.
\end{remark}

\subsection{Background on cross-curvature} \label{sec:background:mtw}

This section contains the needed optimal transport material to establish our semi-local criterion for the five-point property (Theorem~\ref{thm:sc-fpp-am}). To start with we suppose that $c\in C^4(X\times Y)$, and will give more precise assumptions later in \ref{ass:cross-curv}. We start by defining the c-exponential map~\cite{mtw,Loeper2009}. When $x$ and $y$ are fixed, the equation 
\begin{equation} \label{eq:xi-c-exp}
    \xi=-\nabla_xc(x,y)
\end{equation}
defines a tangent vector $\xi$ at $x$ (or more precisely a covector). This vector tells us in which direction $c(\cdot,y)$ decreases the fastest at $x$. 

\begin{definition}[c-exponential map] \label{def:c-exp}
    When considering $x$ and $\xi$ fixed, the point $y$ defined by \eqref{eq:xi-c-exp} is written as 
    \begin{equation}\label{eq:c-exp}
        y=\cexp_x(\xi),
    \end{equation}
    when it exists and is unique, and $\cexp$ is known as the c-exponential map.
\end{definition}
The c-exponential map is a generalization of the Riemannian exponential map~\cite[Chapter 5]{Petersen_book}. To see why, let $(M,\mathsf{g})$ be a Riemannian manifold and take $X=Y=M$ together with the cost $c(x,y)=\frac12d^2(x,y)$, where $d$ is the Riemannian distance. The Riemannian exponential map takes as input a point $x\in M$ and a vector $\xi\in T_xM$ and outputs the point $y=\exp_x(\xi)$ resulting from shooting at $x$ a constant-speed geodesic  with initial velocity $\xi$ until time $1$. In that situation it can be shown that $\xi$ can be recovered from $x$ and $y$ via the relation $\xi=-\nabla_x\frac12d^2(x,y)$, and by analogy inverting \eqref{eq:xi-c-exp} into \eqref{eq:c-exp} motivates the naming of c-exponential.

\begin{example}[Lagrangian costs] \label{ex:lagrangian-cost}
    Suppose that $X=Y$ is a subset of a finite-dimensional vector space and define the cost
    \begin{equation} \label{eq:lagrangian-cost}
        c(x,y)=\inf_q \int_0^1 L(q(t),\dot q(t))\,dt,
    \end{equation}
    where the infimum runs over all paths $q$ joining $x$ to $y$ and where $L(x,v)dt$ is a ``Lagrangian'' function measuring the cost of moving from $x$ to $x+vdt$ during time $dt$. For example the Riemannian distance squared corresponds to $L(x,v)=\mathsf{g}_x(v,v)$ where $\mathsf{g}$ is the Riemannian metric. Under some structure conditions on $L$, in particular strict convexity in the second variables, \eqref{eq:lagrangian-cost} defines a well-behaved cost, see~\cite[Chapter 7]{Villani_book_2009}. 

    Introducing the Hamiltonian $H(x,p)$ which is the convex conjugate of $L$ in the second variable, a solution $q$ to the variational problem~\eqref{eq:lagrangian-cost} then satisfies Hamilton's equations. These are a first-order ordinary differential equation on $(q(t),p(t))$, and the variable $p$ is called the \emph{momentum}. We then have the identity 
    \[
        -\nabla_xc(x,y)=p(0).
    \]
    In other words, the quantity $-\nabla_xc(x,y)$ is the initial momentum of the ``physical'' path joining $x$ to $y$ (the path solving Hamilton's equations). 
\end{example}

The next notion is an extension of straight lines (segments) in Euclidean spaces.

\begin{definition}[$c$-segments] \label{def:c-segments}
    A curve 
    $s\mapsto (x(s),y)$ in $X\times Y$ is said to be a horizontal $c$-segment or simply a $c$-segment if 
    \[
        \frac{d^2}{ds^2}\nabla_yc(x(s),y) = 0.
    \]
    Similarly a curve $t\mapsto (x,y(t))$ is said to be a vertical $c$-segment if 
    \[
        \frac{d^2}{dt^2}\nabla_xc(x,y(t))=0.
    \]
\end{definition}

Here is how we may think about $c$-segments. For a fixed $x\in X$, the object $\xi(t)\coloneqq -\nabla_xc(x,y(t))$ defines a tangent vector at $x$ (or rather a cotangent vector), thus as $t$ varies all the $\xi(t)$ live in the same space and can be compared. Rephrased in terms of the c-exponential map defined by \eqref{eq:c-exp}, if $\ddot\xi(t)=0$, i.e.\ $\xi(t)$ is a classical segment in the tangent space at $x$ (which is a vector space), then $t\mapsto \big(x,\cexp_x(\xi(t))\big)$ is a vertical $c$-segment. By exchanging the roles of $x$ and $y$, horizontal $c$-segments can be understood in a similar way.

\begin{definition}[Cross-curvature] \label{def:cross-curvature}
    Let $(x,y)$ be a point in $X\times Y$, $\xi$ a tangent vector at $x$ and $\eta$ a tangent vector at $y$. Let $s\mapsto (x(s),y)$ be a horizontal $c$-segment starting at $x(0)=x$ with initial velocity $\dot x(0)=\xi$, and let $t\mapsto y(t)$ be any curve starting at $y(0)=y$ with initial velocity $\dot y(0)=\eta$. Then the \emph{cross-curvature} or \emph{Ma--Trudinger--Wang tensor} (MTW for short) is defined by 
    \begin{equation} \label{eq:def-cc}
        \S_c(x,y)(\xi,\eta) = -\left.\frac{\partial^4}{\partial s^2\partial t^2}\right|_{s=t=0} c(x(s),y(t)).
    \end{equation}
    We then say that $c$ has \emph{nonnegative cross-curvature} if for all $x,y,\xi,\eta$,
    \[
        \S_c(x,y)(\xi,\eta) \ge 0.
    \]
\end{definition}

The MTW tensor was introduced by Ma, Trudinger and Wang in \cite{mtw} (with a different multiplicative constant). They identified positivity of $\S_c$ on orthogonal $(\xi,\eta)$ as a key ingredient to obtain regularity for the solutions to the optimal transport problem with cost $c(x,y)$. Later on Kim and McCann \cite{KimMcCann2010} understood the geometric nature of $\S_c$ by reframing it as the curvature of a natural pseudo-Riemannian metric based on $c$. The naming of ``cross-curvature'' is due to them and the coordinate-free form \eqref{eq:def-cc} can be found in~\cite[Lemma 4.5]{KimMcCann2010} (with a different multiplicative constant). We refer to \cite{KimMcCann2010,KimMcCann2012} and \cite[Chapter 12]{Villani_book_2009} to learn more on this topic.

\begin{remark} \label{rem:sym-cc}
    In \eqref{eq:def-cc} we require $x(s)$ to be a $c$-segment. The formula is equally valid if instead $t\mapsto (x,y(t))$ is a $c$-segment and $x(s)$ is any curve.
\end{remark}

We continue and define the \emph{cross-difference} considered by McCann in \cite{McCann_glimpse2014,McCann1999}, 
\begin{equation}\label{eq:def_cross_diff}
    \delta_c(x',y';x,y) \coloneqq c(x,y')+c(x',y)-c(x,y)-c(x',y').
\end{equation}
In optimal transport and matching models~\cite[Chapter 2]{Galichon_book}, $c(x,y)$ measures the cost of matching $x$ with $y$. Then $\delta_c(x',y';x,y)$ can be interpreted as the positive or negative loss incurred by a central planner changing assignment $\{x\to y,x'\to y'\}$ to assignment $\{x\to y',x'\to y\}$. 

Let us record a few elementary but useful properties of the object $\delta_c(-,-,-,-)$. It is linear in $c$, skew-symmetric in the slots $(1,3)$ as well as $(2,4)$ and has the ``interchange symmetry'' $\delta_c(x',y';x,y)=\delta_c(x,y;x',y')$. Additionally $\delta_c$ only depends on the actual interaction between $x$ and $y$ in the sense that if $\phi(x,y)=c(x,y)+g(x)+h(y)$, then 
\begin{equation} \label{eq:prop_cross_diff}
    \delta_\phi(x',y';x,y) = \delta_c(x',y';x,y).
\end{equation}

Before stating the next result, we now make more precise the assumptions needed when working with $c$-segments and cross-curvature:

\begin{enumerate}[resume*=ass] 
    \item $c\in C^4(X\times Y)$ and: \label{ass:cross-curv}
    
    \begin{enumerate}[\normalfont\textbf{(A\arabic{enumi}\alph*)}]
        \item for all $(x,y)\in X\times Y$, the $d\times d$ matrix $\nabla^2_{xy}c(x,y)$ is invertible; \label{ass:non-degenerate}
        
        \item for any $x,x'\in X$ and $y,y'\in Y$, the set $X$ contains a horizontal $c$-segment $(s\in [0,1])\mapsto (x(s),y)$ with endpoints $x(0)=x$ and $x(1)=x'$ and the set $Y$ contains a vertical $c$-segment $(t\in [0,1]) \mapsto (x,y(t))$ with endpoints $y(0)=y$ and $y(1)=y'$. \label{ass:biconvex}
    \end{enumerate}
\end{enumerate}
We require the cost to be four times differentiable since four derivatives are taken in \eqref{eq:def-cc}. Assumption \ref{ass:non-degenerate} is sometimes called non-degeneracy of the cost. Note that it forces $X$ and $Y$ to have the same dimension. The need to invert $\nabla^2_{xy}c$ is clear with coordinate formulations, see \eqref{eq:horizontal-c-segments}, \eqref{eq:vertical-c-segments}, \eqref{eq:cc-coord}. Assumption \ref{ass:biconvex} says that we generally need $X\times Y$  to be ``biconvex'' to ensure that $c$-segments between points do actually exist.

The next result contains the form in which in practice we will use nonnegative cross-curvature. It is a characterization of cross-curvature due to Kim and McCann.

\begin{lemma}[{\cite[Theorem 2.10]{KimMcCann2012}}] \label{lemma:delta-c-segments}
    Suppose that \ref{ass:cross-curv} holds and that $c$ has nonnegative cross-curvature. Fix $x,x'\in X$ and $y,y'\in Y$ and let $(s\in [0,1])\mapsto (x(s),y)$ be a horizontal $c$-segment with endpoints $x(0)=x$ and $x(1)=x'$. Then
    \begin{equation}\label{eq:diff-cost}
        s\mapsto -c(x(s),y')+c(x(s),y) \text{ is a convex function.}
    \end{equation}
    As a direct consequence,
    \begin{equation}\label{eq:delta-and-path}
        \delta_c(x',y';x,y) \geq  -\bracket{\nabla_xc(x,y') - \nabla_xc(x,y),\dot{x}(0)},
    \end{equation}
    where $\dot x(s)$ stands for $\frac{d x(s)}{d s}$.    
\end{lemma}

\begin{remark}
    \cite[Theorem 2.10]{KimMcCann2012} actually shows that nonnegative cross-curvature is \emph{equivalent} to \eqref{eq:diff-cost} holding for all $x,x',y,y'$.
\end{remark}

\begin{proof}[Proof of \Cref{lemma:delta-c-segments}]
    Statement \eqref{eq:diff-cost} is contained in \cite[Theorem 2.10]{KimMcCann2012}. For \eqref{eq:delta-and-path}, let $b(s)=-c(x(s),y')+c(x(s),y)$. By convexity of $b$ we have $b(s)-b(0)\geq s\,b'(0)$. At $s=1$ this is \eqref{eq:delta-and-path}.
\end{proof}

\paragraph{How to compute with cross-curvature?} When doing calculations involving cross-curvature, it is preferable to use coordinates. Following~\cite{KimMcCann2010} we denote $x$-derivatives by unbarred indices: 
\[
    \partial_i\coloneqq \frac{\partial}{\partial x^i}, \quad \partial_{ij}\coloneqq \frac{\partial^2}{\partial x^i\partial x^j},\quad\text{etc},
\] 
while $y$-derivatives are denoted with barred indices:  
\[
    \partial_\ib\coloneqq \frac{\partial}{\partial y^\ib}, \quad \partial_{\ib\jb}\coloneqq \frac{\partial^2}{\partial y^\ib\partial y^\jb},\quad\partial_{i\jb}\coloneqq \frac{\partial^2}{\partial x^i\partial y^\jb}, \quad\text{etc}.
\]
We also write $c_i=\partial_ic$, $c_\ib=\partial_\ib c$, $c_{ij}=\partial_{ij}c$, and so on, and $c^{\jb i}$ denotes the inverse of the matrix $c_{i\jb}$, which exists by \ref{ass:non-degenerate}. Finally we adopt the Einstein summation convention where summation over repeated indices is not explicitly written.

Let us first look at $c$-segments. Computing the $s$-derivatives explicitly we see that $(x(s),y)$ is a horizontal $c$-segment when $\frac{d}{ds}[c_{j\kb}(x(s),y)\dot x^j]=0$, i.e.\ $c_{j\kb}(x(s),y)\ddot x^j + c_{ij\kb}(x(s),y)\dot x^i\dot x^j=0$. This last expression can be written as
\begin{equation} \label{eq:horizontal-c-segments}
    \ddot{x}^k+c^{k\mb}c_{\mb i j} \dot{x}^i\dot{x}^j=0,
\end{equation}
where the quantities involving $c$ are evaluated at $(x(s),y)$. Similarly, $t\mapsto (x,y(t))$ is a vertical $c$-segment when 
\begin{equation} \label{eq:vertical-c-segments}
    \ddot{y}^\kb + c^{\kb m} c_{m \ib\jb}\dot{y}^\ib\dot{y}^\jb=0,
\end{equation}
where the quantities involving $c$ are evaluated at $(x,y(t))$. If \eqref{eq:horizontal-c-segments} and \eqref{eq:vertical-c-segments} look like geodesic equations, it is because they are, see the next paragraph about the Kim--McCann geometry. Cross-curvature can be written in coordinates as 
\begin{equation} \label{eq:cc-coord}
    \S_c(x,y)(\xi,\eta) = (c_{ik\mb}c^{\mb r}c_{r\jb\lb}-c_{i\jb k\lb})\xi^i\eta^\jb\xi^k\eta^\lb.
\end{equation}
This is usually the formulation used when computing the cross-curvature of a cost given by a formula, as in Example~\ref{ex:costs-with-nonnegative-curvature}.

\paragraph{The Kim--McCann geometry.} We have introduced the notions of $c$-segments and cross-curvature directly, without justification. All these concepts are in fact geometric in nature when understood through the lens of the Kim--McCann geometry. Given the triplet $(X,Y,c)$, Kim and McCann introduced in \cite{KimMcCann2010} a pseudo-Riemannian metric on the product space $X\times Y$ given by the cross-derivatives $-\nabla^2_{xy}c$. More precisely, the pseudo-metric at point $(x,y)\in X\times Y$ in direction $(\xi,\eta)\in T_{(x,y)}(X\times Y)$ is given by 
\[
    -\nabla^2_{xy}c(x,y)(\xi,\eta).
\]
Note that this is indeed quadratic in $(\xi,\eta)$. 
This metric is intimately tied to the cross-difference, since as $\xi,\eta\to0$ we have
\[
    \delta_c(x+\xi,y+\eta;x,y) = -\nabla^2_{xy}c(x,y)(\xi,\eta) + o(\abs{\xi}^2+\abs{\eta}^2).
\]
This metric treats $X$ and $Y$ as pure topological manifolds and only depends on $c(x,y)$ and not on any Euclidean or Riemannian structure that may exist on $X$ and $Y$.
Kim and McCann show that $c$-segments are geodesics for this metric. More specifically they are the particular geodesics for which either the $x$- or the $y$-component is constant in time. Moreover Kim and McCann show that cross-curvature is equivalent to the Riemann curvature tensor, in the sense that it is the (unnormalized) sectional curvature of their metric. As a consequence of these results we have the following facts.

\begin{prop}\label{prop:facts-cc}
    Let $c(x,y)$ be a cost on $X\times Y$ satisfying \ref{ass:non-degenerate}.
    \begin{enumerate}[(i)]
        \item Let $\phi(x,y)=c(x,y)+g(x)+h(y)$. Then $\S_\phi=\S_c$. \label{prop:facts-cc:gh}
        \item $\S_c$ is invariant by a change of coordinates on $X$ or on $Y$. In particular formula~\eqref{eq:cc-coord} is valid in any coordinates. \label{prop:facts-cc:invariant}
        \item Define $\tilde c(y,x)=c(x,y)$ on $Y\times X$. Then $\S_c(x,y)(\xi,\eta)=\S_{\tilde c}(y,x)(\eta,\xi)$. \label{prop:facts-cc:symmetry}
    \end{enumerate}
\end{prop}

We conclude this section with a small zoology of costs with nonnegative cross-curvature (proofs in \Cref{lemma:costs-with-nonnegative-curvature} in the appendix). See also~\cite{LeeLi2012} for additional examples based on a Riemannian distance. 

\begin{example}[Costs with nonnegative cross-curvature] \label{ex:costs-with-nonnegative-curvature}
    Let $X,Y$ be open subsets of $\Rd$. 
    \begin{enumerate}[1.]
        \item The quadratic cost $c(x,y)=\norm{x-y}^2$ where  $\norm{\cdot}$ is a Euclidean norm satisfies  $\S_c=0$.
        
        \item More generally, $c(x,y)=\norm{A(x)-B(y)}^2$ satisfies  $\S_c=0$. Here $A\colon X\to\Rd$ and $B\colon Y\to\Rd$ are smooth diffeomorphisms onto their image.
        
        \item\label{ex:costs-with-nonnegative-curvature:bregman} 
        Bregman divergences $c(x,y)=u(x|y)\coloneqq u(x)-u(y)-\bracket{\nabla u(y),x-y}$ satisfy  $\S_c=0$.
        
        \item Let $c(x,y)=\sum_ie^{x_i-y_i}$, or $c(x,y)=\sum_i\nu_ie^{(x_i-y_i)/\eps}$ with $\eps\neq 0$ and $\nu_i\neq 0$. Then $\S_c=0$.    
        
        \item More generally let $c(x,y)=\sum_{ij} e^{(x_i-y_j)/\eps}K_{ij}$ where $K$ is a $d\times d$ invertible matrix and $\eps\neq 0$. Then $\S_c=0$. 
        
        \item The log-divergence $c(x,y)=u(x)-u(y)+\frac1\alpha\log(1-\alpha\bracket{\nabla u(y),x-y})$ satisfies $\S_c(\xi,\eta) = 2\alpha \big[\nabla^2_{xy}c(x,y)(\xi,\eta)\big]^2$ \cite{WongYang2022}. Therefore for $\alpha>0$ we have $\S_c\geq 0$. See \Cref{ex:log-div} for more background on the log-divergence.
        
        \item More generally any cost of the form $c(x,y)=\log\big(A_0+A_1(y)+\bracket{A_2(x),A_3(y)}\big)$ where the functions $A_i$, $1\leq i\leq 3$ are diffeomorphisms, satisfies $\S_c(\xi,\eta) = 2\big[\nabla^2_{xy}c(x,y)(\xi,\eta)\big]^2\geq 0$.
        
        \item The square of the geodesic distance on a sphere has nonnegative cross-curvature \cite{KimMcCann2012}.
        
        \item (Tensor products \cite{KimMcCann2012}) Let $\tilde X$ and $\tilde Y$ be two additional subsets of $\Rd$. Let $c$ be a cost on $X\times Y$ and $\tilde c$ a cost on $\tilde X\times \tilde Y$ and suppose that $c$ and $\tilde c$ both have nonnegative cross-curvature. Then the cost $C((x,\tilde x),(y,\tilde y))=c(x,y)+c(\tilde x,\tilde y)$ on $(X\times \tilde X)\times (Y\times \tilde Y)$ has nonnegative cross-curvature.
    \end{enumerate}
\end{example}

\subsection{A semi-local criterion for the five-point property} \label{sec:am-sufficient-fpp}

The five-point property \eqref{fpp} is a \emph{nonlocal} inequality, which makes it in general hard to prove directly. In this section, by connecting \eqref{fpp} to the rich theory developed in the past two decades in optimal transport (including cross-difference, $c$-segments and cross-curvature) we derive a semi-local criterion for \eqref{fpp} that can be used for a certain class of functions $\phi$. We explain what we mean by ``semi-local'' in the discussion following \Cref{thm:sc-fpp-am}.

First let us draw a parallel to a well-known situation (which we borrow from \cite[Chapter 26]{Villani_book_2009}): standard convexity of a function $q\colon\R\to\R$. The ``synthetic'' definition of convexity is that $q((1-t)a+tb)\leq (1-t)q(a)+tq(b)$ for all $a,b\in\R$ and all $t\in [0,1]$. This is a property which one can use once convexity of $q$ is known but, in general, this is not \emph{how} convexity is obtained for a particular $q$. Instead, if $q$ is twice differentiable, one rather uses the ``analytic'' definition $q''(t)\geq 0$ which is a local inequality and usually much easier to prove when $q$ is given by a formula. Similarly, the five-point property \eqref{fpp} can be thought of as a synthetic definition while Theorem~\ref{thm:sc-fpp-am} offers an analytic criterion.

Let us start by writing our function $\phi$ as the sum of three terms,
\begin{equation} \label{eq:phi-split}
    \phi(x,y)=c(x,y)+g(x)+h(y).
\end{equation}
Note that we can always do this, even if it means taking $g(x)=0$, $h(y)=0$ and $c(x,y)=\phi(x,y)$. 
We introduce $c$ for several reasons. 
First note that if $\phi$ is written as \eqref{eq:phi-split} we have $\nabla^2_{xy}\phi=\nabla^2_{xy}c$. In the previous section we reviewed the notions of cross-difference, $c$-segments and cross-curvature. These are all based on a function $c(x,y)$. Since they can all be connected to the Kim--McCann metric which is defined by the cross-derivatives of the cost $-\nabla^2_{xy}c$, we see that these notions are all invariant under the addition of functions of a single variable such as $g(x)$ or $h(y)$ to the cost $c(x,y)$. In other words $c$-segments and $\phi$-segments are the same thing, $\S_c=\S_\phi$ and $\delta_c=\delta_\phi$. In the same way assumption \ref{ass:cross-curv} is the same for $c$ and $\phi$.

Another reason to write $\phi$ as \eqref{eq:phi-split} is that this is the form of our gradient descent with general cost in Section~\ref{sec:gd}. 
Additionally, many problems come naturally in the form~\eqref{eq:phi-split}, where $c(x,y)$ may represent the ``pure interaction'' between $x$ and $y$ while $g$ and $h$ are ``regularizers'' that ensure that $x$ and $y$ belong to certain spaces.

\begin{theorem}[Sufficient conditions for the five-point property] \label{thm:sc-fpp-am}
    \leavevmode 
    
    \noindent Suppose that \ref{ass:phi-min-min} and \ref{ass:cross-curv} hold. Suppose that $\phi$, or equivalently $c$ in \eqref{eq:phi-split}, has nonnegative cross-curvature. Define $F(x)=\inf_{y\in Y}\phi(x,y)$ and assume that $F$ is differentiable on $X$. 
    \begin{enumerate}[(i)]
        \item If $t\mapsto F(x(t))$ is convex on every $c$-segment $t\mapsto (x(t),y)$ satisfying $\nabla_x\phi(x(0),y)=0$, then $\phi$ satisfies the five-point property \eqref{fpp}.
        
        \item Let $\lambda>0$. If $t\mapsto F(x(t))-\lambda \phi(x(t),y)$ is convex on the same $c$-segments as for \textup{(i)}, then $\phi$ satisfies the strong five-point property \eqref{sfpp}.
    \end{enumerate}
\end{theorem}

We say that \Cref{thm:sc-fpp-am} provides a \emph{semi-local} criterion to obtain the five-point property. 
We use semi-local to emphasize that convexity of $F(x(t))$ is a local property which is practical to check by computing two time-derivatives. On the other hand the $c$-segments of interest have the nonlocal condition $\nabla_x\phi(x(0),y)=0$. 

Before we proceed with the proof of \Cref{thm:sc-fpp-am}, let us write \eqref{fpp} in a different form, using the function $F(x)=\inf_{y\in Y}\phi(x,y)$. The five-point property is 
\begin{equation} \label{eq:thm-sc-fpp-am-0}
    \phi(x,y_1) + \phi(x_0,y_0) \leq \phi(x,y) + \phi(x,y_0),
\end{equation}
for all $x\in X,y,y_0\in Y$, and with $x_0=S_\phi(y_0)$ and $y_1=T_\phi(x_0)$. 
After substracting $\phi(x_0,y_1)=F(x_0)$ on both sides of \eqref{eq:thm-sc-fpp-am-0} and taking an infimum over $y\in Y$ we obtain the following equivalent form of \eqref{eq:thm-sc-fpp-am-0},
\begin{equation} \label{eq:fpp-F}
    F(x) \geq F(x_0) +\delta_\phi(x,y_0;x_0,y_1).
\end{equation}
Similarly \eqref{sfpp} can be written as 
\begin{equation} \label{eq:sfpp-F}
    F(x) \geq F(x_0) +\delta_\phi(x,y_0;x_0,y_1) + \lambda[\phi(x,y_0)-\phi(x_0,y_0)].
\end{equation}
This formulation is directly related to \emph{cross-convexity}, introduced in Section~\ref{sec:gd:rates}.

\begin{proof}[Proof of \Cref{thm:sc-fpp-am}]
    (i): Let $x\in X,y,y_0\in Y$ and set $x_0=S_\phi(y_0)$ and $y_1=T_\phi(x_0)$. Let us establish the five-point property in the form \eqref{eq:fpp-F},
    \begin{equation} \label{eq:thm-sc-fpp-am-1}
        F(x) \geq F(x_0)  +\delta_\phi(x,y_0;x_0,y_1) = F(x_0)  -\delta_\phi(x,y_1;x_0,y_0).
    \end{equation}
    This last equality follows from a skew-symmetry property of $\delta_\phi$, see \eqref{eq:def_cross_diff} and the subsequent discussion.
    Let now $t\mapsto (x(t),y_0)$ be a $c$-segment with endpoints $x(0)=x_0$ and $x(1)=x$. This $c$-segment exists by \ref{ass:biconvex}. Since $\phi$ has nonnegative cross-curvature, we can use \Cref{lemma:delta-c-segments} and obtain that
    \begin{equation} \label{eq:thm-sc-fpp-am-2}
        -\delta_\phi(x,y_1;x_0,y_0) \leq \bracket{\nabla_x\phi(x_0,y_1) - \nabla_x\phi(x_0,y_0),\dot x(0)}.
    \end{equation}
    By the envelope theorem stated in \Cref{lemma:envelope} we have $\nabla_x\phi(x_0,y_1)=\nabla F(x_0)=\nabla F(x(0))$. Additionally $\nabla_x\phi(x_0,y_0)=0$ by definition of $x_0$. 
    The convexity of $F$ along $(x(t),y)$ then implies 
    \begin{equation} \label{eq:thm-sc-fpp-am-3}
        \bracket{\nabla F(x(0)),\dot x(0)}\leq F(x(1))-F(x(0))=F(x)-F(x_0).
    \end{equation}
    Combining~\eqref{eq:thm-sc-fpp-am-2} and \eqref{eq:thm-sc-fpp-am-3} gives us \eqref{eq:thm-sc-fpp-am-1}.

    (ii): We start from \eqref{eq:thm-sc-fpp-am-3} but replacing now $F$ by the function $x\mapsto F(x)-\lambda \phi(x,y_0)$. It gives us
    \[
        \bracket{\nabla F(x_0)-\nabla_x\phi(x_0,y_0),\dot x(0)} \leq F(x)-\lambda\phi(x,y_0)-F(x_0) + \lambda\phi(x_0,y_0).
    \]
    Combined with \eqref{eq:thm-sc-fpp-am-2} it implies  \eqref{sfpp}.
\end{proof}

We use the tools developped in this section to study several alternating minimization examples in Section~\ref{sec:examples}.

% Gradient descent with general cost

\section{Gradient descent with a general cost} \label{sec:gd}

In this section we present a new algorithm which we call \emph{gradient descent with a general cost}. Given an objective function $f$ and its gradient (or differential) $\nabla f(x_n)$ at a current iterate $x_n$, the method chooses the next iterate $x_{n+1}$ with the help of a cost function $c(x,y)$. This function $c$ can be chosen freely under the condition that $f$ be, in a sense, \emph{smooth} with respect to $c$  (specifically that $f$ be $c$-concave, as explained in Section~\ref{sec:gd:main}). We develop in Section~\ref{sec:gd:rates} a convergence theory based solely on the geometry of $c$ and properties of $f$. Gradient descent with a general cost unifies several optimization methods, each corresponding to a specific cost $c(x,y)$: vanilla gradient descent, mirror descent, Riemannian gradient descent, natural gradient descent, Newton's method, and others. Our framework provides convergence rates for all these examples and in particular we obtain the first global non-asymptotic convergence rates for natural gradient descent and Newton's method. Some examples are treated in detail in Section~\ref{sec:examples}. We also present a forward--backward (i.e.\ explicit--implicit) version of our algorithm in Section~\ref{sec:fb}.

\subsection{Main algorithm} \label{sec:gd:main}

Let $X$ and $Y$ be two open subsets of $\Rd$, or two $d$-dimensional manifolds, and let $c\colon X\times Y\to\R$ be a function. Since the quantities we will introduce are intimately tied to optimal transportation, we will refer to $c$ as the \emph{cost function}. Broadly speaking, we are interested in minimizing a function $f\colon X\to\R$ (more generally in Section~\ref{sec:fb} we consider the sum of two functions $f+g$). To start with we define the \emph{$c$-transform}, one of the fundamental objects in optimal transport~\cite[Chapter 5]{Villani_book_2009}.

\begin{definition}[$c$-transform] \label{def:c-transform}    
    The $c$-transform of a function $f\colon X\to\R$ is the function $f^c\colon Y\to\R$ defined by 
        \begin{equation} \label{eq:def-c-transform}
            f^c(y)=\sup_{x\in X} f(x)-c(x,y).
        \end{equation}
\end{definition}

One way to view the $c$-transform is that $f^c(y)$ is the lowest value $\alpha\in\R$ such that the function $x\mapsto c(x,y)+\alpha$ majorizes $f$, i.e.\ its graph lies above the graph of $f$, as in Fig.~\ref{fig:c-transform}. We therefore always have the inequality
\[
    f(x)\leq c(x,y) + f^c(y),
\]
for all $x\in X$ and $y\in Y$. We call $\phi(x,y)\coloneqq c(x,y) + f^c(y)$ the \emph{surrogate} function. Our main algorithm, \emph{gradient descent with a general cost}, then consists in alternatively minimizing $\phi(x,y)$, see \eqref{algo:step1}--\eqref{algo:step2}. It also takes a more explicit form~\eqref{algo:step1_grad}--\eqref{algo:step2_grad} that is often easier to work with.

\begin{figure}[h]
    \centering
    % base cost h(x-y) with h(x)=1/(x+1.5) + (x+1.5)^3/15
    % min at x=
    \begin{tikzpicture}
        % x-axis
        \draw (-3,-0.2) -- (3,-0.2);
        % f(x)
        \draw [thick,domain=-2.1:2.5] plot[smooth] (\x, { (\x+3)^3/46/2 }); 
        % c(x,y)+f^c(y)
        \draw [dashed,domain=-2:2] plot[smooth] (\x, { (\x)^2 / 2 + 1.9-0.55}); 
        % \draw [dashed,domain=-2:2] plot[smooth] (\x, { (\x)^2 / 2 + 2.9-0.55}); 
        
        \draw [thick,domain=-2:2] plot[smooth] (\x, { (\x)^2 / 2 + 0.9-0.55}); 
        
        \draw (3,1.7) node {$f(x)$};
        \draw (-3.7,2.1) node {$x\mapsto c(x,y)+f^c(y)$};
        \draw (-3.5,3.3) node {$x\mapsto c(x,y)+\alpha$};
        % \draw (-3.6,3.4) node {$x\mapsto c(x,y)+\alpha$};
        
    \end{tikzpicture}

\caption{The $c$-transform of $f$. For a fixed $y\in Y$, the dashed line represents a function $x\mapsto c(x,y)+\alpha$ majorizing $f$. The smallest of such functions is $x\mapsto c(x,y)+f^c(y)$, here represented in solid line.  }
\label{fig:c-transform}
\end{figure}
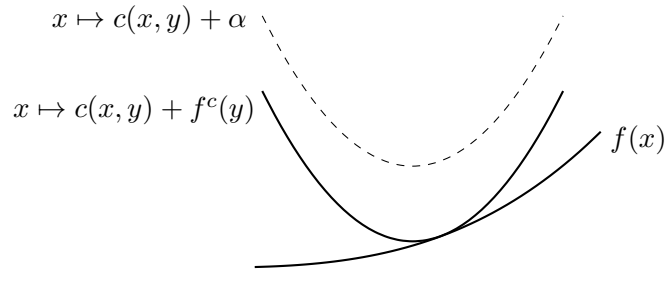

\begin{algorithm}[Gradient descent with a general cost] \label{algo}
    Initialize $x_0\in X$ and alternate the following two steps,
\begin{align}
    y_{n+1} &= \argmin_{y\in Y} c(x_n,y) + f^c(y), \label{algo:step1}\\
    x_{n+1} &= \argmin_{x\in X} c(x,y_{n+1}) + f^c(y_{n+1}). \label{algo:step2}
\end{align}
\end{algorithm}

Let us explain the main ideas behind the two steps~\eqref{algo:step1} and~\eqref{algo:step2}. 
The majorizing functions $x\mapsto c(x,y)+f^c(y)$ form a family indexed by $y\in Y$. The $y_{n+1}$ update \eqref{algo:step1} selects the one element in that family  that is tangent to $f$ at $x_n$. See Fig.~\ref{fig:gd}. Such an element is guaranteed to exist essentially when $f$ is $c$-concave, see Definition~\ref{def:c-concave} below. The update \eqref{algo:step2} is easier to understand and simply finds the next iterate $x_{n+1}$ by minimizing the surrogate function $x\mapsto c(x,y_{n+1})+f^c(y_{n+1})$ found at the previous step. Algorithm~\ref{algo} therefore fits into the ``majorize--minimize'' principle~\cite{lange2016mm}.

\begin{figure}[h]
    \centering
    % base cost h(x-y) with h(x)=1/(x+1.5) + (x+1.5)^3/15
    % min at x=
    \begin{tikzpicture}
        \draw (-3,-0.2) -- (3,-0.2);
        \draw [thick,domain=-2.1:2.5] plot[smooth] (\x, { (\x+3)^3/46 }); 
        \draw [thick,domain=-1.18:2.1] plot[smooth] (\x, { 1/(\x+1.5) + (\x+1.5)^3/15 - 0.03}); 
        \draw [dashed,domain=-0.57:2.5] plot[smooth] (\x, { 1/(\x+1.0) + (\x+1.0)^3/15 + 0.59}); 
        % \draw [dashed,domain=-2.20:1.2] plot[smooth] (\x, { 1/(\x+2.5) + (\x+2.5)^3/15 - 0.68}); 
        \draw [dashed,domain=-2.1:1.2] plot[smooth] (\x, { 1/(\x+2.5) + (\x+2.5)^3/15 - 0.68}); 
        % \fill (0.,0) circle [radius=2pt] node (zero) {};
        % \fill (1,0) circle [radius=2pt] node (one) {};
        %%% x_n %%% 
        \fill (1.0,-0.2) circle [radius=2pt] node[below] (xn) {$x_n$};
        \fill (1.0,1.41167) circle [radius=2pt] node (zn) {};
        \fill (1.0,2.49) circle [radius=2pt] node (ztn) {};
        \fill (1.0,1.65) circle [radius=2pt] node (zttn) {};
        \draw [dashed] (xn) -- (ztn) ;        

        %%% x_{n+1} %%%
        \fill (0,-0.2) circle [radius=2pt] node[below] (xnpo) {$x_{n+1}$};
        \fill (0,0.86167) circle [radius=2pt] node (znpo) {};
        \draw [dashed] (xnpo) -- (znpo) ;        

        \draw (3,3) node {$f(x)$};
        % \draw (-3.8,1.5) node {$x\mapsto c(x,y)+f^c(y)$};
        % \draw (-1.7,3.5) node {$x\mapsto c(x,y_{n+1})+f^c(y_{n+1})$};
        \draw (-3.6,2.6) node {$x\mapsto c(x,y_{n+1})+f^c(y_{n+1})$};
    \end{tikzpicture}

\caption{Iterates of Algorithm~\ref{algo}. The dashed functions represent some surrogates $x\mapsto c(x,y)+f^c(y)$ for various values of $y$. The solid line surrogate is the one for which the value at $x_n$ is minimized, i.e.\ $y=y_{n+1}$.}
\label{fig:gd}
\end{figure}
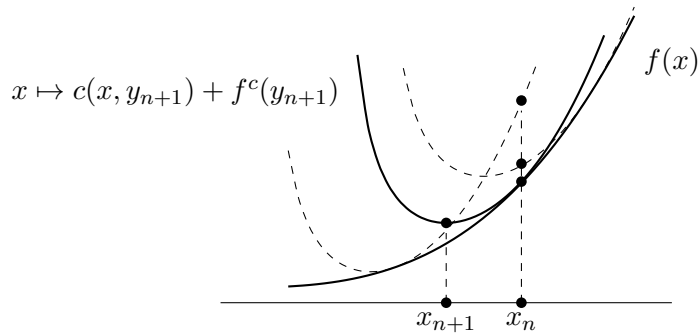

The form \eqref{algo:step1}--\eqref{algo:step2} of our algorithm presents multiple advantages. Firstly it formalizes in a concise manner the idea of \emph{minimizing a surrogate function} often found in optimization. Secondly it takes the form of an alternating minimization. This confers stability to the method and we establish in Section~\ref{sec:am} new convergence rates for alternating minimization that are especially well adapted to Algorithm~\ref{algo}, see Section~\ref{sec:gd:rates}. 
Lastly, \eqref{algo:step1}--\eqref{algo:step2} is a robust formulation in the sense that it doesn't require differentiability or even continuity of $f$ and $c$. In the same vein $X$ and $Y$ could a priori be infinite-dimensional or even general sets, as in Section~\ref{sec:am-rates}. This is thanks to the strong mathematical foundations of optimal transportation, where the $c$-transform and other objects can be used in general nonsmooth contexts. We do not explore these avenues here and choose for simplicity to remain in the differentiable realm.  

On the other hand, a drawback of~\eqref{algo:step1}--\eqref{algo:step2} is that the $c$-transform term $f^c$ is itself obtained through an optimization problem~\eqref{eq:def-c-transform}.
When implementing the algorithm we may not want to or be able to compute $f^c$.
In that case we use a more explicit form:

\begin{manualalg}{2'}[A more explicit version of Algorithm~\ref{algo}] \label{algo:gd-explicit}
    Initialize $x_0\in X$ and alternate the following two steps,
    \begin{align}
        -&\nabla_x c(x_n, y_{n+1})=-\nabla f(x_n), \label{algo:step1_grad}\\
        &\nabla_x c(x_{n+1}, y_{n+1}) =0. \label{algo:step2_grad}
    \end{align}    
\end{manualalg}

We show in \Cref{prop:alg_grad} below that under a certain \emph{smoothness} condition on $f$ known as $c$-concavity, steps \eqref{algo:step1}--\eqref{algo:step2} can be written as \eqref{algo:step1_grad}--\eqref{algo:step2_grad}, which have no $c$-transform term. But before that let us briefly comment on the object $-\nabla_xc$ that shows up in \eqref{algo:step1_grad} and \eqref{algo:step2_grad}. When inverted, it defines the \emph{c-exponential map}, see Definition~\ref{def:c-exp} in the previous section. 
Therefore \eqref{algo:step1_grad} can be written as  
\begin{equation*}
    y_{n+1}=\cexp_{x_n}(-\nabla f(x_n)).
\end{equation*}
In general $X$ and $Y$ are totally distinct spaces. Thus after constructing $y_{n+1}\in Y$ we need a way to map back into $X$ and obtain a new iterate $x_{n+1}$. This is the role of \eqref{algo:step2_grad} which acts as a bridge between $X$ and $Y$.

Let us now define $c$-concavity, a central notion in optimal transport which in the context of this paper can be seen as a smoothness property (here \emph{smoothness} is used in the sense of an upper bound on the Hessian, as in the optimization literature).

\begin{definition}[$c$-concavity] \label{def:c-concave}
    We say that a function $f\colon X\to\R$ is $c$-concave if there exists a function $h\colon Y\to \R$ such that
    \begin{equation}\label{eq:def_c_concav}
        f(x)=\inf_{y\in Y}c(x,y)+h(y),
    \end{equation}
    for all $x\in X$.    
\end{definition}

An elementary but essential observation is that we may replace $h$ by $f^c$ in \eqref{eq:def_c_concav}~\cite[Chapter 5]{Villani_book_2009} (in fact $f^c$ is the smallest of such $h$'s). In terms of our surrogate function $\phi(x,y)=c(x,y)+f^c(y)$, this says that $f$ is $c$-concave if and only if it can be recovered from $\phi$ as 
\[
    f(x)=\inf_{y\in Y} c(x,y)+f^c(y)=\inf_{y\in Y} \phi(x,y).
\]

To state a precise result relating Algorithm~\ref{algo} to Algorithm~\ref{algo:gd-explicit} we introduce the following assumption:

\begin{enumerate}[resume*=ass]    
    \item $c$ and $f$ satisfy: \label{ass:fc-min-min}
    \begin{enumerate}[\normalfont\textbf{(A\arabic{enumi}\alph*)}]
        \item for each $y\in Y$, the function $x\mapsto c(x,y)$ attains its minimum at a unique point in $X$; \label{ass:c-unique-min}
        \item for each $x\in X$, the function $y\mapsto c(x,y)+f^c(y)$ attains its minimum at a unique point in $Y$. \label{ass:c-fc-unique-min}
    \end{enumerate}
\end{enumerate}

Note that \ref{ass:fc-min-min} is a version of assumption~\ref{ass:phi-min-min} introduced for alternating minimization problems. It ensures that Algorithm~\ref{algo} is well-defined.

\begin{prop} \label{prop:alg_grad}
    Suppose that $f$ is $c$-concave, that $f$ and $c$ are differentiable and that \ref{ass:fc-min-min} holds. Then \eqref{algo:step1} can be written as \eqref{algo:step1_grad} and \eqref{algo:step2} can be written as \eqref{algo:step2_grad}. Therefore Algorithms~\ref{algo} and \ref{algo:gd-explicit} are the same.
\end{prop}
\begin{proof}
    Since $f$ is $c$-concave we have $f(x)=\inf_{y\in Y}\phi(x,y)$ with $\phi(x,y)=c(x,y)+f^c(y)$. By the envelope theorem in Lemma~\ref{lemma:envelope} together with \eqref{algo:step1} we obtain $\nabla f(x_n)=\nabla_x\phi(x_n,y_{n+1})=\nabla_xc(x_n,y_{n+1})$, which gives us \eqref{algo:step1_grad}. The $x$-update of the algorithm is more immediate and \eqref{algo:step2} directly gives \eqref{algo:step2_grad}.
\end{proof}

\begin{remark}[Invariance under reparametrization in $y$]\label{rem:reparam}
    Algorithm~\ref{algo} is invariant under reparametrization in $y$, in the following sense. Let $\tilde Y$ be a set and $S\colon Y\to\tilde Y$ a bijection. Then doing the change of variables $\yt=S(y)$ in the cost $c(x,y)$ will adjust the objects $f^c$ and $\phi=c+f^c$ correspondingly but will not change the quantities defined on $X$. Indeed defining $\tilde c(x,\yt)=c(x,S^{-1}(\yt))$ and $\tilde\phi(x,\yt)=\phi(x,S^{-1}(\yt))$ we have 
    \[
        f^{\tilde c}(\yt)=\sup_{x\in X}f(x)-\tilde c(x,\yt)=\sup_{x\in X}f(x)-c(x,y)=f^c(y),
    \] 
    when $\yt=S(y)$. This implies $\tilde\phi(x,\yt)=\phi(x,y)$, and in particular $\inf_{y\in Y}\phi(x,y)=\inf_{\yt\in\tilde Y}\tilde\phi(x,\tilde y)$. In Algorithm~\ref{algo} the $y$-step will therefore produce different iterates $\yt_n=S(y_n)$ but the $x_n$ iterates will be the same. This gives the user great flexibility for choosing the ``dual space'' $Y$ or, in other words, the expression of the cost. See for instance the two formulations of mirror descent in Examples~\ref{ex:mirror-descent} and \ref{ex:mirror-descent_dual}.
\end{remark}

\begin{remark}[Invariance when adding to the cost a function of $y$]
    When the cost $c$ is changed into $\tilde c(x,y)=c(x,y)+a(y)$, we have that $f^{\tilde c}(y)=f^c(y)-a(y)$. Therefore $\tilde c(x,y)+f^{\tilde c}(y)=c(x,y)+f^c(y)$ and the surrogate $\phi(x,y)$ is left unchanged by this operation.     
\end{remark}

We now look at various examples of Algorithm~\ref{algo:gd-explicit} under different choices of cost functions. These examples show that a substantial number of popular gradient descent-like optimization methods can be unified by our framework. Note that \eqref{algo:step1}--\eqref{algo:step2} also says that all these methods have an alternating minimization formulation.

\begin{example}[Gradient descent] \label{ex:gd}
    Let $(X,\norm{\cdot})$ be a Euclidean space, take $Y=X$ and $c(x,y)=\frac{L}{2}\norm{x-y}^2$ for some $L>0$. We have that $-\nabla_xc(x,y)=L(y-x)$, therefore the $y$-update \eqref{algo:step1_grad} can be written as $y_{n+1}-x_n=-\frac{1}{L}\nabla f(x_n)$. As for the $x$-update \eqref{algo:step2_grad} it simply says that $x_{n+1}=y_{n+1}$. Therefore Algorithm~\ref{algo:gd-explicit} takes the form
    \[
        x_{n+1}-x_n=-\frac{1}{L}\nabla f(x_n).
    \]
    This is the classical gradient descent of $f$ with fixed step size $1/L$. This example is developed in Section~\ref{sec:examples:mirror-descent}, where we also show that $c$-concavity corresponds to $L$-smoothness.
\end{example}

\begin{example}[Mirror descent] \label{ex:mirror-descent}
    Let $X$ be a finite-dimensional vector space and let $u\colon X\to\R$ be a strictly convex function. Take $Y=X$ and $c(x,y)=u(x|y)$, where $u(x|y)\coloneqq u(x)-u(y)-\bracket{\nabla u(y),x-y}$ is the \emph{Bregman divergence} of $u$. Then the steps \eqref{algo:step1_grad}--\eqref{algo:step2_grad} can be combined into
    \begin{equation*}
        \nabla u(x_{n+1})-\nabla u(x_n)=-\nabla f(x_n).
    \end{equation*}
    Thus Algorithm~\ref{algo:gd-explicit} is the mirror descent of $f$ with mirror function $u$. This example is treated in more detail in Section~\ref{sec:examples:mirror-descent} with some background on Bregman divergences. We show in particular that $c$-concavity corresponds to relative smoothness.
\end{example}

\begin{example}[Another view on mirror descent]\label{ex:mirror-descent_dual}
    Let $X$ and $u$ be as in Example~\ref{ex:mirror-descent}. Then take $Y=X^*$ to be the dual vector space of $X$ and consider the cost $c(x,y)=u(x)+u^*(y)-\bracket{x,y}$, where $u^*$ is the convex conjugate of $u$. This ``Fenchel--Young gap'' cost is essentially the Bregman divergence of $u$ up to a reparametrization in $y$, and see Remark~\ref{rem:reparam}. To be more precise, we have that $c(x,y)=u(x|\yt)$ if $y=\nabla u(\yt)$. Therefore we expect this cost to give us the mirror descent iteration, which is indeed the case as we now verify. We have $\nabla_xc(x,y)=\nabla u(x)-y$ and therefore Algorithm~\ref{algo:gd-explicit} takes the form 
    \begin{align*}
        y_{n+1}-\nabla u(x_n) &= -\nabla f(x_n),\\
        \nabla u(x_{n+1}) &= y_{n+1}.
    \end{align*}
    Eliminating the $y$ variable we recover mirror descent. A potential advantage of this perspective is that $c(x,y)=u(x)+u^*(y)-\bracket{x,y}$ can be defined for a general class of lower semicontinuous convex functions $u$, even in infinite dimensions, and without worrying about differentiability. Then \eqref{algo:step1}--\eqref{algo:step2} may provide a way to define mirror descent in an infinite-dimensional or nondifferentiable context.
\end{example}

\begin{example}[Natural gradient descent]
    Let $X$ be a finite-dimensional vector space and let $u\colon X\to\R$ be a convex function with a positive-definite Hessian. Take $Y=X$ and $c(x,y)=u(y|x)$ (note the reversed order of $x$ and $y$ compared to \Cref{ex:mirror-descent}). Then $-\nabla_xc(x,y)=\nabla^2u(x)(y-x)$ and Algorithm~\ref{algo:gd-explicit} takes the form
    \begin{equation*}
        x_{n+1}-x_n=-(\nabla^2u(x_n))^{-1}\nabla f(x_n).
    \end{equation*}
    This is the \emph{natural gradient descent} iteration \cite{Amari1998,Martens2020}, with Hessian metric $\nabla^2u$. We treat this example in detail in Section~\ref{sec:examples:natural-gd}. In particular when $u=f$ this is Newton's method, which we study in Section~\ref{sec:newton}.
\end{example}

\begin{example}[A nonlinear gradient descent]
    Let $X=Y$ be a finite-dimensional vector space and consider a translation-invariant cost $c(x,y)=\ell(x-y)$, with $\ell(\cdot)$ nonnegative, strictly convex and differentiable and such that  $\ell(0)=0$. Let $\ell^*$ be its convex conjugate. Then Algorithm~\ref{algo:gd-explicit} takes the form
    \[
        x_{n+1}-x_n = -\nabla \ell^*(\nabla f(x_n)).
    \]
    This is a type of gradient descent with a nonlinear preconditioner $\nabla\ell^*$. It was studied in~\cite{MR4236857,laude2022anisotropic}.
\end{example}

\begin{example}[Riemannian gradient descent]
    Let $(M,\mathsf g)$ be a Riemannian manifold. Take $X=Y=M$ with the cost $c(x,y)=\frac L2d^2(x,y)$, where $L>0$ and $d$ is the Riemannian distance. Then Algorithm~\ref{algo:gd-explicit} takes the form
    \[
        x_{n+1} = \exp_{x_n}(-\frac{1}{L}\nabla f(x_n)),
    \]
    where $\exp$ stands for the Riemannian exponential map (see the discussion following Definition~\ref{def:c-exp}). This is the Riemannian gradient descent iteration, where the next iterate $x_{n+1}$ is obtained by shooting a geodesic at $x_n$ with initial velocity $-\frac{1}{L}\nabla f(x_n)$. We treat this example in more detail in Section~\ref{sec:examples:riemannian}.
\end{example}

\begin{example}[Log-divergence cost] \label{ex:log-div}
    Let $X=Y$ be a convex subset of $\Rd$, fix $\alpha>0$ and let $u\colon X\to\R$. Consider the cost
    \[
        c(x,y)=u(x)-u(y)+\frac{1}{\alpha}\log(1-\alpha\bk{\nabla u(y),x-y}).
    \]
    This function was introduced and studied under the name ``log-divergence'' in \cite{PalWong2016,PalWong2018}. When $x\mapsto e^{-\alpha u(x)}$ is concave over $X$, the cost $c(x,y)$ is well-defined and $c(x,y)\geq 0$. Note that when $\alpha\to 0$ we recover the Bregman divergence $u(x|y)$. With this cost Algorithm~\ref{algo:gd-explicit} takes the form 
    \begin{equation}    \label{eq:algo:log-div}
        \nabla u(x_{n+1}) = \mu_{n+1} \big(\nabla u(x_n) - \nabla f(x_n)\big),
    \end{equation}
    where $\mu_{n+1}\in \R$ is a solution to the scalar equation 
    \begin{equation}    \label{eq:scalar-eq:mu}
        \alpha\bracket{(\nabla u)^{-1}\big(\mu\nabla h(x_n)\big)-x_n,\nabla h(x_n)}\mu-\mu+1=0,
    \end{equation}
    assuming $u$ has invertible gradient.
    Here we write $h= u-f$, and $(\nabla u)^{-1}$ denotes the inverse of the gradient of $u$ as a mapping $X\to\Rd$. 
    
    In the particular case $u(x)=\frac L2\norm{x}^2$, \eqref{eq:scalar-eq:mu} becomes a quadratic equation that can be solved in closed form, and \eqref{eq:algo:log-div} reads $x_{n+1}=\mu_{n+1}(x_n-\frac1L\nabla f(x_n))$. Thus we obtain a scaling of gradient descent which is reminiscent of accelerated gradient methods (Polyak's heavy ball and Nesterov's method).

    Finally we note that \cite{KainthWongRudzicz2022} introduces and studies a ``conformal mirror flow'', which can be seen as a continuous-time version of \eqref{eq:algo:log-div}. They show that up to a reparametrization of time it is equivalent to a mirror flow, i.e.\ a gradient flow with respect to a Hessian metric. This doesn't seem to be the case at the discrete level \eqref{eq:algo:log-div}.   

\end{example}

%%%%%%%%%%%%%%%%%%%%%%%%%%%%%%%%%%%%%%%
%%%%%%%%%  CONVERGENCE RATES  %%%%%%%%%
%%%%%%%%%%%%%%%%%%%%%%%%%%%%%%%%%%%%%%%

\subsection{Convergence rates} \label{sec:gd:rates}

In this section we present non-asymptotic sublinear and linear convergence rates for Algorithm~\ref{algo}. Since Algorithm~\ref{algo} is written as an alternating minimization we can apply the theory developed in Section~\ref{sec:am}. We reformulate the five-point property on the function $\phi(x,y)=c(x,y)+f^c(y)$ in terms of $f$ and $c$. More precisely, we show that it corresponds to two properties of $f$, a smoothness property ($c$-concavity) and a new convexity-type property which we introduce under the name cross-convexity (Definition \ref{def:sc}). Note that this falls in line with the classical theory for gradient descent where sublinear rates require both smoothness and convexity.

To begin with, we recall the \emph{cross-difference} of $c$ defined by \eqref{eq:def_cross_diff} in Section~\ref{sec:background:mtw},
\begin{equation*}
    \delta_c(x',y';x,y) \coloneqq c(x,y')+c(x',y)-c(x,y)-c(x',y').
\end{equation*}

\begin{definition}[cross-convexity] \label{def:sc}
    Suppose that $f$ and $c$ are differentiable. We say that $f$ is $c$-cross-convex if 
    for all $x,\xb\in X$ and any $\yb,\yh\in Y$ verifying $\nabla_xc(\xb,\yb)=0$ and $-\nabla_xc(\xb,\yh)=-\nabla f(\xb)$ we have 
    \begin{equation}\label{sc}
        f(x) \geq f(\xb) + \delta_c(x,\yb;\xb,\yh).
    \end{equation}     
    In addition let $\lambda>0$. We say that $f$ is $\lambda$-strongly $c$-cross-convex if under the same conditions as above we have 
    \begin{equation}\label{lsc}
        f(x) \geq f(\xb) + \delta_c(x,\yb;\xb,\yh) + \lambda(c(x,\yb)-c(\xb,\yb)).
    \end{equation}    
\end{definition}

The next result says that when combined with $c$-concavity, cross-convexity gives the five-point property for the surrogate. This is not surprising given the resemblance of \eqref{sc} with the form \eqref{eq:fpp-F}
of the five-point property.

\begin{prop}\label{prop:conc-conv-fpp}
    Suppose that $f$ and $c$ are differentiable, that \ref{ass:fc-min-min} holds and let $\phi(x,y)=c(x,y)+f^c(y)$. If $f$ is $c$-concave and $c$-cross-convex then $\phi$ satisfies the five-point property \eqref{fpp}.

    Moreover let $\lambda>0$. If $f$ is $c$-concave and $\lambda$-strongly $c$-cross-convex then $\phi$ satisfies the $\lambda$-strong five-point property \eqref{sfpp}. 
\end{prop}

\begin{proof}
    The cases $\lambda=0$ and $\lambda>0$ can be treated at the same time and we therefore take $\lambda\geq 0$. Let $F(x)=\inf_{y\in Y}\phi(x,y)$. Since $f$ is $c$-concave we have that $F(x)=f(x)$. Let us prove \eqref{sfpp} in the form~\eqref{eq:sfpp-F}: for all $x\in X$ and $y_0\in Y$,
    \begin{equation}\label{eq:proof-prop-cc-fp-1}
        F(x) \geq F(x_0) +\delta_\phi(x,y_0;x_0,y_1) + \lambda[\phi(x,y_0)-\phi(x_0,y_0)],
    \end{equation}
    with $x_0=\argmin_{x\in X}c(x,y_0)$ and $y_1=\argmin_{y\in Y} c(x_0,y)+f^c(y)$. 

    By definition of $x_0$ we have $\nabla_xc(x_0,y_0)=0$. By the envelope theorem in \Cref{lemma:envelope} we have that $-\nabla_xc(x_0,y_1)=-\nabla f(x_0)$. We can therefore apply the definition of cross-convexity with $\xb=x_0$, $\yb=y_0$ and $\yh=y_1$. Then \eqref{eq:proof-prop-cc-fp-1} directly follows from \eqref{lsc}, after using that $F=f$, that $\delta_\phi=\delta_c$ by \eqref{eq:prop_cross_diff} and checking that $c(x,y_0)-c(x_0,y_0)=\phi(x,y_0)-\phi(x_0,y_0)$.    
\end{proof}

We are now ready to state our convergence rates for gradient descent with a general cost. 

\begin{theorem}[Convergence rates for Algorithm~\ref{algo}] \label{thm:cv_rates}
    Suppose that $f$ and $c$ are differentiable and satisfy \ref{ass:fc-min-min}. Suppose that $f$ is bounded below and that it attains its infimum $f_*$ at a point $x_*$. Consider the iterates $x_n,y_n$ defined by Algorithm~\ref{algo}. Then the following statements hold.    
    \begin{enumerate}[(i)]
        \item Suppose that $f$ is $c$-concave. Then we have the descent property
    \[
        f(x_{n+1})\le f(x_n)-[c(x_n,y_{n+1})-c(x_{n+1},y_{n+1})] \le f(x_n),
    \]    
   which implies the stopping criterion 
   \[
    \min_{0\le k\le n-1} [c(x_k,y_{k+1})-c(x_{k+1},y_{k+1})]\le \frac{f(x_0)-f_*}{n}.
   \]
   
    \item Suppose in addition that $f$ is $c$-cross-convex. Then 
    for any $x\in X, n\geq 1$,
    \begin{equation}\label{eq:descent_sublinear}
         f(x_n)\le f(x) + \frac{c(x,y_0)-c(x_0,y_0)}{n}.
    \end{equation}
    In particular $f(x_n)-f_* =O(1/n)$. 
    
    \item Suppose in addition that $f$ is $\lambda$-strongly $c$-cross-convex for some $\lambda\in(0,1)$. Then for any $x\in X, n\geq 1$,
    \begin{equation}\label{eq:descent_linear}
         f(x_n)\le f(x) + \frac{\lambda \,(c(x,y_0)-c(x_0,y_0))}{\Lambda^n-1},
    \end{equation}
    where $\Lambda\coloneqq(1-\lambda)^{-1}>1$. In particular $f(x_n)-f_* = O((1-\lambda)^n)$.
    \end{enumerate}
\end{theorem}

In some sense \Cref{thm:cv_rates} is just a corollary of the convergence rates for alternating minimization established in \Cref{thm:am-rates}.
\begin{proof}
    (i): Let $\phi(x,y)=c(x,y)+f^c(y)$. Since $f$ is $c$-concave we have that $f(x)=\inf_{y\in Y}\phi(x,y)$. By \Cref{thm:am-rates}(i) we have 
    \begin{equation}
        \phi(x_{n+1},y_{n+2})\leq \phi(x_{n+1},y_{n+1})\leq \phi(x_n,y_{n+1}).
    \end{equation}
    The $y$-update \eqref{algo:step1} implies $\phi(x_{n+1},y_{n+2})=f(x_{n+1})$ and $\phi(x_n,y_{n+1})=f(x_n)$. We have $\phi(x_{n+1},y_{n+1})=c(x_{n+1},y_{n+1})+f^c(y_{n+1})$ and the $x$-update \eqref{algo:step2} implies $f^c(y_{n+1})=f(x_n)-c(x_n,y_{n+1})$. 

    (ii): If $f$ is both $c$-concave and $c$-cross-convex then by \Cref{prop:conc-conv-fpp} we know that $\phi$ satisfies the five-point property. By \Cref{thm:am-rates}(ii) we have 
    \begin{equation}
        \phi(x_n,y_n) \leq \phi(x,y) + \frac{\phi(x,y_0) - \phi(x_0,y_0)}{n}.
    \end{equation}
    We bound $\phi(x_n,y_n)\geq \phi(x_n,y_{n+1})=f(x_n)$, and take an infimum over $y$ to obtain $f(x)=\inf_y\phi(x,y)$ in the right-hand side. Moreover $\phi(x,y_0) - \phi(x_0,y_0)=c(x,y_0) - c(x_0,y_0)$.
    
    (iii): Works the same way as for (ii).
    
\end{proof}

Let us now examine more closely the conditions required to obtain a convergence rate such as \eqref{eq:descent_sublinear}. We need $f$ to be both $c$-concave and cross-convex. In general $c$-concavity and cross-convexity are nonlocal notions that are difficult to prove. This is the same situation as in Section~\ref{sec:am} where it was highly desirable to have local criteria for the nonlocal five-point property. Here too we can state local criteria for $c$-concavity and cross-convexity. The one for $c$-concavity is known in optimal transport:

\begin{theorem}[Local criterion for $c$-concavity {\cite[Theorem 12.46]{Villani_book_2009}}] 
    \label{thm:differential-criterion-c-concavity}
    Suppose that \ref{ass:cross-curv} holds and that $c$ has nonnegative cross-curvature. Let $f$ be a twice-differentiable function. Suppose that for all $\xb\in X$, there exists $\yh\in Y$ satisfying $-\nabla_xc(\xb,\yh)=-\nabla f(\xb)$ and such that
    \[
        \nabla^2f(\xb) \leq \nabla^2_{xx}c(\xb,\yh).
    \]    
    Then $f$ is $c$-concave. Conversely suppose that $f$ is $c$-concave. Then if $(\xb,\yh)\in X\times Y$ satisfies $-\nabla_xc(\xb,\yh)=-\nabla f(\xb)$, necessarily $\nabla^2f(\xb) \leq \nabla^2_{xx}c(\xb,\yh)$.
\end{theorem}

Since our setting is slightly different from the one in Villani's book, we reprove \Cref{thm:differential-criterion-c-concavity}
in Appendix~\ref{app:sec-gd}. Next in the same spirit as \Cref{thm:sc-fpp-am} we offer a semi-local sufficient condition for cross-convexity.

\begin{theorem}[Sufficient conditions for cross-convexity] \label{thm:sc-fpp+gd}    
    Suppose that \ref{ass:cross-curv} and \ref{ass:c-unique-min} hold and that $c$ has nonnegative cross-curvature. Suppose that $f$ is differentiable.
    \begin{enumerate}[(i)]
        \item If $t\mapsto f(x(t))$ is convex on every $c$-segment $t\mapsto (x(t),\yb)$ satisfying $\nabla_xc(x(0),\yb)=0$, then $f$ is $c$-cross-convex.
        
        \item Let $\lambda>0$. If $t\mapsto f(x(t))-\lambda c(x(t),\yb)$ is convex on the same $c$-segments as for \textup{(i)}, then $f$ is $\lambda$-strongly $c$-cross-convex.
    \end{enumerate}
\end{theorem}

Similar to \Cref{thm:differential-criterion-c-concavity}, there is a converse result that if $f$ is $c$-cross-convex then  whenever $\nabla_xc(\xb,\yb)=0$ and $-\nabla_xc(\xb,\yh)=-\nabla f(\xb)$, necessarily
\begin{equation}\label{eq:necessary-condition-cc}
    \nabla^2f(\xb)\geq \nabla^2_{xx}c(\xb,\yh)-\nabla^2_{xx}c(\xb,\yb).
\end{equation}
This follows from \Cref{thm:sc-fpp+gd}(i) by looking at $\frac{d^2}{dt^2}f(x(t))$ at $t=0$. We do not know whether \eqref{eq:necessary-condition-cc} is sufficient to obtain cross-convexity. 
If $f$ is $\lambda$-strongly cross-convex then for the same configuration of points,
\begin{equation}\label{eq:necessary-condition-scc}
    \nabla^2f(\xb)\geq \nabla^2_{xx}c(\xb,\yh)-(1-\lambda)\nabla^2_{xx}c(\xb,\yb).
\end{equation}

\begin{proof}[Proof of \Cref{thm:sc-fpp+gd}]
    This proof is similar to the proof of \Cref{thm:sc-fpp-am}.
    
    (i):
    Fix $x,\xb\in X$ and let $\yb,\yh\in Y$ be any points satisfying $\nabla_xc(\xb,\yb)=0$ and $-\nabla_xc(\xb,\yh)=-\nabla f(\xb)$. Let $t\mapsto (x(t),\yb)$ be a $c$-segment with endpoints $x(0)=\xb$ and $x(1)=x$. This $c$-segment exists by \ref{ass:biconvex}. By~\ref{ass:c-unique-min} the set $\argmin_{x\in X}c(x,\yb)$ is nonempty and reduces to a singleton. Then the first-order conditions force this argmin to be $\{\xb\}$. Thus this $c$-segment is of the desired form. By convexity of $f$ along this $c$-segment we have that 
    \begin{equation*}
        f(x)-f(\xb)=f(x(1))-f(x(0))\geq \bracket{\nabla f(\xb),\dot x(0)}.
    \end{equation*}
    But $\nabla f(\xb)=\nabla_xc(\xb,\yh)=\nabla_xc(\xb,\yh)-\nabla_xc(\xb,\yb)$. Then \Cref{lemma:delta-c-segments} gives us 
    \begin{equation*}
        \bracket{\nabla_xc(\xb,\yh)-\nabla_xc(\xb,\yb),\dot x(0)} \geq -\delta_c(x,\yh;\xb,\yb),
    \end{equation*}
     and $-\delta_c(x,\yh;\xb,\yb)=\delta_c(x,\yb;\xb,\yh)$ by the properties of the cross-difference. As a consequence, $f(x)-f(\xb)\geq \delta_c(x,\yb;\xb,\yh)$.

     (ii): Similarly to (i), the convexity of $t\mapsto f(x(t))-\lambda c(x(t),\yb)$ implies
     \[
        f(x)-f(\xb) - \lambda(c(x,\yb)-c(\xb,\yb))\geq \bracket{\nabla f(\xb)-\lambda\nabla_xc(\xb,\yb),\dot x(0)}.
    \]
    Since $\nabla_xc(\xb,\yb)=0$ we can finish the argument as in (i).
\end{proof}

We illustrate practical applications of Theorem~\ref{thm:sc-fpp+gd} in Section~\ref{sec:examples} on a variety of examples.

\paragraph{Related work.} In~\cite[Section 3]{Bolte2016majorization} Bolte and Pauwels obtain (sub)linear convergence rates for a class of majorization--minimization procedures, in a ``semialgebraic'' setting. 
They assume that an upperbound $\phi(x,y)$ of $f(x)$ is given, rather than construct it through a $c$-transform. An example of such a surrogate is given by the moving balls method \cite{Auslender2010moving}. Bolte and Pauwels tackle constraint approximation and use the Kurdyka--Łojasiewicz property to establish their convergence rates \cite[Theorem 3.1]{Bolte2016majorization}.

\subsection{Forward--backward splitting}\label{sec:fb}

In this section we look at the case where the objective function is the sum of two functions $f,g:X\rightarrow \R$,
\begin{equation}\label{eq:pb:fwd-bwd}
    \min_{x\in X} F(x)\coloneqq  f(x)+g(x),
\end{equation}
and where we intend to do an explicit gradient step on $f$ and an implicit step on $g$. Let us be more precise. Given \eqref{eq:pb:fwd-bwd}, if we can find a cost $c$ such that $f$ is $c$-concave, we can then replace $f(x)$ by its majorant $c(x,y)+f^c(y)$. On the other hand we don't want to assume that $g$ is $c$-concave, i.e.\ we would like to consider functions $g$ that are not \emph{smooth}. Assuming that $f$ is $c$-concave, \eqref{eq:pb:fwd-bwd} can then be written as 
\begin{equation}\label{eq:def-phi-fb}
    \min_{x\in X,y\in Y} \phi(x,y)\coloneqq c(x,y)+f^c(y)+g(x).
\end{equation}
We consider the following assumptions.
\begin{enumerate}[resume*=ass]    
    \item $f$, $g$ and $c$ satisfy \label{ass:fg-min-min}
    \begin{enumerate}[\normalfont\textbf{(A\arabic{enumi}\alph*)}]
        \item for each $x\in X$, $\inf_{y\in Y}c(x,y)=0$ and the infimum is attained at a point in $Y$; \label{ass:fb-c}
        \item for each $x\in X$, the function $y\mapsto c(x,y)+f^c(y)$ attains its minimum at a unique point in $Y$;    \label{ass:fb-f} 
        \item for each $y\in Y$, the function $x\mapsto c(x,y)+g(x)$ attains its minimum at a unique point in $X$.\label{ass:fb-g}
    \end{enumerate}
\end{enumerate}
The roles of \ref{ass:fb-f} and \ref{ass:fb-g} are clear in view of doing an alternating minimization on $\phi$, they are the direct analogue of \ref{ass:phi-min-min}. \ref{ass:fb-c} is related to the $c$-concavity of the function $x\mapsto 0$. Its role will be more clear in \Cref{prop:fb_fpp} and \Cref{thm:sc-cross-conc}.
The forward--backward version of \emph{gradient descent with a general cost} then takes the following form.

\begin{algorithm}[Forward--Backward with a general cost] \label{algo:fb}
    Initialize $x_0\in X$ and alternate the following two steps,
    \begin{align}
        y_{n+1} &= \argmin_{y\in Y} c(x_n,y) + f^c(y) + g(x_n), \label{algo:fb:step1}\\
        x_{n+1} &= \argmin_{x\in X} c(x,y_{n+1}) + f^c(y_{n+1}) + g(x). \label{algo:fb:step2}
    \end{align}    
\end{algorithm}
This is simply the alternating minimization of $\phi$ defined by \eqref{eq:def-phi-fb}. 
A direct analogue of \Cref{prop:alg_grad}, whose proof we omit, is

\begin{prop}[A more explicit version of Algorithm~\ref{algo:fb}]\label{cor:fw-bck}
    Suppose that \ref{ass:fg-min-min} holds and that $f$, $g$ and $c$ are differentiable. Suppose that $f$ is $c$-concave. Then Algorithm~\ref{algo:fb} can be written as
    \begin{align}
        -\nabla_x c(x_n, y_{n+1})&=-\nabla f(x_n), \label{algo:fb:step1_grad}\\
        -\nabla_x c(x_{n+1}, y_{n+1}) &=\nabla g(x_{n+1}). \label{algo:fb:step2_grad}
    \end{align}
\end{prop}

\Cref{algo:fb} can thus be described as alternating an explicit gradient update on $f$ (the forward step \eqref{algo:fb:step1_grad}) and an implicit ``proximal'' update on $g$ (the backward step \eqref{algo:fb:step2} or \eqref{algo:fb:step2_grad}). 
The case $f=0$ corresponds to pure proximal updates. The case $g=0$ corresponds to pure gradient updates, i.e.\ Algorithm~\ref{algo}. 

\begin{remark}[Alternating minimization, gradient descent, forward--backward] \label{rem:connections-am-fb-gd}
    Now that we have seen Algorithms~\ref{algo:am}, \ref{algo} and \ref{algo:fb} it might be useful to make additional connections between them. Let $\phi(x,y)$, $c(x,y)$, $g(x)$ and $h(y)$ be four functions, and suppose that $\phi(x,y)=c(x,y)+g(x)+h(y)$. Then define $F(x)=\inf_y\phi(x,y)$ and $f(x)=\inf_yc(x,y)+h(y)$. Note that $F(x)=f(x)+g(x)$ and that $f$ is automatically $c$-concave. Let us suppose that the objects at play are differentiable, that minimizers exist and are unique, etc. Then the alternating minimization of $\phi$ can be written as 
    \begin{equation}\label{eq:am-as-ggd}
        \begin{aligned}
            -\nabla_x \phi(x_n, y_{n+1})&=-\nabla F(x_n),\\
            -\nabla_x \phi(x_{n+1}, y_{n+1}) &=0.
        \end{aligned}
    \end{equation}
    Here the envelope theorem is used to formulate the $y$-update using $F$. Therefore Algorithm~\ref{algo:am} can always be written as Algorithm~\ref{algo}, by taking $\phi$ as the cost and $F$ as the objective function. Next, expressing $\phi$ in terms of $c$ we obtain the equivalent form of \eqref{eq:am-as-ggd},
    \begin{equation}\label{eq:am-as-fb}
        \begin{aligned}
            -\nabla_x c(x_n, y_{n+1})&=-\nabla f(x_n),\\
            -\nabla_x c(x_{n+1}, y_{n+1}) &=\nabla g(x_n).
        \end{aligned}
    \end{equation}
    This is exactly \eqref{algo:fb:step1_grad}--\eqref{algo:fb:step2_grad}, the explicit version of Algorithm~\ref{algo:fb}. The reason why we obtain this forward--backward iteration is that changing $h$ into $f^c$ or vice versa leaves the $y$-update unchanged. In other words, in \eqref{algo:fb:step1}, or even in \eqref{algo:step1}, $f^c$ can be replaced by any other function $h$ that satisfies $\inf_{y\in Y} c(x,y)+h(y)=\inf_{y\in Y}c(x,y)+f^c(y)$, for all $x\in X$.
\end{remark}

Let us now focus on convergence rates. In Section~\ref{sec:gd:rates}, when $\phi(x,y)=c(x,y)+f^c(y)$ we introduced cross-convexity as the property for a $c$-concave $f$ corresponding to the five-point property of $\phi$. When instead $\phi(x,y)=c(x,y)+g(x)$ and \ref{ass:fb-c} holds, the five-point property corresponds to a type of cross-convexity with reversed inequalities. We name it cross-concavity and define it below. We suppose that $c$ is differentiable.

\begin{definition}[cross-concavity]\label{def:cross-concavity}
    We say that a differentiable function $f\colon X\to\R$ is $c$-cross-concave if for all $x,\xb\in X$ and any $\yb,\yh\in Y$ verifying $\nabla_xc(\xb,\yb)=0$ and $-\nabla_xc(\xb,\yh)=-\nabla f(\xb)$ we have 
    \[
        f(x)\leq f(\xb)+\delta_c(x,\yb;\xb,\yh).
    \]
    In addition let $\lambda>0$. We say that $f$ is $\lambda$-strongly $c$-cross-concave if under the same conditions as above we have 
    \[
        f(x)\leq f(\xb)+\delta_c(x,\yb;\xb,\yh)-\lambda(c(x,\yb)-c(\xb,\yb)).
    \]
\end{definition}

To obtain convergence rates for forward--backward we will require $-g$ to be $c$-cross-concave. Note that in contrast to the standard Euclidean case, cross-concavity of $-g$ is in general not equivalent to cross-convexity of $g$. This difference can be seen in particular when $c$ has nonnegative cross-curvature, comparing the semi-local criteria of \Cref{thm:sc-fpp+gd} (for cross-convexity) and \Cref{thm:sc-cross-conc} below (for cross-concavity). These two results suppose convexity/concavity of the function on different $c$-segments.

We are now ready to state our convergence rates for the forward--backward iteration, and delay proofs to Appendix~\ref{app:sec-gd}. 

\begin{theorem}[Convergence rates for Algorithm~\ref{algo:fb}]\label{thm:cv_fb}
    Suppose that $f$, $g$ and $c$ are differentiable and satisfy \ref{ass:fg-min-min}. Consider the iterates $x_n,y_n$ defined by Algorithm~\ref{algo:fb}. Define $F$ by \eqref{eq:pb:fwd-bwd} and take $\yb_0\in \argmin_{y\in Y}c(x_0,y)$. Then the following statements hold.

    \begin{enumerate}[(i)]        
        \item Suppose that $f$ is $c$-concave. Then we have the descent property 
        \[
            F(x_{n+1}) \le F(x_n).
        \]
    
        \item Suppose in addition that $f$ is $c$-cross-convex and that $-g$ is $c$-cross-concave. Then for any $x\in X$, $n\geq 1$,
        \[
           F(x_n)\le F(x) + \frac{c(x,\yb_0)}{n}.
        \]
    
        \item Suppose in addition that $f$ is $\lambda$-strongly $c$-cross-convex and that $-g$ is $\mu$-strongly $c$-cross-concave for some $\lambda,\mu\in [0,1)$ with $\lambda+\mu>0$. Then for any $x\in X$, $n\geq 1$,
        \[
            F(x_n)\le F(x) + \frac{(\lambda+\mu)\, c(x,\yb_0)}{\Lambda^n-1},
         \]
         with $\Lambda=\frac{1+\mu}{1-\lambda}$.
    \end{enumerate}    
\end{theorem}

Since cross-convexity of $f$ is related to the five-point property of $\phi(x,y)=c(x,y)+f^c(y)$ and cross-concavity of $-g$ is related to the five-point property of $\phi(x,y)=c(x,y)+g(x)$, it might be natural to ask whether jointly they imply the five-point property of $\phi(x,y)+f^c(y)+g(x)$. That is indeed the case, as stated below. However strong versions of cross-convexity and cross-concavity do not give exactly the strong five-point property, which is partly why \Cref{thm:cv_fb} is not a directly corollary of \Cref{thm:am-rates} and needs its own proof.

\begin{prop}\label{prop:fb_fpp}
    Suppose that $c$, $f$ and $g$ are differentiable and satisfy \ref{ass:fg-min-min}. Suppose that $f$ is $c$-concave and $c$-cross-convex and that $-g$ is $c$-cross-concave. Then $\phi(x,y)=c(x,y)+f^c(y)+g(x)$ satisfies the five-point property \eqref{fpp}.
\end{prop}

\Cref{prop:fb_fpp} is proved in Appendix~\ref{app:sec-gd}. To conclude this section, and in the spirit of \Cref{thm:sc-fpp+gd}, let us now state a semi-local criterion for cross-concavity, when $c$ has nonnegative cross-curvature.

\begin{theorem}[Sufficient conditions for cross-concavity] \label{thm:sc-cross-conc}
    Suppose that \ref{ass:cross-curv} and \ref{ass:c-unique-min}  hold and that $c$ has nonnegative cross-curvature. Suppose that $g$ is differentiable.
    \begin{enumerate}[(i)]
        \item If $t\mapsto g(x(t))$ is convex on every $c$-segment $t\mapsto (x(t),\yh)$ satisfying $-\nabla_xc(x(0),\yh)=\nabla g(x(0))$, then $-g$ is $c$-cross-concave.
        
        \item Let $\lambda>0$. If $t\mapsto g(x(t))-\lambda c(x(t),y)$ is convex on the same $c$-segments as for \textup{(i)}, then $-g$ is $\lambda$-strongly $c$-cross-concave.
    \end{enumerate}
\end{theorem}

\Cref{thm:sc-cross-conc} is proven in Appendix~\ref{app:sec-gd}. We see that the $c$-segments considered in \Cref{thm:sc-cross-conc} are different from those in \Cref{thm:sc-fpp+gd}, which were of the form $(x(t),\yb)$ with $\nabla_xc(x(0),\yb)=0$.

\textbf{Context.} \Cref{algo:fb} is a generalization of the classical forward--backward splitting to general costs, beyond the Euclidean setting. We refer to \cite{Combettes2011} for a review of forward--backward methods and to \cite{Davis2016} for a compendium of rates. Generalizations to more than two functions also exist. In the classical Euclidean setting (i.e.\ for quadratic costs), sublinear rates of convergence of the values were shown in \cite[Theorem 3.1]{Beck2009fista} and linear rates mentioned in \cite[Proposition 2]{Bredies2008}. \cite[Theorems 4 and 5]{Nesterov2012} presents both rates, which we recover through our analysis. Note that for the standard Euclidean forward--backward splitting the focus is in general not on the convergence of the values as done in \Cref{thm:cv_fb}, but on the linear convergence of the iterates as in \cite{Chen1997}, obtained by studying the non-expansivess of the operators appearing in the scheme.

Let us now mention an approach that is related to ours, but focused on a continuous-time setting. In their book~\cite{AmbrosioGigliSavare_book}, Ambrosio, Gigli and Savaré develop a theory to rigorously define a gradient flow $\dot x(t)=-\nabla g(x(t))$ on a general \emph{metric space} $(X,d)$. Their starting point is to define a discrete flow $x_{n+1}\in\argmin_x g(x)+\frac{1}{2\tau}d^2(x_n,x)$, where $\tau>0$ is a time step, and then construct the continuous flow $x(t)\approx x_{n\tau}$ by taking the limit $\tau\to 0$. We see that the discrete flow is a proximal method with a movement limiter $\frac{1}{2\tau}d^2(x_n,x)$ that is based on a general metric distance $d(x,y)$. When this distance is the Euclidean distance they recover the classical proximal iteration (implicit Euler). With our notation, their proximal scheme can be seen as the alternating minimization of the function $\phi(x,y)=c(x,y)+g(x)$ where the cost is $c(x,y)=\frac{1}{2\tau}d^2(x,y)$. Note that the interests of \cite{AmbrosioGigliSavare_book} differ from ours since their focus is on continuous flows in infinite-dimensional spaces in a nonsmooth setting.

% Detailed examples

\section{Detailed examples and applications} \label{sec:examples}

In this section, we develop some of the examples sketched in the previous sections and look at new ones. We study the $c$-concavity and cross-convexity conditions on concrete costs and detail the convergence rates obtained by our theory, recovering some well-known ones and obtaining new ones.

\subsection{Gradient descent and mirror descent} \label{sec:examples:mirror-descent}

In this section we study mirror descent, for which the cost function is a Bregman divergence. We show that $c$-concavity corresponds to \emph{relative smoothness}, introduced by \cite{bauschke2017descent}, and that (strong) cross-convexity corresponds to \emph{relative (strong) convexity}, as defined by \cite{LuFreundNesterov}. We recover the convergence rates given in these two articles.

Let $X$ be a finite-dimensional vector space and $u\colon X\to\R$ be a strictly convex, twice differentiable function with a non-singular Hessian. We take $Y=X$ and consider the Bregman divergence cost 
\begin{equation}\label{eq:ex-bregman-cost}
    c(x,y)=u(x|y),
\end{equation}
defined as follows:

\begin{definition}[Bregman divergence \cite{Bregman1967}] \label{def:Bregman}
    The Bregman divergence of $u$ is the function $u(\cdot|\cdot)\colon X\times X\to\R$ defined by
    \begin{equation*}\label{eq:def_bregman}
        u(x'|x)=u(x')-u(x)-\scalh{\nabla u(x)}{x'-x}.
    \end{equation*}
\end{definition}

Bregman divergences are especially useful in convex analysis since a differentiable function $v$ is convex  (resp.\ strictly convex) if and only if for all $x,x'\in X$, $v(x'|x)\ge 0$ (resp.\ with equality only if $x'=x$), see e.g.\ \cite[Lemma 12]{AubinKorbaLeger}. 
Also, Bregman divergences generalize the square of the Euclidean distance: let $\norm{\cdot}$ be any Euclidean norm on $X$ and define $q(x)=\frac12 \norm{x-a}^2$, where $a$ is any point in $X$. Then $q(x'|x)=\frac12 \norm{x'-x}^2$. 

For the Bregman cost \eqref{eq:ex-bregman-cost} we have $-\nabla_xc(x,y)=\nabla u(y)-\nabla u(x)$. Therefore the $y$-update in Algorithm~\ref{algo:gd-explicit} takes the form $\nabla u(y_{n+1})-\nabla u(x_n)=-\nabla f(x_n)$, while the $x$-update is given by $\nabla u(x_{n+1})=\nabla u(y_{n+1})$. By strict convexity of $u$ we have that $x_{n+1}=y_{n+1}$ and therefore Algorithm~\ref{algo:gd-explicit} is given by 
\begin{equation}\label{eq:mirror_descent_grad}
    \nabla u(x_{n+1})-\nabla u(x_n)=-\nabla f(x_n).
\end{equation}
This is the mirror descent iteration~\cite{NemirovskyYudinBook1983,BeckTeboulle2003}. When $u(x)=\frac{L}{2}\norm{x-a}^2$, where $\norm{\cdot}$ is a Euclidean norm, $L>0$ and $a\in X$, then \eqref{eq:mirror_descent_grad} reads 
\[
    x_{n+1}-x_n=-\frac1L\nabla f(x_n).
\]
This is the standard gradient descent iteration with fixed step size $1/L$. Thus mirror descent is an extension of gradient descent where the Euclidean distance squared $\frac{L}{2}\norm{x'-x}^2$ is replaced by a Bregman divergence $u(x'|x)$.

We now turn our attention to $c$-concavity and cross-convexity. We will show that they correspond to the standard notions of (relative) smoothness and strong convexity, which we now review.
In the following definition we have fixed a Euclidean norm $\norm{\cdot}$ on $X$, $q(x)\coloneqq\frac12 \norm{x-a}^2$, $I_d$ denotes the $d\times d$ identity matrix and inequalities between symmetric matrices are meant in the standard order defined by the convex cone of positive semi-definite matrices.

\begin{definition}[Relative smoothness and convexity]\label{def:rel_conv}
    Let $L>0$, $\lambda > 0$, and consider a twice differentiable function $f\colon X\rightarrow \R$.
    \begin{enumerate}[(i)]
        \item \label{it:L_smooth} 
        We say that $f$ is $L$-smooth if $Lq-f$ is convex. Equivalently, if $\nabla^2 f\le LI_d$, or if $f(x'|x)\le \frac{L}{2}\norm{x'-x}^2=Lq(x'|x)$. 
        
        \item \label{it:L_relsmooth} 
        More generally, $f$ is smooth \emph{relatively to} $u$ if $u-f$ is convex. Equivalently, if $\nabla^2f\leq \nabla^2 u$, or if $f(x'|x)\le u(x'|x)$, i.e.\ 
        \begin{equation}\label{eq:L-rel-smooth}
            f(x')\le f(x) + \ps{\nabla f(x),x'-x} + u(x'|x).
        \end{equation}
        
        \item \label{it:mu_conv} 
        We say that $f$ is $\lambda$-strongly convex if $f-\lambda q$ is convex. Equivalently, if $\nabla^2f\geq \lambda I_d$, or if $f(x'|x)\geq \frac{\lambda}{2}\norm{x'-x}^2=\lambda q(x'|x)$.
        
        \item \label{it:mu_relconv}
        More generally, $f$ is $\lambda$-strongly convex \emph{relatively to} $u$ if $f-\lambda u$ is convex. Equivalently, if $\nabla^2f\geq \lambda\nabla^2u$, or if $f(x'|x)\geq \lambda u(x'|x)$.
    \end{enumerate}
\end{definition}

The various equivalences present in the above definition are classical, see for instance \cite[Appendix A.1]{dAspremont2021}. The relative smoothness property \ref{it:L_relsmooth} suggests using the right-hand side of~\eqref{eq:L-rel-smooth} as a majorizing surrogate. Indeed this leads to the variational form of mirror descent~\cite{BeckTeboulle2003}
\begin{equation}\label{eq:def_mirror-descent}
    x_{n+1}=\argmin_{x\in X} f(x_n)+\bracket{\nabla f(x_n),x-x_n}+u(x|x_n).
\end{equation}
Assuming the minimum to exist and be unique, the first-order optimality conditions applied to \eqref{eq:def_mirror-descent} give the mirror descent iteration \eqref{eq:mirror_descent_grad}. The surrogate in \eqref{eq:def_mirror-descent} is the same one we construct for our gradient descent method, up to a reparametrization in $y$ (and the $x_n$ iterates are unchanged when reparametrizing $y$, see Remark~\ref{rem:reparam}). Let us be more precise.
Define as in Section~\ref{sec:gd} the surrogate
\[
    \phi(x,y)=u(x|y)+f^c(y),
\]
and as in \eqref{eq:def_mirror-descent} the function 
\begin{align}
    \tilde\phi(x,\yt)&=f(\yt)+\bracket{\nabla f(\yt),x-\yt} + u(x|\yt)\label{eq:background:def-s2}\\
    &= f(x) + (u-f)(x|\yt),\label{eq:background:def-s}
    \end{align}
for $x,\yt\in X$. Then we have the following result.

\begin{lemma}
    Suppose that $u-f$ is strictly convex. We then have $\tilde\phi(x,\yt)=\phi(x,y)$ when $\nabla u(y)=\nabla u(\yt)-\nabla f(\yt)$.
\end{lemma} 
\begin{proof}
    Fix $y,\yt$ satisfying $\nabla u(y)=\nabla u(\yt)-\nabla f(\yt)$. Define $r(x)=\phi(x,y)-\tilde\phi(x,\yt)$ and let us show that $r$ is identically zero. A direct computation shows that $\nabla r(x)=0$. Let us thus look at the value of $r$ at $x=\yt$. First we have $f^c(y)=\sup_{x'}f(x')-u(x'|y)$. The first-order conditions say that the optimal $x'$ satisfies $\nabla (u-f)(x')=\nabla u(y)$. But $\nabla u(y)=\nabla (u-f)(\yt)$ and by strict convexity of $u-f$ necessarily $x'=\yt$. We deduce that $f^c(y)=f(\yt)-u(\yt|y)$. Therefore $r(\yt)=u(\yt|y)+f^c(y)-f(\yt)=0$.    
\end{proof}

The difference between the surrogates $\phi$ and $\tilde\phi$ lies in how the family of majorizing functions is parametrized: by its minimizer $y$ for $\phi$ or by its tangent point $\yt$ with $f$ for $\tilde\phi$.
We can also see directly that an alternating minimization of $\tilde\phi$ gives us mirror descent.
If $f$ is smooth relatively to $u$ then $\tilde\phi$ is a function majorizing $f$ (since $(u-f)(x|\yt)\ge 0$) and tangent to $f$ at $x=\yt$ (since $(u-f)(\yt|\yt)= 0$). Suppose that $u-f$ is strictly convex. Then in the alternating minimization of $\tilde\phi$, the $y$-update gives $\yt_{n+1}=x_n$ while the $x$-update is the mirror descent step \eqref{eq:def_mirror-descent}. If we compare with the alternating minimization of $\phi$ as done in Algorithm~\ref{algo}, with iterates $x_n,y_n$, we see that $\tilde y_{n+1}=y_n$. 

\begin{remark}[Mirror descent and Bregman proximal point] \label{rmk:md_Breg_prox}
    In the form \eqref{eq:background:def-s}, mirror descent $x_{n+1}=\argmin_xf(x)+(u-f)(x|x_n)$ can be seen as a Bregman proximal method over $f$ with movement limiter $h=u-f$. This algorithm is for instance presented in \cite[Section 3]{Eckstein1993} (and should not be confused with the ``Bregman mirror prox'' algorithm of \cite{Nemirovski2004}). Convexity of $h$ is equivalent to smoothness of $f$ relatively to $u=f+h$. If, for some $\lambda>0$, $f$ is furthermore $\lambda$-strongly convex w.r.t.\ $h$ (or equivalently to $u$, with another constant $\lambda$), then the analysis of \cite{LuFreundNesterov, AubinKorbaLeger} applies to the Bregman proximal point algorithm with regularizer $h$, by writing it as a mirror descent with regularizer $u$. This leads to a novel linear rate, complementary to the sublinear rate of \cite[Theorem 3.4]{chen1993convergence}. The link between the two algorithms was also stressed in \cite[Section 4.1]{Auslender2006} but through inexact subdifferentials rather than through a change of Bregman potential.

    With the form $\tilde \phi$, we can also see the connection with a key ingredient in the classical proof of the convergence of mirror descent \cite[Lemma 3.1]{LuFreundNesterov} and Bregman proximal points \cite[Lemma 3.2]{chen1993convergence}. Indeed, since $\yt_{n+1}=x_n$ implies $h(x_n | \yt_{n+1})=0$, the ``Bregman three-point inequality'' corresponds to the five point property \eqref{fpp} of $\tilde \phi$,
    \begin{equation*}
        f(x) -f(x_n) \geq  h(x|\yt_{n+1})-h(x|\yt_n)+ h(x_n|\yt_n)-h(x_n|\yt_{n+1}).
    \end{equation*}
\end{remark}

In order to connect $c$-concavity with smoothness, and cross-convexity with standard convexity, we can use the theory developed in Section~\ref{sec:background:mtw} since $c$ has zero cross-curvature (Example~\ref{ex:costs-with-nonnegative-curvature},\ref{ex:costs-with-nonnegative-curvature:bregman}). First let us verify that \ref{ass:cross-curv} holds. We have $-\nabla_{xy}c(x,y)=\nabla^2u(y)$. Since the Hessian of $u$ is non-singular we obtain \ref{ass:non-degenerate}. For biconvexity \ref{ass:biconvex} we first need to compute $c$-segments. Since 
\[
    \nabla_yc(x,y)=-\nabla^2u(y)(x-y),
\]
and since again $\nabla^2u(y)$ is non-singular, we conclude that a curve $(x(t),y)$ is a $c$-segment if 
\begin{equation}\label{eq:bregman-cost-c-segments}
    \frac{d^2}{dt^2}x(t)=0,
\end{equation}
i.e.\ $x(t)$ is a (constant-speed) straight line in $X$. Thus \ref{ass:biconvex} holds.

\begin{prop}[$c$-concavity is relative smoothness]\label{prop:c-conv_relsmooth}
    Let $f\colon X\to\R$ be twice differentiable and suppose that $\nabla u$ is surjective as a map from $X$ to $X^*$. Then $f$ is $c$-concave for $c(x,y)=u(x|y)$ if and only if $f$ is smooth relatively to $u$. 
\end{prop}
\begin{proof}
    Fix $\xb\in X$. Since $-\nabla_xc(x,y)=\nabla u(y)-\nabla u(x)$ and $\nabla u$ is surjective we can find $\yh\in X$ such that $-\nabla_xc(\xb,\yh)=-\nabla f(\xb)$. Then $\nabla^2_{xx}u(\xb,\yh)=\nabla^2u(\xb)$ and by \Cref{thm:differential-criterion-c-concavity} we have shown that $c$-concavity is equivalent to $\nabla^2f\leq \nabla^2u$. 
\end{proof}

In the case of the quadratic cost we obtain 

\begin{corollary}[$c$-concavity is smoothness]
    Take $L>0$ and $c(x,y)=\frac L2 \norm{x-y}^2$. Then $f$ is $c$-concave if and only if it is $L$-smooth.
\end{corollary}

Let us now look at cross-convexity. 

\begin{prop}[cross-convexity is convexity]\label{prop:b-conv_strongconv}
    Let $f\colon X\to\R$ be twice differentiable and $c(x,y) = u(x|y)$. Then $f$ is $c$-cross-convex if and only if $f$ is convex.  More generally, let $\lambda> 0$. Then $f$ is $\lambda$-strongly $c$-cross-convex if and only if $f$ is $\lambda$-strongly convex relatively to $u$. 
\end{prop}
\begin{proof} 
    Classical convexity is convexity on straight lines. Since by~\eqref{eq:bregman-cost-c-segments} $c$-segments are straight lines, \Cref{thm:sc-fpp+gd}(i) directly gives us that a convex function $f$ is $c$-cross-convex. On top of that \Cref{thm:sc-fpp+gd}(ii) implies that if $x\mapsto f(x)-\lambda u(x|y)$ is convex, which is equivalent to $f-\lambda u$ convex, then $f$ is $\lambda$-strongly $c$-cross-convex.
    
    The reverse implications follow from \eqref{eq:necessary-condition-cc} and \eqref{eq:necessary-condition-scc}. 
\end{proof}

The above characterizations of $c$-concavity and cross-convexity use the local criteria presented in Theorems~\ref{thm:differential-criterion-c-concavity} and \ref{thm:sc-fpp+gd}. However since Bregman costs (and a fortiori quadratic costs) have a relatively simple structure it is possible to prove \Cref{prop:c-conv_relsmooth} and \Cref{prop:b-conv_strongconv} directly, straight from the definitions of $c$-concavity and cross-convexity. 

\begin{proof}[Alternative proof of \Cref{prop:b-conv_strongconv}]
    For the cost $u(x|y)$ a computation shows that the cross-difference reads
    \[
        \delta_c(x',y';x,y) = \bracket{\nabla u(y')-\nabla u(y),x'-x}.
    \]
    Let $x,\xb,\yb,\yh\in X$ such that $\nabla_xc(\xb,\yb)=0$ and $-\nabla_xc(\xb,\yh)=-\nabla f(\xb)$, i.e.\ $\xb=\yb$ and $\nabla u(\yh)-\nabla u(\xb)= -\nabla f(\xb)$. Then we obtain 
    \[
        \delta_c(x,\yb;\xb,\yh) = \bracket{\nabla f(\xb), x-\xb}.
    \]
    This directly shows that $f$ is cross-convex if and only if it is convex. Besides $\lambda(c(x,\yb)-c(\xb,\yb))=\lambda u(x|\xb)$ and therefore $f$ is $\lambda$-strongly cross-convex if and only if 
    $f(x)-f(\xb)\geq \bracket{\nabla f(\xb), x-\xb} + \lambda u(x|\xb)$.    
    This is $\lambda$-strong convexity relatively to $u$.
\end{proof}

In view of \Cref{prop:c-conv_relsmooth} and \Cref{prop:b-conv_strongconv} we recover the classical convergence rates for mirror descent and gradient descent: a sublinear $1/n$ rate when $f$ is convex and smooth relatively to $u$ \cite{bauschke2017descent}, and a linear rate if in addition $f$ is $\lambda$-strongly convex relatively to $u$ \cite{LuFreundNesterov}.

\subsection{Natural gradient descent} \label{sec:examples:natural-gd}

In this section we look at natural gradient descent and establish the first global convergence rates for this method, under generic assumptions. We also provide in Remark~\ref{rem:md-ngd} an illuminating perspective on the connection between natural gradient descent and mirror descent first discovered by \cite{Raskutti2015}.

Let $X$ be a finite-dimensional vector space and $u\colon X\to\R$ be a strictly convex function that is three times differentiable and with a non-singular Hessian. Take $Y=X$ and consider the cost
\[
    c(x,y)=u(y|x),
\]
where $u(y|x)$ denotes the Bregman divergence of $u$ (Definition~\ref{def:Bregman}). Note that the order of $x$ and $y$ is reversed compared to the Bregman divergence defining mirror descent in the previous section.

Let us start by deriving the expression of \emph{gradient descent with a general cost}, in its form Algorithm~\ref{algo:gd-explicit}. We compute
\[
    -\nabla_xc(x,y)=\nabla^2u(x)(y-x).
\]
Thus given an objective function $f$, the $y$-update $-\nabla_x c(x_n, y_{n+1})=-\nabla f(x_n)$ takes the form $y_{n+1}=x_n-\nabla^2u(x_n)^{-1}\nabla f(x_n)$. As for the $x$-update $\nabla_xc(x_{n+1},y_{n+1})=0$ it says that $x_{n+1}=y_{n+1}$. The variable $y_{n+1}$ can be eliminated and we obtain the following form for Algorithm~\ref{algo:gd-explicit}:
\begin{equation}\label{eq:sec-examples-ngd-1}
    x_{n+1}-x_n=-\nabla^2u(x_n)^{-1}\nabla f(x_n).
\end{equation}
This iteration is often called natural gradient descent~\cite{Amari1998,Martens2020}, especially in information geometry, when $X$ is a space of probability measures and $u$ is the entropy.

\begin{remark}[Mirror descent and natural gradient descent]\label{rem:md-ngd}
    It has been known since \cite{Raskutti2015} that natural gradient descent can be written as mirror descent. Our framework makes this connection very natural. Indeed we say that \eqref{eq:sec-examples-ngd-1} can be written as the alternating minimization of $\phi(x,y)=u(y|x)+f^c(y)$. A standard property of Bregman divergences tells us that
    \begin{equation}\label{eq:prop-of-Bregman-div}
        u(y|x)=u^*(\nabla u(x)|\nabla u(y)),
    \end{equation}
    where $u^*$ denotes the convex conjugate of $u$. As a consequence, 
    \[
        \phi(x,y)=u^*(\nabla u(x)|\nabla u(y))+f^c(y).
    \]
    We can immediately do the change of variables $\yt=\nabla u(y)$ and adjust $\phi$ and $c$ accordingly as in Remark~\ref{rem:reparam} to obtain 
    \begin{equation}\label{eq:rem-ngd-md}
        \tilde \phi(x,\yt)=u^*(\nabla u(x)|\yt)+f^{\tilde c}(\yt).        
    \end{equation}
    In terms of alternating minimization, $\phi$ and $\tilde\phi$ will produce strictly the same iterates $x_n$. Now doing the change of variables
    $\xt=\nabla u(x)$, we see that \eqref{eq:rem-ngd-md} takes the form $u^*(\xt|\yt)+f^{\tilde c}(\yt)$. By Section~\ref{sec:examples:mirror-descent} this leads to a mirror descent in the $\xt$ variable. More precisely, defining the function $\ft$ by $f(x)=\ft(\nabla u(x))$ we obtain that \eqref{eq:sec-examples-ngd-1} can be written as the mirror descent of $\ft$ with mirror function $u^*$,
    \[
        \nabla u^*(\xt_{n+1})-\nabla u^*(\xt_n)=-\nabla \ft(\xt_n).
    \]    
    This can of course also be checked directly.
\end{remark}

Now let us look at $c$-concavity and cross-convexity. In Example~\ref{ex:costs-with-nonnegative-curvature} we show that the cost $\tilde c(x,y)=u(x|y)$ has zero cross-curvature. Since $\S_c$ is symmetric when exchanging $x$ and $y$, in the sense of \Cref{prop:facts-cc}\ref{prop:facts-cc:symmetry}, the cost $c(x,y)=u(y|x)$ also has zero cross-curvature. We can then use Theorems~\ref{thm:differential-criterion-c-concavity} and \ref{thm:sc-fpp+gd} to obtain simple characterizations of $c$-concavity and cross-convexity. We start by determining what the $c$-segments are. 

\begin{lemma}\label{lemma:ngd-c-segments}
    A path $t\mapsto (x(t),y)$ is a horizontal $c$-segment if and only if
    \begin{equation} \label{eq:ngd:c-segments}
        \frac{d^2}{dt^2}\nabla u(x(t)) = 0.
    \end{equation}
    Note in particular that this does not depend on the point $y$. 
\end{lemma}
\begin{proof}
    We have $\nabla_yc(x,y)=\nabla u(y)-\nabla u(x)$. Therefore a path $t\mapsto (x(t),y)$ is a $c$-segment if and only if $0=\frac{d^2}{dt^2}\big[\nabla u(y)-\nabla u(x(t))\big]=-\frac{d^2}{dt^2}\nabla u(x(t))$. 
\end{proof}

Since $u$ is strictly convex we can see $\nabla u$ as an injective map from the vector space $X$ to its dual space $X^*$. Let $D=(\nabla u)(X)$. \Cref{lemma:ngd-c-segments} tells us that $x(t)$ is a $c$-segment if and only if it is the image by $\nabla u^*$ of a regular segment $z(t)\in D$, where $u^*$ is the convex conjugate of $u$.

\begin{lemma}[$c$-concavity and cross-convexity]\label{lemma:ngd-local-characterizations}
    Let $f\colon X\to\R$ be twice differentiable. Define $\tilde f(z)\coloneqq f(\nabla u^*(z))$ on $D$. Then
    \begin{enumerate}[(i)]
        \item\label{lemma:ngd-local-characterizations:i} $f$ is $c$-concave if and only if for all $x,\xi$,
        \begin{equation}\label{eq:sec-ngd-c-concave}
            \nabla^2f(x)(\xi,\xi) \le \nabla^3u(x)\big(\nabla^2u(x)^{-1}\nabla f(x),\xi,\xi\big) + \nabla^2u(x)(\xi,\xi);
        \end{equation}
    
        \item\label{lemma:ngd-local-characterizations:ii} $f$ is $c$-cross-convex if and only if the function $\tilde f$ is convex. Equivalently, if for all $x,\xi$,   
        \begin{equation}\label{eq:sec-ngd-fpi}
            \nabla^2f(x)(\xi,\xi) \geq \nabla^3u(x)\big(\nabla^2u(x)^{-1}\nabla f(x),\xi,\xi\big).
        \end{equation}

        \item\label{lemma:ngd-local-characterizations:iii}  Let $\lambda>0$. Then $f$ is $\lambda$-strongly $c$-cross-convex if and only if the function $\tilde f-\lambda u^*$ is convex. Equivalently, if for all $x,\xi$,
        \begin{equation}\label{eq:sec-ngd-sfpi}
            \nabla^2f(x)(\xi,\xi) \geq \nabla^3u(x)\big(\nabla^2u(x)^{-1}\nabla f(x),\xi,\xi\big) + \lambda\nabla^2u(x)(\xi,\xi).
        \end{equation}
\end{enumerate}
\end{lemma}

\begin{proof}    
\ref{lemma:ngd-local-characterizations:i}:     
    We use \Cref{thm:differential-criterion-c-concavity}. Let $\xb\in X$. Then $\yh\in Y$ satisfies $-\nabla_xc(\xb,\yh)=-\nabla f(\xb)$ if and only if $\yh=\xb-\nabla^2u(\xb)^{-1}\nabla f(\xb)$. We compute $\nabla_{xx}^2c(x,y)(\xi,\xi)=-\nabla^3u(x)(y-x,\xi,\xi)+\nabla^2u(x)(\xi,\xi)$ and thus 
    \[
        \nabla^2_{xx}c(\xb,\yh)(\xi,\xi) = \nabla^3u(\xb)\big(\nabla^2u(\xb)^{-1}\nabla f(\xb),\xi,\xi\big) + \nabla^2u(\xb)(\xi,\xi).
    \]
    By \Cref{thm:differential-criterion-c-concavity}, if $\nabla^2f(\xb)\leq \nabla^2_{xx}c(\xb,\yh)$ for all $\xb$ and corresponding $\yh$, then $f$ is $c$-concave. The converse implication holds as well by \Cref{thm:differential-criterion-c-concavity}.
    
    \ref{lemma:ngd-local-characterizations:ii}: We use \Cref{thm:sc-fpp+gd}(i). Suppose that $f$ is convex on $c$-segments $x(t)$. Here $c$-segments are given by~\eqref{eq:ngd:c-segments} and do not depend on $y$, so we do not need to be more precise on which $c$-segments we consider. Let $z(t)=\nabla u(x(t))$, i.e.\ $x(t)=\nabla u^*(z(t))$. Then convexity of $f(x(t))$ is equivalent to the convexity of $\tilde f(z(t))$ and, since $z(t)$ are straight lines, to the convexity of $\tilde f$. Under these conditions we obtain that $f$ is $c$-cross-convex. 
    
    For the differential formulation \eqref{eq:sec-ngd-fpi} and the converse statement we can use \eqref{eq:necessary-condition-cc}. If $f$ is $c$-cross-convex then for any $\xb\in X$ and for $\yh=\xb-\nabla^2u(\xb)^{-1}\nabla f(\xb)$ we have that $\nabla^2f(\xb)\geq \nabla^2_{xx}c(\xb,\yh)-\nabla^2_{xx}c(\xb,\yb)=\nabla^3u(\xb)\big(\nabla^2u(\xb)^{-1}\nabla f(\xb),\xi,\xi\big)$. 

    \ref{lemma:ngd-local-characterizations:iii}: Works similarly to (ii). By \Cref{thm:sc-fpp+gd}(ii) strong cross-convexity follow from the convexity of $b(t)\coloneqq f(x(t))-\lambda u(y|x(t))$. By a standard property of Bregman divergences we have $u(y|x(t))=u^*(\nabla u(x(t))|\nabla u(y))$. Therefore $b(t)$ can be written in the $z$ variable as 
    \begin{equation*}
        b(t)=\tilde f(z(t))-\lambda u^*(z(t)|\nabla u(y)).
    \end{equation*}
    We conclude that convexity of $\tilde f-\lambda u^*$ implies strong cross-convexity of $f$. The differential expression \eqref{eq:sec-ngd-sfpi} and the converse statement can be obtained as in (ii).

\end{proof}

Applying our convergence rates established in \Cref{thm:cv_rates} we obtain the following new rates for natural gradient descent:

\begin{theorem}[Convergence rates for natural gradient descent]\label{thm:sec-ngd-cv}    
    Let $f$ be a differentiable function and consider the iterates $x_n$ given by \eqref{eq:sec-examples-ngd-1}. We have the following rates for natural gradient descent.
    \begin{enumerate}[(i)]
        \item Suppose that $f$ satisfies \eqref{eq:sec-ngd-c-concave} and \eqref{eq:sec-ngd-fpi}. Then for all $x\in X$ and $n\geq 1$,
        \[
            f(x_n)\le f(x) + \frac{u(x_0|x)}{n}.
        \]
        \item
        Let $\lambda\in (0,1)$ and suppose that $f$ satisfies \eqref{eq:sec-ngd-c-concave} and \eqref{eq:sec-ngd-sfpi}. Then for all $x\in X$ and $n\geq 1$,
        \[
            f(x_n)\leq f(x) + \frac{\lambda u(x_0|x)}{\Lambda^n-1},
        \]
        where $\Lambda=(1-\lambda)^{-1}$.
    \end{enumerate}
\end{theorem}

\textbf{Related work:} To the best of our knowledge the convergence rates presented in \Cref{thm:sec-ngd-cv} are new. The fact that Bregman divergences can induce a natural gradient descent was in some sense already present in \cite{Raskutti2015}, see also Remark~\ref{rem:md-ngd}. Sublinear convergence rates in the context of the Fisher matrix are mentioned in \cite[Section 14, eq.(18)]{Martens2020} which refers to a previous result, \cite[Theorem 2]{Hazan2007}, for the online Newton step over $\alpha$-exp-concave functions with bounded gradients. Linear rates in continuous time for a symmetric positive definite metric $M$ are discussed in \cite[Definition 3, Theorem 1]{Wensing2020} as being equivalent to $f$ being strongly geodesically convex, i.e.\ $\nabla^2 f \ge \alpha M$ for some $\alpha>0$, where $\nabla^2f$ now denotes the Riemannian Hessian. Note that natural gradient descent applies to general Riemannian metrics and not only Hessians of potentials $u$. Existence of such Hessians is shown to exist for commuting and regular parametrization of the submanifold $X$ \cite[Theorem 4.9]{li2022implicit}.

\subsection{Newton's method} \label{sec:newton}

In this section we consider the standard Newton's method~\eqref{eq:sec-newton-algo}, without step sizes, line search procedures, or further approximations. We establish new \emph{global} convergence rates, valid for any initial point $x_0$. 

Take $X$ to be a $d$-dimensional vector space and let $f\colon X\to\R$ be a three-times differentiable convex function with positive definite Hessian. Let $Y=X$ and consider the cost 
\[
    c(x,y)=f(y|x).
\]
This is a particular case of natural gradient descent (see Section~\ref{sec:examples:natural-gd}) where $u(x)=f(x)$. 
Then gradient descent with general cost reads
\begin{equation}\label{eq:sec-newton-algo}
        x_{n+1}-x_n=-\nabla^2f(x_n)^{-1}\nabla f(x_n).
\end{equation}
This is \emph{Newton's method} \cite[Chapter 5]{Nesterov2018}. Let us state our convergence rates for natural gradient descent in the case where $u=f$. We first combine the $c$-concavity and (strong) cross-convexity assumptions \eqref{eq:sec-ngd-c-concave}--\eqref{eq:sec-ngd-sfpi} as follows. Let $0 \le \lambda<1$ and consider the property: for all $x,\xi$,
\begin{equation} \label{eq:ineq-newton}
    0 \leq \nabla^3f(x)\big((\nabla^2f)^{-1}(x)\nabla f(x),\xi,\xi\big) \leq (1-\lambda)\nabla^2f(x)(\xi,\xi).
\end{equation}
Suppose in addition that $f$ is bounded below and that it attains its infimum $f_*$ at a point $x_*\in X$. 
Then \Cref{thm:sec-ngd-cv} specializes into

\begin{theorem}[Global rates for Newton's method]\label{thm:sec-Newton-cv}    
    Suppose that $f$ satisfies \eqref{eq:ineq-newton} with $\lambda=0$. Then for any initial point $x_0\in X$ and any $n\geq 1$,
    \[
        f(x_n)-f_*\leq \frac{f(x_0)-f_*}{n}.
    \]
    If in addition $f$ satisfies \eqref{eq:ineq-newton} with $\lambda>0$, then for any initial point $x_0\in X$ and $n\geq 1$,
    \[
        f(x_n)-f_*\leq\frac{\lambda [f(x_0)-f_*]}{\Lambda^n-1},
    \]
    with $\Lambda=(1-\lambda)^{-1}>1$.     
\end{theorem}

The strength of \Cref{thm:sec-Newton-cv} is that it gives convergence of the values $f(x_n)$ even when $x_0$ and $x_*$ are far apart. On the other hand it doesn't at all capture \emph{quadratic convergence}, which is known to hold under certain conditions on $f$ when $x_n$ gets close enough to $x_*$. 

An important observation on the assumption \eqref{eq:ineq-newton} is that it is \emph{affine-invariant}. This means that it makes sense when $X$ is an affine space, i.e.\ a vector space whose origin has been forgotten. In a more practical manner, $f$ satisfies \eqref{eq:ineq-newton} if and only if the function $x\mapsto f(Ax+b)$ satisfies \eqref{eq:ineq-newton} (with the same constant $\lambda$), for any invertible matrix $A$ and vector $b$. Indeed \eqref{eq:ineq-newton} corresponds to cross-convexity and $c$-concavity properties. These two properties are geometric in nature since they only depend on the cost $c$. Since here the cost is given by a Bregman divergence of $f$, which is an affine-invariant quantity, we automatically see that \eqref{eq:ineq-newton} is also affine-invariant.

\begin{remark}[Comparison with self-concordance]
    Property \eqref{eq:ineq-newton} looks similar to \emph{self-concordance}. Following \cite[Chapter 5]{Nesterov2018}, $f$ is said to be self-concordant if there exists a constant $M\ge 0$ such that\begin{equation}\label{eq:self-concordant}
        |\nabla^3 f(x)(\xi,\xi,\xi)| \le 2 M \big(\nabla^2 f(x)(\xi,\xi)\big)^{\nicefrac{3}{2}},\quad \forall x,\xi \in X.
\end{equation}
However properties \eqref{eq:ineq-newton} and \eqref{eq:self-concordant} are different. For instance, when $X= \R$, $f(x)=e^x$ satisfies \eqref{eq:ineq-newton} for $\lambda=0$ but it is not self-concordant. On the other hand, on $X=(0,\infty)$, $f(x)=-\log(x)$ is self-concordant but does not satisfy \eqref{eq:ineq-newton}. Note also that \eqref{eq:ineq-newton} is invariant by the scaling $f\mapsto \alpha f$ while \eqref{eq:self-concordant} is not (the constant $M$ needs to be changed).

Concerning the local quadratic convergence of Newton's method, both \eqref{eq:ineq-newton} and \eqref{eq:self-concordant} imply, when setting $\xi_x=\nabla^2f(x)^{-1}\nabla f(x)$, that for all $x$ such that $2 M(\nabla^2f(x)^{-1}(\nabla f(x),\nabla f(x)))^{\frac{1}{2}}\le 1$, we have
\begin{equation}
  \nabla^3f(x)\big(\xi_x,\xi_x,\xi_x\big) \leq \nabla^2f(x)^{-1}(\nabla f(x),\nabla f(x)),\, \forall x \in X.\label{eq:sec-newton-decrement}
\end{equation}
 The term on the right-hand side is the Newton decrement which is a crucial quantity for the standard analysis of the convergence of Newton's method. Actually \eqref{eq:sec-newton-decrement} is all one needs to prove the result of \cite[pp.\ 347-356, Theorems 5.1.13 and 5.2.2,1.]{Nesterov2018} on the quadratic convergence of Newton's method in a neighborhood of $x_*$.
\end{remark}

\textbf{Related work:} To the best of our knowledge, the only result on global convergence of the standard Newton's method was given by Ortega and Rheinboldt in~\cite[Section 13.3.4]{Ortega2000}, under excessively restrictive monotonicity assumptions. Otherwise most approaches resort to a globalization strategy, where the convergence analysis is separated into two parts, quadratic in some local region but slow elsewhere, though with global polynomial time. Nevertheless, global convergence is known for some approximate Newton schemes. First come the relaxed or damped Newton's method, where there is in front of the Hessian term a step size $\gamma<1$. In \cite{karimireddy2019global}, $\gamma$ is chosen based on an assumption of ``stable Hessians'' (which is weaker than having both bounded Hessians and self-concordance) and prove a linear convergence in this case. For a generic stepsize $\gamma<1$, a sublinear bound on the norm of the gradient norm is shown  in \cite[Corollary 5.1]{Yuan2022} under assumptions strictly weaker than the monotonicity of \cite{Ortega2000}. Many authors resort to further approximations than just taking another constant stepsize. These include cubic regularization \cite{Nesterov2006} and its variants \cite{misch2023regularized}, both obtaining $O(\frac{1}{n^2})$ rates on the values. We refer to these articles and to \cite[Section 1.2, d)]{Yuan2022} for more references on alternative schemes and on the known non-convergence cases of Newton's method. Since many of these approximation schemes are defined through upper-bounds $\phi(x,y)$ on $f(x)$, our framework also applies to these. For instance, we recover the classical Levenberg--Marquardt approximation by considering the cost $c(x,y)=f(y|x)+\frac{\eps}{2}\|x-y\|^2$.

\subsection{Riemannian gradient descent} \label{sec:examples:riemannian}

Let $(M,\mathsf{g})$ be a geodesically convex Riemannian manifold. Take $X=Y=M$ and consider the cost $c(x,y)=\frac{L}{2}d^2(x,y)$, where $d$ is the Riemannian distance and $L>0$. We use $\mathsf{g}$ to denote the Riemmanian metric to avoid a clash of notation with the functions $g(x)$ used in Section~\ref{sec:fb}. 

On $M$ away from the cut locus, the relation 
\[
    \xi = -\nabla_xc(x,y),
\]
defines a tangent vector $\xi\in T_xM$ and is equivalently written as
\[
    y = \exp_x(\xi/L).
\] 
We refer for instance to the appendix of \cite[Chapter 12]{Villani_book_2009}. Here $\exp$ is the (Riemannian) exponential map, which is defined as follows: $\exp_x(\xi)=\gamma(1)$, where $\gamma\colon [0,1]\to M$ is the (constant-speed) geodesic starting at $x\in M$ with initial velocity $\xi\in T_xM$. As a consequence Algorithm~\ref{algo:gd-explicit} can be written as 
\begin{equation}\label{eq:riemannian-gradient-descent}
    x_{n+1}=\exp_{x_n}\Big(-\frac{1}{L}\nabla f(x_n)\Big).
\end{equation}
This is the Riemannian gradient descent method \cite{UdristeBook1994,daCruzNeto_deLima_Oliveira1998}, a natural generalization of gradient descent to Riemannian manifolds.

Let us look at the form taken by cross-convexity and $c$-concavity in the Riemannian setting. When $c(x,y)=\frac{L}{2}d^2(x,y)$, if $x,\xb,\yb,\yh$ are four points in $M$ satisfying $\nabla_xc(\xb,\yb)=0$ and $-\nabla_xc(\xb,\yh)=-\nabla f(\xb)$, then $\xb=\yb$ and we have 
\begin{equation}\label{eq:cross-diff-riem}
    \delta_c(x,\yb;\xb,\yh)=\frac{L}{2}\Big(d^2(x,\yh)-d^2(x,\xb)-d^2(\xb,\yh)\Big).
\end{equation}
Thus a function $f$ is $c$-cross-convex if
\begin{equation} \label{eq:riem-csc}
    f(x)\geq f(\xb) + \frac{L}{2}\Big(d^2(x,\yh)-d^2(x,\xb)-d^2(\xb,\yh)\Big),
\end{equation}
for all $x,\xb,\yh$ as above. 
On the other hand, on a Riemannian manifold there is already a natural notion of convexity: \emph{geodesic convexity}. A function $f$ is geodesically convex if $t\mapsto f(\gamma(t))$ is convex for any constant-speed geodesic $\gamma$. This is equivalent to 
\begin{equation}\label{eq:riem:geodesic-conv}
    \nabla^2f\geq 0,
\end{equation}
where $\nabla^2f$ denotes the Riemannian Hessian of $f$, and in particular we see that \eqref{eq:riem:geodesic-conv} is a local characterization. The next result relates cross-convexity, cross-concavity and geodesic convexity and is proven in Appendix~\ref{app:sec-ex}.

\begin{prop}\label{prop:compare-cc-gc}
    Let $c(x,y)=\frac{L}{2}d^2(x,y)$. Suppose that $(M,\mathsf{g})$ has nonnegative sectional curvature. Then
    \begin{enumerate}[(i)]
        \item $f$ geodesically convex $\implies$ $f$ $c$-cross-convex.
        \item $-g$ $c$-cross-concave $\implies$ $g$ geodesically convex.
    \end{enumerate}
    Suppose that $(M,\mathsf{g})$ has nonpositive sectional curvature. Then
    \begin{enumerate}[(i),resume]
        \item $f$ $c$-cross-convex $\implies$ $f$ geodesically convex.
        \item $g$ geodesically convex $\implies$  $-g$ $c$-cross-concave.
    \end{enumerate}
\end{prop}

We now turn our attention to $c$-concavity. To relate $c$-concavity to notions of smoothness that are used for \eqref{eq:riemannian-gradient-descent} in the literature, we consider the following properties:

\begin{enumerate}[label=\normalfont(\Alph*)]
    \item $f$ is $c$-concave; \label{riem:c-concave}
    \item $f$ has $L$-Lipschitz gradients; \label{riem:Lip-gradient}
    \item $\nabla^2f \leq L \mathsf{g}$; \label{riem:bound-Hessian}
    \item $f(x)\leq f(\xb)+\bracket{\nabla f(\xb),\xi}+\frac{L}{2} d^2(x,\xb)$, where $x=\exp_\xb(\xi)$. \label{riem:f-bound-above}
\end{enumerate}

Here \ref{riem:Lip-gradient} means the following~\cite{daCruzNeto_deLima_Oliveira1998}: let $\gamma$ be a geodesic in $M$ and let $\theta_{0\to 1}$ denote the parallel transport from $T_{\gamma(0)}M$ to $T_{\gamma(1)}M$ along $\gamma$. Then we say that $f$ has $L$-Lipschitz gradient if 
\[
    \abs{\nabla f(\gamma(1)) - \theta_{0\to 1}\nabla f(\gamma(0))} \leq L d(\gamma_1,\gamma_0).
\]
We then have the following result, proven in Appendix~\ref{app:sec-ex}.

\begin{prop}\label{prop:c-concave-equiv}
    The following statements hold.
    \begin{enumerate}[(i)]
        \item \ref{riem:bound-Hessian}$\iff$\ref{riem:f-bound-above}
        \item Suppose that $(M,\mathsf{g})$ has nonnegative curvature. Then \ref{riem:c-concave}$\implies$\ref{riem:bound-Hessian}.       
        \item Suppose that $(M,\mathsf{g})$ has nonpositive curvature. Then \ref{riem:bound-Hessian}$\implies$\ref{riem:c-concave}.
        \item \ref{riem:Lip-gradient}$\implies$\ref{riem:bound-Hessian}
    \end{enumerate}
\end{prop}

We note that \ref{riem:bound-Hessian} is a local property thus desirable in practice. It can be stronger or weaker than $c$-concavity, as shown in (ii)-(iii). Our convergence rates read as follows.

\begin{prop}[Convergence rates for Riemannian gradient descent]
    Suppose that $f$ and $c$ are differentiable and satisfy \ref{ass:fc-min-min}. Suppose that $f$ is bounded below and that it attains its infimum $f_*$ at a point $x_*$. Consider the iterates $x_n$ defined by \eqref{eq:riemannian-gradient-descent}. Then the following statements hold.
    \begin{enumerate}[(i)]
        \item Suppose that $f$ is $c$-concave. We then have the descent property
    \[
        f(x_{n+1})\le f(x_n)-\frac{1}{2L}\abs{\nabla f(x_n)}^2,
    \]    
   which implies the stopping criterion 
   \[
    \min_{0\le k\le n-1} \frac{1}{2L} \abs{\nabla f(x_k)}^2\leq \frac{f(x_0)-f_*}{n}.
   \]
   
    \item Suppose in addition that $f$ is $c$-cross-convex. Then 
    for any $x\in M, n\geq 1$,
    \begin{equation} \label{eq:riem:sublinear-rate}
         f(x_n)\le f(x) + \frac{L\,d^2(x,x_0)}{2n}.
    \end{equation}
    
    \item Suppose in addition that $f$ is $\lambda$-strongly $c$-cross-convex for some $\lambda\in(0,1)$. Then for any $x\in X, n\geq 1$,
    \begin{equation} \label{eq:riem:linear-rate}
         f(x_n)\le f(x) + \frac{\lambda L \,d^2(x,x_0)}{2(\Lambda^n-1)},
    \end{equation}
    where $\Lambda\coloneqq(1-\lambda)^{-1}>1$.
    \end{enumerate}
\end{prop}

\paragraph{Related work.} There exists a vast literature on optimization on manifolds, and on adapting Euclidean techniques to work in the Riemannian setting. We refer to the textbooks \cite{UdristeBook1994,MR2364186} and the recent \cite{MR4533407} for an introduction to the subject. The Riemannian gradient descent method is natural in this context and early works include \cite{MR362899,MR663521}. 

A notable work is \cite{daCruzNeto_deLima_Oliveira1998}, which introduces the manifold version of $L$-Lipschitz gradients, i.e.\ property~\ref{riem:Lip-gradient}, and champions the Toponogov comparison theorem to use non-Euclidean trigonometry. Moreover, although it is not stated in this way, \cite{daCruzNeto_deLima_Oliveira1998} contains a sublinear $1/n$ convergence rate: under geodesic convexity, $L$-Lipschitz gradient and nonnegative curvature their theorems 5.1 and 5.2 combine into \eqref{eq:riem:sublinear-rate}. \cite{MR3634803} obtains the same rate and also studies the subgradient method. The linear rate \eqref{eq:riem:linear-rate} appears even earlier in \cite[Chapter 7]{UdristeBook1994}. \cite{ZhangSraManifold2016} studies rates of Riemannian gradient descent in a nonpositive curvature setting. In particular when curvature is bounded below by $-k^2<0$, and in the convex and $L$-Lipschitz gradient case they obtain a $1/n$ sublinear rate.

\subsection{POCS (Projection Onto Convex Sets)} \label{sec:pocs}

Let $(H,\norm{\cdot})$ be a Euclidean space and let $B,C$ be two closed convex subsets of $H$. The POCS algorithm \cite{Bauschke_PLC_book} searches for the intersection of $B$ and $C$ by successive projections onto $B$ and $C$: given $x_n\in B$, compute
\begin{equation} \label{eq:pocs}
    \begin{aligned}
        y_{n+1} &= \argmin_{y\in C} \norm{x_n-y},\\
        x_{n+1} &= \argmin_{x\in B} \norm{x-y_{n+1}}.
    \end{aligned}    
\end{equation}
Thus POCS is an alternating minimization method. There are at least two ways to formulate \eqref{eq:pocs} as Algorithm~\ref{algo:am}. (i) Take $X=Y=H$, with the cost $c(x,y)=\frac12\norm{x-y}^2$ and the indicator functions $g=\iota_B$ and $h=\iota_C$. Then define $\phi(x,y)=c(x,y)+g(x)+h(y)$. Note that $g$ and $h$ take the value $\infty$ and therefore don't fit strictly speaking into the theory developed in Section~\ref{sec:am}, but let us ignore this for now. (ii) Take $X=B$, $Y=C$ and consider the function $\phi(x,y)=\frac12\norm{x-y}^2$. We follow approach (ii) which provides a better opportunity to discuss some assumptions and current limitations of our theory. 

Since the quadratic cost has zero cross-curvature, by \Cref{thm:sc-fpp-am} the five-point property can be obtained from local considerations. First let us take a closer look at \ref{ass:cross-curv}. Non-degeneracy \ref{ass:non-degenerate} holds only if $B$ and $C$ have the ``same dimension'', for instance if they both have non-empty interior, or if they are both $k$-dimensional affine subspaces of $H$. This restriction is not so natural in the context of \eqref{eq:pocs}. It is a result of cross-curvature not being well-understood in the ``unequal dimensions'' setting in optimal transport, a more difficult problem for regularity than when $X$ and $Y$ have the same dimension~\cite{ChiapporiMcCannPass2017,McCannPass_unequal}. Note that this dimension restriction is absent in formulation (i) in the previous paragraph.

In Section~\ref{sec:examples:mirror-descent} we showed that for the quadratic cost, $c$-segments are just regular segments. Assumption \ref{ass:biconvex} is then equivalent to the convexity of $B$ and $C$. In view of applying \Cref{thm:sc-fpp-am} let us now focus on the marginal function $F(x)=\inf_{y\in Y}\phi(x,y)$. Here it takes the form 
\[
    F(x)=\min_{y\in C}\frac12\norm{x-y}^2=\frac12d_C^2(x).
\]
Since $C$ is a convex set, $F$ is a convex function~\cite{Bauschke_PLC_book}, and therefore convex on segments, which are here the $c$-segments. 
By \Cref{thm:sc-fpp-am} we deduce that $\phi$ has the five-point property \eqref{fpp}. We may now apply \Cref{thm:am-rates} and obtain the following convergence rate.

\begin{prop}[Convergence rate for POCS]
    For any $x\in B$ and $n\geq 1$,
    \[
        d_C^2(x_n) - d_C(x)^2 \le \frac{\norm{x-x_0}^2}{n},
    \]
    In particular, if $B\cap C\neq\emptyset$ then 
    \begin{equation}\label{eq:rate-pocs}
        d_C^2(x_n)\le \frac{\norm{x-x_0}^2}{n}.
    \end{equation}
    
\end{prop}

Furthermore $x_n$ converges in norm to a point in $B\cap C$, if it is nonempty, as a consequence of \cite[Theorem 3.3, iv)]{Bauschke1993}.

We conclude this section by pointing out that POCS can be formulated as a projected gradient descent,
\begin{equation}\label{eq:pocs-as-gd}
    x_{n+1}=P_B(x_n-\nabla F(x_n)),
\end{equation}
where $P_B$ denotes the projection onto the convex set $B$. The form \eqref{eq:pocs-as-gd} can be understood by Remark~\ref{rem:connections-am-fb-gd} in Section~\ref{sec:fb}. Let us adopt the perspective (i) mentioned at the beginning of this example with $X=Y=H$, $c(x,y)=\frac12\norm{x-y}^2$ and $\phi(x,y)=c(x,y)+\iota_B(x)+\iota_C(y)$. Since the cost is quadratic, the alternating minimization of $\phi$ yields a standard Euclidean forward--backward iteration alternating an explicit gradient step on $F$, $y_{n+1}-x_n=-\nabla F(x_n)$ and a proximal step $x_{n+1}=\argmin_{x\in B}\norm{x-y_{n+1}}^2$, which is here just a projection onto $B$.

\paragraph{Related work.} The sublinear rate \eqref{eq:rate-pocs} is hard to find in the literature.
Linear rates of convergence for POCS occur for (boundedly) linear regular sets as per \cite[Theorems 3.12, 6.3]{Bauschke1993}, a condition satisfied for instance if the interior of $\{x-y \,|\, x\in B,\, y\in C\}$ is nonempty (see also \cite[Section 4]{Deutsch2008} where it is shown that linear regularity is necessary and sufficient for linear rates).

\subsection{Bregman alternating  minimization} \label{sec:alt-bregman-prox}

Let $X=Y$ be a Euclidean space and $u,g,h\colon X\to\R$ three convex functions. Consider the minimization problem
\begin{equation}\label{eq:alt-bregman-minmin}
    \min_{x,y\in X} u(x|y) + g(x) + h(y).
\end{equation}
Here $u(x|y)$ denotes the Bregman divergence of $u$ (see Definition~\ref{def:Bregman}). Let $\phi(x,y)=u(x|y)+g(x)+h(y)$ and consider the alternating minimization of $\phi$,
\begin{align}
        y_{n+1}&=\argmin_{y\in Y}u(x_n|y)+h(y),\label{eq:alternating-bregman-prox-1}    \\
        x_{n+1}&= \argmin_{x\in X} u(x|y_{n+1})+g(x).  \label{eq:alternating-bregman-prox-2}      
\end{align}
Algorithm~\eqref{eq:alternating-bregman-prox-1}--\eqref{eq:alternating-bregman-prox-2} is a natural generalization of POCS in the previous example, where we had $u(x|y)=\frac12\norm{x-y}^2$ and $g,h$ indicator functions of convex sets. It also prepares the subsequent two examples, Sinkhorn and EM, for which $u(x|y)$ is the Kullback--Leibler divergence.

In this section we make two observations on~\eqref{eq:alternating-bregman-prox-1}--\eqref{eq:alternating-bregman-prox-2}. The first one is that it can be formulated as a ``forward--backward mirror descent'' and that convexity of the marginal function $F(x)=\inf_y\phi(x,y)$ is sufficient to obtain convergence rates. The second observation is that \eqref{eq:alternating-bregman-prox-1}--\eqref{eq:alternating-bregman-prox-2} can alternatively be formulated as a ``forward--backward natural gradient descent''. These two facts are consequences of Remark~\ref{rem:connections-am-fb-gd}.

\paragraph{Mirror descent.} Define 
\begin{equation}\label{eq:alt-bregman-md}
    f(x)=\inf_{y\in Y} u(x|y)+h(y). 
\end{equation}
Then by the envelope theorem, the $y$-update \eqref{eq:alternating-bregman-prox-1} can be written as an explicit mirror descent step
\begin{equation}\label{eq:alternating-bregman-prox-3}
    \nabla u(y_{n+1})-\nabla u(x_n)=-\nabla f(x_n).
\end{equation}
As for the $x$-update \eqref{eq:alternating-bregman-prox-2} it is a type of Bregman proximal step. We now claim that convexity of $f+g$ is sufficient to obtain a sublinear convergence rates.

\begin{prop}\label{prop:fwd-bwd-mirror-descent}
    Suppose that $f$ and $g$ are differentiable and that $u$ is strictly convex, twice-differentiable with a non-singular Hessian. Consider the iteration \eqref{eq:alternating-bregman-prox-1}--\eqref{eq:alternating-bregman-prox-2}, or equivalently \eqref{eq:alternating-bregman-prox-3}--\eqref{eq:alternating-bregman-prox-2}.
    \begin{enumerate}[(i)]    
        \item Suppose that $f+g$ is convex. Then 
        for any $x\in X, n\geq 1$,
        \begin{equation*}
            f(x_n)+g(x_n)\le f(x)+g(x) + \frac{u(x|y_0)-u(x_0|y_0)}{n}.
        \end{equation*}
                
        \item Suppose that $f+g$ is $\lambda$-strongly convex for some $\lambda\in(0,1)$. Then for any $x\in X, n\geq 1$,
        \begin{equation*}
            f(x_n)+g(x_n)\le f(x)+g(x) + \frac{\lambda \,(u(x|y_0)-u(x_0|y_0))}{\Lambda^n-1},
        \end{equation*}
        where $\Lambda\coloneqq(1-\lambda)^{-1}>1$. 
    \end{enumerate}
\end{prop}

The particularity of \Cref{prop:fwd-bwd-mirror-descent} is twofold. First $f$ is $c$-concave by construction, since it is defined by $f(x)=\inf_{y\in Y} u(x|y)+h(y)$. Here this means that $f$ is smooth relatively to $u$, see \Cref{prop:c-conv_relsmooth}. But more interestingly only convexity of $f+g$ is required, as opposed to the convexity of both $f$ and $g$ (which of course implies convexity of the sum). Let us explain where this comes from, and why this is not the case for general forward--backward schemes studied in \Cref{thm:cv_fb} which required $f$ to be cross-convex and $-g$ to be cross-concave.

By \Cref{thm:sc-fpp-am}, $\phi$ has the five-point property as soon as $f(x)+g(x)=\inf\phi(x,\cdot)$ is convex on $c$-segments $(x(t),y)$ satisfying $\nabla_x\phi(x(0),y)=0$, i.e.\ $\nabla_xc(x(0),y)=-\nabla g(x(0))$. In general this is a property that depends on $f$ and $g$ in a complicated manner. But for $c(x,y)=u(x|y)$, $c$-segments $(x(t),y)$ are given by regular segments, and in particular they do not depend on $y$. As a consequence convexity of $f+g$ ensures that the five-point property holds. By a similar reasoning, strong convexity of $f+g$ implies the strong five-point property.

\paragraph{Natural gradient descent.} Thanks to Section~\ref{sec:examples:natural-gd} we know that the cost $u(y|x)$ leads to natural gradient descent. Therefore by exchanging the roles of $x$ and $y$ in \eqref{eq:alt-bregman-minmin} and by Remark~\ref{rem:connections-am-fb-gd} we know that \eqref{eq:alternating-bregman-prox-1}--\eqref{eq:alternating-bregman-prox-2} can be formulated as a forward--backward type of natural gradient descent. Let us thus define $\ft(y)=\inf_{x\in X} u(x|y)+g(x)$ an objective function on $Y$ so that \eqref{eq:alt-bregman-minmin} can also be written as $\min_{y\in Y} \ft(y)+h(y)$. Let us change the order of \eqref{eq:alternating-bregman-prox-1} and \eqref{eq:alternating-bregman-prox-2} and do the $x$ update first (changing the index $n+1\to n$), since now it is $x$ that is the extra variable. Using the envelope theorem we write the $x$-update \eqref{eq:alternating-bregman-prox-2} as a natural gradient step 
\begin{equation}\label{eq:alt-bregman-ngd}
    x_n=y_n-\nabla^2u(y_n)^{-1}\nabla \ft(y_n).
\end{equation}
The $y$-update $y_{n+1}=\argmin_{y\in Y}u(x_n|y)+h(y)$ is then a Bregman proximal step.

\paragraph{Related work.} For $f$, $g$ the indicator of convex sets, \cite{Bregman1967} first introduced problem \eqref{eq:alt-bregman-minmin} to find an element in the intersection of the sets, as a generalization of the alternating projection of the POCS algorithm, presented in \Cref{sec:pocs}. Problem \eqref{eq:alt-bregman-minmin} was then considered in \cite{BausckeCombettesNoll2006} where convexity of $f$, $g$ and \emph{joint convexity} of $u(x|y)$ were assumed to show convergence of these proximal iterates. Joint convexity of $u(x|y)$ is a restrictive assumption since besides the Kullback--Leibler (KL) divergence, very few Bregman divergences have this property. Yet it is a powerful property since it implies for instance the convexity of $f$ in \eqref{eq:alt-bregman-md}. The reformulations of \eqref{eq:alt-bregman-minmin} as a mirror descent or as a natural gradient descent appear to be new. The fact that convergence rates can be obtained assuming only convexity of the sum $f+g$ in \Cref{prop:fwd-bwd-mirror-descent} also seems to be new.

\subsection{Sinkhorn}\label{sec:sinkhorn}

The Sinkhorn algorithm, also known as the iterative proportional fitting procedure (IPFP) and the RAS method, is used to solve entropic optimal transport and matrix scaling problems, see \cite{idel} and \cite[Chapter 4]{PeyreCuturiBook}. In this section we present a new formulation of Sinkhorn as an alternating minimization on the primal problem, in the form \eqref{sinkhorn:step1}--\eqref{sinkhorn:step2}. Using the theory developed in Section~\ref{sec:am} we recover a sublinear rate on the marginals first obtained in~\cite{Leger2020}.

Let us start by defining the entropic optimal transport problem. Let $(\X,\mu)$ and $(\Y,\nu)$ be two probability spaces and set 
\[
    C=\Pi(\mu,*),\quad D=\Pi(*,\nu).
\]
Here $\Pi(\mu,*)$ and $\Pi(*,\nu)$ denote the set of couplings over $\X\times \Y$ (i.e.\ joint laws) having first marginal $\mu$ and second marginal $\nu$, respectively. Let us also define $\Pi(\mu,\nu)=\Pi(\mu,*)\cap\Pi(*,\nu)$. Given $\eps>0$ and a $\mu\otimes\nu$-measurable function $b(x,y)$, the \emph{entropic optimal transport problem} is
\begin{equation}\label{eq:eot-problem}
    \min_{\pi \in \Pi(\mu,\nu)}\KL(\pi|e^{-b/\eps}\mu\otimes\nu).
\end{equation}    
Here $\KL(\cdot|\cdot)$ denotes the Kullback--Leibler divergence, defined by $\KL(\pi|\bar \pi)=\int \log\left(\nicefrac{d\pi}{d\bar \pi}\right)d\pi$
when $\pi$ is absolutely continuous with respect to $\bar \pi$ (denoted by $\pi \ll \bar \pi$), and $+\infty$ otherwise.

The Sinkhorn algorithm searches for the solution to \eqref{eq:eot-problem} by initializing $\pi_0(dx,dy)=e^{-b(x,y)/\eps}\mu(dx)\nu(dy)$ and by alternating ``Bregman projections'' onto $\Pi(\mu,*)$ and $\Pi(*,\nu)$, 
\begin{align}
    \gamma_{n+1} &=\argmin_{\gamma\in \Pi(\mu,*)} \KL(\gamma|\pi_n),\label{sinkhorn:1}\\
    \pi_{n+1} &= \argmin_{\pi\in \Pi(*,\nu)} \KL(\pi|\gamma_{n+1}).\label{sinkhorn:2}
\end{align}
When $\X$ and $\Y$ are finite sets this corresponds to alternatively rescaling the columns and rows of the matrix $\pi_0$. Moreover Sinkhorn can also be formulated as an alternating minimization on the dual problem to \eqref{eq:eot-problem}. A recent series of work, including our own \cite{mishchenko2019sinkhorn,Leger2020, MenschPeyre_online_Sinkhorn,AubinKorbaLeger} have studied various ways to write Sinkhorn as a mirror descent, either in the primal or the dual space. Here instead, we show that Sinkhorn can be written an alternating minimization of the $\KL$ divergence, directly in primal variables.

\begin{prop}[Sinkhorn as a primal alternating minimization]
    Initialize $\pi_0(dx,dy)=e^{-b(x,y)/\eps}\mu(dx)\nu(dy)$ and iterate
    \begin{align}
        \gamma_{n+1} &=\argmin_{\gamma\in \Pi(\mu,*)} \KL(\pi_n|\gamma),\label{sinkhorn:step1}\\
        \pi_{n+1} &= \argmin_{\pi\in \Pi(*,\nu)} \KL(\pi|\gamma_{n+1}).\label{sinkhorn:step2}
    \end{align}
    Then $\pi_n$ and $\gamma_n$ are the iterates of Sinkhorn.
\end{prop}
\begin{proof}
    Since \eqref{sinkhorn:2} and \eqref{sinkhorn:step2} are the same, we just need to show that \eqref{sinkhorn:step1} is equivalent to \eqref{sinkhorn:1}. When $q\in \calP(\X\times \Y)$ is a joint probability measure, denote by $p_\X q$ its $\X$-marginal and define $K_q(x,dy)=\nicefrac{q(dx,dy)}{ p_{\X}q(dx)}$. We then have $q= p_\X q \otimes K_{q}$ and for $q,\bar q\in \calP(\X\times \Y)$ with $ q \ll \bar q$, it is not hard to check the disintegration formula     
    \begin{equation}\label{eq:disintegration}
        \KL( q | \bar q)=\KL( p_\X  q| p_\X \bar q)+\int_\X \KL(K_{ q}| K_{\bar q}) dp_\X q.
    \end{equation}
    In particular if the $\X$-marginals of $q$ and $\bar q$ are fixed, then $\KL( q | \bar q)$ is minimal when $K_q=K_{\bar q}$. 
    In both minimization problems \eqref{sinkhorn:1} and \eqref{sinkhorn:step1} the $\X$-marginal of $\gamma$ is contrained to be $\mu$. Therefore the solution to both problems is the same, given by $K_{\gamma_{n+1}}=K_{\pi_n}$ i.e.\ 
    \begin{equation}
        \gamma_{n+1}(dx,dy)=\pi_n(dx,dy) \,\frac{\mu(dx)}{p_\X\pi_n(dx)}.
    \end{equation}
\end{proof}

Although \eqref{sinkhorn:1}--\eqref{sinkhorn:2} and \eqref{sinkhorn:step1}--\eqref{sinkhorn:step2} look very similar, note that only \eqref{sinkhorn:step1}--\eqref{sinkhorn:step2} is an alternating minimization. We can then apply the theory developed in Section~\ref{sec:am}.

To connect with our notation we define $\calX=\calY=\calP(\X\times\Y)$ to be the space of joint probability measures, $c(\pi,\gamma)=\KL(\pi|\gamma)$, $g(\pi)=\iota_{\Pi(*,\nu)}$ and $h=\iota_{\Pi(\mu,*)}$, and $\Phi(\pi,\gamma)=c(\pi,\gamma)+g(\pi)+h(\gamma)$. Here $g$ and $h$ take the value $\infty$ and therefore don't fit strictly speaking into the theory developed in Section~\ref{sec:am}, but we ignore this small trouble which won't affect us as long as we stay in spaces where $g$ and $h$ are finite.

There are several ways to show that $\Phi$ satisfies the five-point property. One is to show it directly, which was already done by Csiszár and Tusnády in \cite[Section 3]{CsiszarTusnady1984}. Alternatively since $\KL$ is a Bregman divergence, we can use the results of Sections~\ref{sec:am-sufficient-fpp} and \ref{sec:examples:mirror-descent}. Set
\[
    F(\pi)=\inf_{\gamma\in\Pi(\mu,*)} \Phi(\pi,\gamma)=\KL(p_\X\pi|\mu).
\]
The function $\Phi(\pi,\gamma)=\KL(\pi|\gamma)$ is known to be \emph{jointly convex} in $(\pi,\gamma)$. Since the constraint $\Pi(\mu,*)$ is convex, we deduce that $F$ is convex, and by \Cref{thm:sc-fpp-am} that $\Phi$ satisfies \eqref{fpp}. By \Cref{thm:am-rates} the following sublinear rate holds: for any $\pi\in \Pi(\mu,\nu)$, we have
\begin{equation}\label{eq:rate-sinkhorn}
    \KL(p_\X\pi_n|\mu)\le \frac{\KL(\pi|\gamma_0)}{n}.
\end{equation}
This is the same rate obtained in~\cite{Leger2020,AubinKorbaLeger}. 

Let us make two final observations. By Remark~\ref{rem:connections-am-fb-gd} and \eqref{sinkhorn:step1}--\eqref{sinkhorn:step2} (see also \eqref{eq:alt-bregman-md}--\eqref{eq:alternating-bregman-prox-3}) we automatically know that Sinkhorn can be written as a mirror descent over $\pi\in \Pi(*,\nu)$, with objective function $F$ and Bregman divergence $\KL$. This was the formulation put forth in \cite{AubinKorbaLeger}. Since $F$ is smooth by construction and convex, \eqref{eq:rate-sinkhorn} directly follows.

Alternatively we can focus on the $\gamma$ variable and use \eqref{eq:alt-bregman-ngd}. Then Sinkhorn can be written as a natural gradient descent over $\Pi(\mu,*)$ with objective function $\tilde F(\gamma)=\inf_{\pi\in\Pi(*,\nu)}\KL(\pi|\gamma)=\KL(\nu|p_\Y\gamma)$ and potential function $H(\gamma)=\KL(\gamma|R)$ where $R$ is any reference measure.

\subsection{Expectation--Maximization}\label{sec:EM}

The Expectation--Maximization (EM) algorithm is an alternating minimization method used to estimate the parameters of a statistical model \cite{Neal1998,lange2016mm}. Let $\X$ be a set of observed data, $\Z$ be a set of latent (i.e.\ hidden) data and let $\{p_\theta\in\calP(\X \times \Z) : \theta\in\Theta\}$ be a \emph{statistical model}, where $\Theta$ is a set of parameters. Here $\calP(\X \times \Z)$ denotes the space of probability measures over $\X\times\Z$, in other words the set of joint laws or couplings. Then, having observed $\mu\in\calP(\X)$ (for instance an empirical measure) we want to find a parameter $\theta\in\Theta$ that maximizes the \emph{likelihood}. This is equivalent to the minimization problem~\cite{Neal1998}
\begin{equation}\label{eq:em-min}
    \min_{\theta\in \Theta} \KL(\mu|p_{\X} p_\theta),
\end{equation}
where given a coupling $p(dx,dz)\in\calP(\X\times\Z)$ we denote its first marginal by $p_\X p(dx)\coloneqq \int_\Z p(dx,dz)$. The Kullback--Leibler divergence $\KL$ is defined in Section~\ref{sec:sinkhorn}.

Let $F(\theta) = \KL(\mu|p_{\X} p_\theta)$ be the objective function in \eqref{eq:em-min}. Because we suppose that we only really have access to $p_\theta(dx,dz)$ and that $p_\X p_\theta$ is not tractable, we cannot work directly with $F(\theta)$ (and do a gradient descent or mirror descent on $F$, say). Instead the idea of EM is to introduce a surrogate for $F$. Let $\Pi(\mu,*)$ denote the couplings $\pi(dx,dz)$ with first marginal $\mu$. Then the \emph{data processing inequality} says that for any $\pi\in\Pi(\mu,*)$,
\begin{equation}\label{eq:data-processing}
    \KL(\mu|p_\X p) \leq \KL(\pi|p),
\end{equation}
and that equality is attained when $\pi=\frac{\mu(dx)}{p_{\X}p(dx)}p(dx,dz)$, see for example \eqref{eq:disintegration}. If we define the surrogate $\Phi(\theta,\pi)=\KL(\pi|p_\theta)$ for any $\theta\in\Theta$ and $\pi\in\Pi(\mu,*)$, then \eqref{eq:data-processing} implies $F(\theta) \leq  \Phi(\theta,\pi)$ for all $\pi\in\Pi(\mu,*)$. With the equality case in mind we obtain
\[
    F(\theta) =\inf_{\pi\in\Pi(\mu,*)} \Phi(\theta,\pi).
\]
The EM algorithm is then the alternating minimization of $\Phi$,

\begin{align}
    \pi_{n+1} &= \argmin_{\pi\in\Pi(\mu,*)} \KL(\pi|p_{\theta_n}), \label{eq:E-step}\\
    \theta_{n+1} &= \argmin_{\theta\in \Theta} \KL(\pi_{n+1}|p_{\theta}). \label{eq:M-step}
\end{align}
Usually \eqref{eq:E-step} is called the E-step and \eqref{eq:M-step} is called the M-step. Let us now derive, under certain conditions, three formulations of EM: as a natural gradient descent in $p$, as a mirror descent in $\theta$ and as a mirror descent in $\pi$. The first one is new, the second one recovers~\cite{Kunstner2021Homeomorphic} and the last one is taken from~\cite{AubinKorbaLeger}. All these formulations can be seen as consequences of Remark~\ref{rem:connections-am-fb-gd}, and see also \eqref{eq:alternating-bregman-prox-3} and \eqref{eq:alt-bregman-ngd}. For some of these formulations it will be convenient to introduce a relative entropy $H(\cdot)=\KL(\cdot|R)$ for some reference measure $R\in\calP(\X\times\Z)$, so that the Kullback--Leibler divergence can be written as the Bregman divergence of $H$.

\subsubsection{EM as projected natural gradient descent}

Let us adapt the notation to this particular case. 
Set $\calX=\calY=\calP(\X\times\Z)$ and $B=\Pi(\mu,*)$ and $C=\calP_\Theta\coloneqq \{p_\theta\in \calP(\X\times\Z) : \theta\in\Theta\}$. Then consider the cost $c(p,\pi)=\KL(\pi|p)$ together with the indicator functions $g(p)=\iota_C(p)$ and $h(\pi)=\iota_B(\pi)$. We then define $\Phi(\pi,p)=\KL(\pi|p)+g(p)+h(\pi)$, as well as 
\begin{equation*}
    f(p)=\inf_\pi c(p,\pi)+h(\pi)=\inf_{\pi\in\Pi(\mu,*)} \KL(\pi|p)=\KL(\mu|p_\X p). 
\end{equation*}
Consider the EM iterates $\pi_n,p_n$, where we write $p_n\coloneqq p_{\theta_n}$.
Following Remark~\ref{rem:connections-am-fb-gd}, by the envelope theorem we have $\nabla f(p_n)=-\nabla^2H(p_n)(\pi_{n+1}-p_n)$. Thus the $\pi_{n+1}$ update can be written as a natural gradient step
\[
    \pi_{n+1}-p_n = -\nabla^2H(p_n)^{-1}\nabla f(p_n).
\] 
The $p_{n+1}$ update then projects $\pi_{n+1}$ back onto the parameter constraint 
\begin{equation*}
    p_{n+1}=\argmin_{p\in \calP_\Theta} \KL(\pi_{n+1}|p). 
\end{equation*}

\subsubsection{Mirror descent in $\theta$ for exponential families}

In this section we revisit results of \cite{Kunstner2021Homeomorphic} which shows that in the context of exponentials families, EM can be written as a mirror descent over $\theta$. Suppose that $\Theta$ is a convex subset of $\Rd$, that we are given $s\colon\X\times\Z\to\Rd$ and that
\[
    p_\theta(dx,dz) = e^{\bracket{s(x,z),\theta} - A(\theta)}R(dx,dz).
\]
Here $R$ is a reference measure on $\X\times\Z$ and $A(\theta)=\log(\int e^{\bracket{s,\theta}}dR)$ is the normalization coefficient that ensures that $p_\theta$ has total mass $1$. By Hölder's inequality $A$ is a convex function. Set $\calX=\Theta$, $\calY=\Pi(\mu,*)$, $\Phi(\theta,\pi) = \KL(\pi|p_\theta)$ and 
\begin{equation*}
    F(\theta)=\inf_{\pi\in\Pi(\mu,*)} \Phi(\theta,\pi)=\KL(\mu|p_\X p_\theta).
\end{equation*}
Given the form of $p_\theta$, a computation shows that 
\[
    \Phi(\theta,\pi) = A(\theta) - \bracket{\theta,\int_{\X\times\Z} s\,d\pi} + \KL(\pi|R).
\]
Here we have the sum of the convex function $A(\theta)$ and a linear term in $\theta$, thus we recognize the start of a Bregman divergence $A(\theta|*)+\dots$. This already tells us we can expect a mirror descent. Instead of computing this Bregman divergence precisely we follow Remark~\ref{rem:connections-am-fb-gd} and use the envelope theorem to write the E-step $\pi_{n+1}=\argmin_{\pi\in \Pi(\mu,*)}\Phi(\theta_n,\pi)$ as 
\[
    \nabla F(\theta_n)=\nabla A(\theta_n) - \int s\,d\pi_{n+1}.
\]
As for the M-step $\theta_{n+1}=\argmin_{\theta\in\Theta}\Phi(\theta,\pi_{n+1})$ it takes the form $\nabla A(\theta_{n+1})-\int s\,d\pi_{n+1}=0$. Combining the E- and M-step we obtain the mirror descent
\[
    \nabla A(\theta_{n+1})-\nabla A(\theta_n)=-\nabla F(\theta_n).
\]
Note that by construction $F$ is $c$-concave thus smooth relatively to $\KL$. If the function $s(x,z)$ makes $F$ convex (resp. strongly convex) we would then have sublinear convergence rates (resp. linear convergence rates).

\subsubsection{Mirror descent in $\pi$ and Latent EM}

In this section we revisit results of \cite{AubinKorbaLeger} where we show that EM can always be written as a mirror descent in $\pi$ and where we show that when $\theta$ is a certain non-parametric distribution over the latent space, the mirror descent has a convex objective function, from which convergence rates ensue.

Here $\pi$ will be the ``primal variable'' so we define 
$\calX=\calY=\calP(\X\times\Z)$, $c(\pi,p)=\KL(\pi|p)$, $g(\pi)=\iota_B(\pi)$ with $B=\Pi(\mu,*)$ and $h(p)=\iota_C(p)$ with $C=\calP_\Theta\coloneqq \{p_\theta\in \calP(\X\times\Z) : \theta\in\Theta\}$. We then define $\Phi(\pi,p)=\KL(\pi|p)+g(\pi)+h(p)$ and 
\begin{equation}\label{eq:latent-em-f}
    f(\pi)=\inf_{p} \KL(\pi|p)+h(p)=\inf_{p\in\calP_\Theta} \KL(\pi|p).
\end{equation}
Following Remark~\ref{rem:connections-am-fb-gd}, we have by the envelope theorem that $\nabla f(\pi_n)=\nabla H(\pi_n)-\nabla H(p_n)$, with $p_n\coloneqq p_{\theta_n}$. Thus the M-step \eqref{eq:M-step} can be written as a mirror descent update
\[   
    \nabla H(p_n)-\nabla H(\pi_n)=-\nabla f(\pi_n).
\] 
The E-step is then the projection $\pi_{n+1}=\argmin_{\pi\in\Pi(\mu,*)} \KL(\pi|p_n)$ and EM is formulated as a projected mirror descent method. Note that since $\Pi(\mu,*)$ is a nice affine subspace of $\calP(\X\times\Z)$ we could alternatively take $\calX=\Pi(\mu,*)$, work directly in $\Pi(\mu,*)$ and use the intrinsic gradient $\nabla_\calX$. This gives a (non projected) mirror descent within $\Pi(\mu,*)$. This was the approach in~\cite{AubinKorbaLeger}.

Let us now look more closely at \eqref{eq:latent-em-f}. First $f$ is $c$-concave by construction which by \Cref{prop:c-conv_relsmooth} says that $f$ is smooth relatively to $\KL$. Additionally $c(\pi,p)=\KL(\pi|p)$ is known to be \emph{jointly convex} over $(\pi,p)$. As a consequence, $f$ is convex as soon as $\calP_\Theta$ is a convex subset of $\calP(\X\times\Z)$. Let us now study a simple instance where that is the case.

Take $\Theta\subset\calP(\Z)$ to be a family of probability measures over the latent space and assume that it is a convex subset of $\calP(\Z)$. Fix a Markov kernel $K(dx|z)$ from $\Z$ to $\X$ and consider the model
\begin{equation}\label{eq:latent-em}
    p_\theta(dx,dz) = K(dx|z)\theta(dz).
\end{equation}
We called this setting Latent EM in \cite{AubinKorbaLeger} and refer to that work for background and more detail on this problem. There we also connected Latent EM to the Richardson--Lucy algorithm for image deconvolution~\cite{Richardson1972,Lucy1974}. In \eqref{eq:latent-em} $\Theta$ is convex and $p_\theta$ is the image of $\theta$ by a linear operator, therefore $\calP_\Theta$ is convex. This implies convexity of $f$ and since $g$ is also convex
we directly recover a sublinear convergence rate as in~\cite{AubinKorbaLeger}.

\section*{Acknowledgements}
The authors would like to thank Robert M. Gower and Antonin Chambolle for helpful comments and suggestions.

\appendix

\section{Appendix}

\subsection{Proofs of Section~\ref{sec:am}}

\begin{lemma}[Envelope theorem] \label{lemma:envelope}
    Let $X,Y$ be two open subsets of $\Rd$ and let $\phi$ be a differentiable function on $X\times Y$. Fix $\xb\in X$. Let $F(x)=\inf_{y\in Y}\phi(x,y)$ and suppose that $F$ is differentiable at $\xb$. If the minimization problem $\inf_{y\in Y}\phi(\xb,y)$ has a minimum at $\yb$, then 
    \[
      \nabla F(\xb)=\nabla_x\phi(\xb,\yb).
    \]
  \end{lemma}
  \begin{proof}
      By definition of $F$ we have for any $x\in X$ that $F(x)\leq \phi(x,\yb)$ and we have equality at $x=\xb$. Therefore $\xb$ is a global minimizer of $G(x)\coloneqq \phi(x,\yb)-F(x)$ and since $G$ is differentiable at $\xb$ its gradient vanishes.
\end{proof}

\begin{lemma} \label{lemma:costs-with-nonnegative-curvature}
    The costs in \Cref{ex:costs-with-nonnegative-curvature} have nonnegative curvature.
\end{lemma}
\begin{proof}    
    \begin{enumerate}[1.]
        \item Let $c(x,y)=\norm{x-y}^2$. Then the mixed partial derivatives $c_{i\jb k}$ and $c_{i\jb\lb}$ both vanish. Therefore $\S_c=0$.
        
        \item Let $c(x,y)=\norm{A(x)-B(y)}^2$. Then $c(x,y)=\norm{x'-y'}^2$ in the new coordinates $x'=A(x)$, $y'=B(y)$. By \Cref{prop:facts-cc}\ref{prop:facts-cc:invariant} cross-curvature is invariant under a change of coordinates. Therefore like for the quadratic cost $\S_c=0$. 
        
        \item We have $c(x,y)=u(x)-u(y)-\bracket{\nabla u(y),x-y}$. By \Cref{prop:facts-cc}\ref{prop:facts-cc:gh} $\S_c$ is unaffected by adding functions of $x$ or functions of $y$. Therefore we may as well suppose that $c(x,y)=-\bracket{\nabla u(y),x}$. But now by a change of variables $y'=\nabla u(y)$ this cost is equivalent to the bilinear cost $-\bracket{y',x}$, itself equivalent to the quadratic cost. Thus $\S_c=0$. 
        
        \item This is a particular case of the next point.
        
        \item In the coordinates $\xt_i=e^{x_i/\eps}$, $\yt_j=e^{-y_j/\eps}$ the cost becomes $\sum_{ij}\xt_iK_{ij}\yt_j$. We recover the bilinear cost which has zero cross-curvature for the same reasons as in 1.

        \item This is proven by Wong and Yang in \cite{WongYang2022}. Since their notation slightly differ form ours and we defined $\S_c$ with a different multiplicative constant we redo the calculations here. Again by \Cref{prop:facts-cc}\ref{prop:facts-cc:gh} we may drop $u(x)$ and $u(y)$ and pretend that $c(x,y)=\frac1\alpha\log(1-\alpha\bracket{\nabla u(y),x-y})$. 
        
        We have $e^{\alpha c(x,y)} = 1-\alpha \bracket{\nabla u(y),x-y}$. Taking two $x$-derivatives this implies that $e^{\alpha c}[c_{ik} + \alpha c_i c_k]=0$ and thus $c_{ik} + \alpha c_ic_k=0$. Taking a $y$-derivative we obtain  $c_{i\jb k}+\alpha[c_{i\jb}c_k + c_{\jb k}c_i]=0$ and multiplying by the inverse matrix $c^{\jb m}$, 
        \[
            c_{i\jb k} c^{\jb m} = -\alpha[\delta_i^mc_k+\delta_k^mc_i].
        \]
        Then another $y$-derivative yields
        \[
            (c_{i\jb k\lb} - c_{ik\sbar}c^{\sbar t}c_{t\jb\lb}) c^{\jb m} = -\alpha [\delta_i^mc_{k\lb}+\delta_k^mc_{i\lb}].
        \]
        This implies that
        \[
            c_{ik\sbar}c^{\sbar t}c_{t\jb\lb} - c_{i\jb k\lb} = \alpha [c_{i\jb}c_{k\lb}+c_{\jb k}c_{i\lb}].
        \]
        \item This cost satisfies $e^{\alpha c}[c_{ik} + \alpha c_i c_k]=0$ which is the same identity as for the log-divergence, and then the proof carries the same way.

        \item \cite[Theorem 6.2]{KimMcCann2012}.
        \item \cite[Theorem 3.1]{KimMcCann2012}.
    \end{enumerate}
\end{proof}

\subsection{Proofs of Section~\ref{sec:gd}}\label{app:sec-gd}

\begin{proof}[Proof of \Cref{thm:differential-criterion-c-concavity}]
    \textbf{Direct implication:} we show that $f$ is $c$-concave.
    Let $\xb\in X$. We want to show that $f(\xb)=\inf_{y\in Y} c(\xb,y)+f^c(y)$. Since we always have $f(\xb)\leq \inf_{y\in Y} c(\xb,y)+f^c(y)$ it is sufficient to show that there exists $\yh\in Y$ such that
    \[
        f(\xb)\geq c(\xb,\yh)+f^c(\yh).
    \]
    By definition of $f^c$ this is equivalent to: for all $x\in X$,
    \begin{equation}\label{eq:proof-differential-criterion-1}
        f(x)-c(x,\yh)\leq f(\xb)-c(\xb,\yh).
    \end{equation}
    We choose $\yh$ to be a point satisfying $-\nabla_xc(\xb,\yh)=-\nabla f(\xb)$ (which exists by the assumptions of the theorem). We fix $x\in X$ and prove \eqref{eq:proof-differential-criterion-1}.
    
    Let $t\mapsto(x(t),\yh)$ be the $c$-segment joining $x(0)=\xb$ to $x(1)=x$. This path exists by \ref{ass:biconvex}. Let $a(t)=f(x(t))-c(x(t),\yh)$. Then \eqref{eq:proof-differential-criterion-1} can be written as 
    \[
        a(1)\leq a(0).
    \] 
    Let us show that $a(t)$ attains its maximum at $t=0$. Firstly note that by definition of $\yh$ we have that $a'(0)=0$. Secondly we show that $a(t)$ is a concave function of $t$. Write
    \[
        a(t)=f(x(t))-c(x(t),y) + [c(x(t),y)-c(x(t),\yh)],
    \]
    where $y$ is a yet unspecified point in $Y$. By \Cref{lemma:delta-c-segments} we know that the quantity in brackets is a concave function of $t$ and therefore it suffices to show that $b(t)\coloneqq f(x(t))-c(x(t),y)$ is concave. Differentiating $b(t)$ twice leads to 
    \[
        b''(t)=\big(\nabla^2f(x(t))-\nabla^2_{xx}c(x(t),y)\big)(\dot x,\dot x) + \bracket{\nabla f(x(t))-\nabla_{x}c(x(t),y),\ddot x}.
    \]
    We now fix $t$ and choose $y$ such that $\nabla f(x(t))=\nabla_{x}c(x(t),y)$; by our assumptions $a''(t)\leq 0$.

    \textbf{Reverse implication:} With the same configuration of points as above, if $f$ is $c$-concave then we have \eqref{eq:proof-differential-criterion-1}. Doing a Taylor expansion of the left-hand side in $x$ about $\xb$ directly implies
    \[
        \nabla^2f(\xb)-\nabla^2_{xx}c(\xb,\yh)\leq 0.
    \]
\end{proof}

\begin{proof}[Proof of \Cref{thm:cv_fb}]
    (i): By definition of the $c$-transform we can bound $f(x_{n+1}) \leq c(x_{n+1},y_{n+1})+f^c(y_{n+1})$. Since $f$ is $c$-concave, \eqref{algo:fb:step1} gives $f(x_n)=c(x_n,y_{n+1})+f^c(y_{n+1})$. Thus 
    \begin{equation}\label{eq:proof-fb-rates-1}
        f(x_{n+1}) \leq f(x_n)-c(x_n,y_{n+1}) + c(x_{n+1},y_{n+1}).
    \end{equation}
    On the other hand, by the $x$-update \eqref{algo:fb:step2} we have \begin{equation}\label{eq:proof-fb-rates-0}
        g(x_{n+1}) + c(x_{n+1},y_{n+1})\leq g(x_n)+c(x_n,y_{n+1}).
    \end{equation}
    Summing \eqref{eq:proof-fb-rates-1} and \eqref{eq:proof-fb-rates-0} yields
    \begin{equation}\label{eq:proof-fb-rates-2}
        f(x_{n+1})+g(x_{n+1})\leq f(x_n) +g(x_n).
    \end{equation}

    (ii), (iii): Let $\lambda \geq 0$. By \ref{ass:fb-c}, for any $n$ we can define $\yb_n\in \argmin_{y\in Y}c(x_n,y)$. By the envelope theorem in \Cref{lemma:envelope}, \ref{ass:fb-c} and $c$-concavity of $f$ we have respectively $\nabla_xc(x_n,\yb_n)=0$ and $-\nabla_xc(x_n,y_{n+1})=-\nabla f(x_n)$. Then strong cross-convexity of $f$ gives us
    \begin{multline}\label{eq:proof-fb-rates-3}
        f(x_n)\leq f(x)+[c(x,\yb_n)-c(x_n,\yb_n)]-c(x,y_{n+1})+c(x_n,y_{n+1})\\-\lambda (c(x,\yb_n)-c(x_n,\yb_n)).
    \end{multline}
    On the other hand we have that $\nabla_xc(x_{n+1},\yb_{n+1})=0$ and by \eqref{algo:fb:step1}, $-\nabla_xc(x_{n+1},y_{n+1})=\nabla g(x_{n+1})$. Then strong cross-concavity of $-g$ gives us 
    \begin{multline}\label{eq:proof-fb-rates-4}
        g(x_{n+1})\leq g(x)-[c(x,\yb_{n+1})-c(x_{n+1},\yb_{n+1})]+c(x,y_{n+1})-c(x_{n+1},y_{n+1})\\-\mu((c(x,\yb_{n+1})-c(x_{n+1},\yb_{n+1}))).
    \end{multline}
    After summing \eqref{eq:proof-fb-rates-3} and \eqref{eq:proof-fb-rates-4} and using that $c(x_n,\yb_n)=0$ by \ref{ass:fb-c}, we obtain 
    \begin{equation*}
        f(x_n)-c(x_n,y_{n+1})+c(x_{n+1},y_{n+1})+g(x_{n+1})\leq F(x)+(1-\lambda)c(x,\yb_n)- (1+\mu)c(x,\yb_{n+1}).
    \end{equation*}
    By \eqref{eq:proof-fb-rates-1} the left-hand side can be bounded from below by $F(x_{n+1})$ and we obtain
    \begin{equation}\label{eq:proof-fb-rates-5}
        F(x_{n+1})\leq F(x)+(1-\lambda)c(x,\yb_n)- (1+\mu)c(x,\yb_{n+1}).
    \end{equation}

    If $\lambda=\mu=0$ we directly sum from $0$ to $n-1$. By the descent property in (i) the left-hand side can be bounded from below by $n F(x_n)$. As for the right-hand side we obtain a telescopic sum that can be bounded above by $nF(x)+c(x,\yb_0)$ since $-c(x,\yb_{n})\leq 0$ by \ref{ass:fb-c}. We obtain the desired result.

    If $\lambda+\mu>0$  we proceed as in the proof of \Cref{thm:am-rates}, multiply both sides of \eqref{eq:proof-fb-rates-5} by $\left(\frac{1+\mu}{1-\lambda}\right)^{n+1}$ and sum from $0$ to $n-1$. We obtain the desired inequality after some simple algebra. 
\end{proof}

\begin{proof}[Proof of \Cref{prop:fb_fpp}]
    Since $f$ is $c$-concave we have $\inf_{y\in Y}\phi(x,y)=f(x)+g(x)=F(x)$. 
    Let us then prove the five-point inequality in the form \eqref{eq:fpp-F}: for all $x\in X$ and $y_0\in Y$,
    \begin{equation} \label{eq:proof-fb_fpp-1}
        F(x)\geq F(x_0)+\delta_\phi(x,y_0;x_0,y_1),
    \end{equation}
    with $x_0=\argmin_{x\in X} c(x,y_0)+g(x)$ and $y_1=\argmin_{y\in Y}c(x_0,y)+f^c(y)$. By first-order conditions we have $-\nabla_xc(x_0,y_0)=\nabla g(x_0)$ and by the envelope theorem in \Cref{lemma:envelope} we have that $-\nabla_xc(x_0,y_1)=-\nabla f(x_0)$. Moreover by \ref{ass:fb-c} and the envelope theorem again we know that there exists a point, call it $\yb_0\in Y$, such that $\nabla_xc(x_0,\yb_0)=0$. 
    
    For the points $x,x_0,\yb_0,y_1$, cross-convexity of $f$ implies
    \begin{equation}\label{eq:proof-fb_fpp-2}
        f(x)\geq f(x_0)+c(x,y_1)-c(x,\yb_0)-c(x_0,y_1)+c(x_0,\yb_0).
    \end{equation}
    On the other hand for the points $x,x_0,\yb_0,y_0$, cross-concavity of $-g$ gives 
    \begin{equation}\label{eq:proof-fb_fpp-3}
        -g(x)\leq -g(x_0)+c(x,y_0)-c(x,\yb_0)-c(x_0,y_0)+c(x_0,\yb_0).
    \end{equation}
    Combining \eqref{eq:proof-fb_fpp-2} and \eqref{eq:proof-fb_fpp-3} yields \eqref{eq:proof-fb_fpp-1}, after using that $\delta_\phi=\delta_c$ by \eqref{eq:prop_cross_diff}.
\end{proof}

\begin{proof}[Proof of \Cref{thm:sc-cross-conc}]
    This proof is very similar to the proof of \Cref{thm:sc-fpp+gd}, thus we only provide a sketch of proof and focus on the main differences with \Cref{thm:sc-fpp+gd}. Let us take $f=-g$ and show that if $f$ is concave on $c$-segments $(x(t),\yh)$ with $-\nabla_xc(x(0),\yh)=-\nabla f(x(0))$ then $f$ is $c$-cross-concave. Considering four points $x,\xb,\yb,\yh$ as usual we want to show that for a certain segment $x(t)$ we have 
    \[
        \bracket{\nabla f(\xb),\dot x(0)}\leq \delta_c(x,\yb;\xb,\yh)=[c(x,\yh)-c(x,\yb)]-[c(\xb,\yh)-c(\xb,\yb)].
    \]
    If it is $(x(t),\yh)$ that is a $c$-segment (and not $(x(t),\yb)$ like in \Cref{thm:sc-fpp+gd}) then by \Cref{lemma:delta-c-segments} the function $t\mapsto c(x(t),\yh)-c(x(t),\yb)$ is convex. As a consequence 
    \begin{align*}
        \delta_c(x,\yb;\xb,\yh)=&[c(x(1),\yh)-c(x(1),\yb)]-[c(x(0),\yh)-c(x(0),\yb)]\\
        &\geq \bracket{\nabla_xc(\xb,\yh)-\nabla_xc(\xb,\yb),\dot x(0)},
    \end{align*}
    and from then we can conclude as in the proof of \Cref{thm:sc-fpp+gd}. The proof of strong cross-concavity (ii) works in the same way.
\end{proof}

\subsection{Proofs of Section~\ref{sec:examples}} \label{app:sec-ex}

\begin{proof}[Proof of \Cref{prop:compare-cc-gc}]
    The idea of this proof is to obtain \eqref{eq:delta-and-path} not using $c$-segments and cross-curvature, but directly geodesics in $M$ and Riemannian curvature. This is because having nonnegative cross-curvature is always stronger than having nonnegative Riemannian curvature~\cite{Loeper2009}, and in general much stronger. 

    Fix $x,\xb,\yh\in M$ such that
    $-\nabla_xc(\xb,\yh)=-\nabla f(\xb)$. Let $\gamma$ be a geodesic joining $\gamma(0)=\xb$ to $\gamma(1)=x$ and denote by $\xi=\dot\gamma(0)$ its initial velocity. Points (i)--(iv) are consequences of the following observations: when $M$ has nonnegative curvature we have 
    \begin{equation}\label{eq:proof-prop-compare-cc-gc-3}
        \delta_c(x,\yb;\xb,\yh)\leq \bracket{\nabla f(\xb),\dot\gamma(0)},
    \end{equation}
    while if $M$ has nonpositive curvature,
    \begin{equation}\label{eq:proof-prop-compare-cc-gc-4}
        \delta_c(x,\yb;\xb,\yh)\geq \bracket{\nabla f(\xb),\dot\gamma(0)}.
    \end{equation}
    Both \eqref{eq:proof-prop-compare-cc-gc-3} and \eqref{eq:proof-prop-compare-cc-gc-4} follow from the so-called \emph{comparison theorems}. When $M$ has nonnegative curvature, the Toponogov compararison theorem~\cite[Chapter 11]{Petersen_book} says that 
    \begin{equation}\label{eq:toponogov}
        d^2\big(\exp_p(\xi),\exp_p(\eta)\big) \leq \abs{\xi-\eta}^2,
    \end{equation}
    for any $p\in M$, $\xi,\eta\in T_pM$. Here we denote by $\abs{\xi-\eta}^2=\mathsf{g}_p(\xi-\eta,\xi-\eta)$ the metric at $p$. Moreover when $M$ has nonpositive curvature, the Rauch comparison theorem~\cite[Chapter 10]{MR1138207} implies
    \begin{equation}\label{eq:rauch}
        d^2\big(\exp_p(\xi),\exp_p(\eta)\big) \geq \abs{\xi-\eta}^2.
    \end{equation}
    Taking $p=\xb$, $\xi=\dot\gamma(0)$, $\eta=-\frac1L\nabla_xc(\xb,\yh)$, using that $\eta=-\frac1L\nabla f(\xb)$ and the identity \eqref{eq:cross-diff-riem}, then \eqref{eq:toponogov} gives \eqref{eq:proof-prop-compare-cc-gc-3} and \eqref{eq:rauch} gives \eqref{eq:proof-prop-compare-cc-gc-4} after a bit of algebra.

    (i): If $f$ is geodesically convex then $f(x)-f(\xb)\geq \bracket{\nabla f(\xb),\dot\gamma(0)}$ and we can use \eqref{eq:proof-prop-compare-cc-gc-3} to deduce cross-convexity.

    (ii): If $f=-g$ is cross-concave then \eqref{eq:proof-prop-compare-cc-gc-3} gives $f(x)-f(\xb)\leq \bracket{\nabla f(\xb),\dot\gamma(0)}$ and we can deduce geodesic concavity of $f$ which is geodesic convexity of $g$.
    
    (iii): If $f$ is cross-convex then \eqref{eq:proof-prop-compare-cc-gc-4} implies $f(x)-f(\xb)\geq \bracket{\nabla f(\xb),\dot\gamma(0)}$ and we can deduce geodesic convexity.

    (iv): If $f=-g$ is geodesically concave then $f(x)-f(\xb)\leq \bracket{\nabla f(\xb),\dot\gamma(0)}$ and we can use \eqref{eq:proof-prop-compare-cc-gc-4} to deduce cross-concavity.
\end{proof}

\begin{proof}[Proof of \Cref{prop:c-concave-equiv}]
    (i): This can be found in \cite[Prop.~16.2]{Villani_book_2009}, up to sign changes.
    
    (ii), (iii): These are similar to the proof of \Cref{prop:compare-cc-gc}. Take three points $x,\xb,\yh\in M$ such that $-\nabla_xc(\xb,\yh)=-\nabla f(\xb)$. Let $\xi,\eta\in T_\xb M$ such that $\exp_\xb(\xi)=x$ and $\exp_\xb(\eta)=\yh$ i.e.\ $\eta=-\frac1L\nabla f(\xb)$. As in the proof of \Cref{thm:differential-criterion-c-concavity}, $c$-concavity says that $f(x)-f(\xb)\leq \frac{L}{2}\big(d^2(x,\yh)-d^2(\xb,\yh)\big)$, which can be written as 
    \begin{equation*}
        f(x)-f(\xb)\leq \frac{L}{2}\big(d^2(x,\yh)-\abs{\eta}^2\big).
    \end{equation*}
    As for \ref{riem:bound-Hessian} by (i) it is equivalent to $f(x)-f(\xb)\leq \bracket{\nabla f(\xb),\xi}+ \frac{L}{2}d^2(x,\xb)$ which can also be written as 
    \[
        f(x)-f(\xb)\leq \frac{L}{2}\big(-2\bracket{\eta,\xi}+ \abs{\xi}^2\big)=\frac{L}{2}\big(\abs{\xi-\eta}^2-\abs{\eta}^2\big).
    \]
    We see that all is left to compare are the quantities $\abs{\xi-\eta}^2$ and $d^2(x,\yh)=d^2(\exp_\xb(\xi),\exp_\xb(\eta))$. The implications \ref{riem:c-concave}$\implies$\ref{riem:bound-Hessian} and \ref{riem:bound-Hessian}$\implies$\ref{riem:c-concave} then work as in \Cref{prop:compare-cc-gc}, using comparison theorems. 

    (iv): \cite[Lemma 2.1]{MR3634803} shows \ref{riem:Lip-gradient}$\implies$\ref{riem:f-bound-above}, Then use (i).
\end{proof}
\subsection{Estimate sequences}

Given $F:X\rightarrow \R$, a sequence $(\phi_n,\lambda_n)_{n\in\N}$, where $\phi_n:X\rightarrow \R$ and $\lambda_n>0$, is a weak estimate sequence at $x_*$ \cite[p6]{zhang2018estimate} if
\begin{equation}\label{eq:est-seq_def}
       \lim_{n\rightarrow \infty} \lambda_n=0, \text{ and }  \phi_{n}(x_*) \le (1-\lambda_n) F(x_*) + \lambda_n \phi_0(x_*).
\end{equation}
If there exists $(x_n)_{n\in\N}\in X^\N$, such that $F(x_n)\le \phi_n(\cdot)$, then an immediate result follows \cite[Lemma 2.2.1]{Nesterov2018}
\begin{equation}\label{eq:est-seq_cv }
       F(x_n)-F(x_*) \le \phi_n(x_*)-F(x_*) \le \lambda_n (\phi_0(x_*)-F(x_*)).
\end{equation}
Introducing a sequence $(\alpha_n)_{n\in\N}\in[0,1]^\N$ whose sum diverges, a practical way to build such estimate sequences is to find $(\phi_n)_{n\in\N}$ satisfying
\begin{equation}\label{eq:est-seq_proof}
        \phi_{n+1}(x_*) \le (1-\alpha_n) \phi_{n}(x_*) + \alpha_n F_n(x_*),
\end{equation}
with $F_n$ a lower bound of $F$. Then, by \cite[Proposition 2.2]{Baes2009estimate}, for $\lambda_{n}\coloneqq\Pi_{i=0}^{n-1}(1- \alpha_i)$, $(\phi_n,\lambda_n)_{n\in\N}$ is a weak estimate sequence at $x_*$.

\begin{lemma}[Estimate sequence from \eqref{sfpp}]\label{lem:estimate-seq}
    Consider $\phi:X\times Y \rightarrow \R$ satisfying \eqref{sfpp} for some $\lambda \in [0,1)$, and set $F(x)\coloneqq\inf_{y\in Y} \phi(x,y)$ and $\phi_n(x)\coloneqq\phi(x,y_{n+1})$. Then, defining
    \begin{equation}\label{eq:rate_estimate_seq_app}
        \alpha_n=\max\left(\lambda \frac{\phi(x_*,y_{n+1})-\phi(x_{n+1},y_{n+1})}{(1-\lambda)(\phi(x_*,y_{n})-\phi(x_{n},y_{n}))}, \frac{\phi(x_{n+1},y_{n+1})-f_*}{\phi(x_*,y_{n})-\phi(x_{n},y_{n})}\right),
    \end{equation}
    Then either $\sum\alpha_n$ converges and $\phi(x_{n},y_{n})-f_*=\calO(1/n)$, or $\sum \alpha_n$ diverges and
    $(\phi_n,\lambda_n)_{n\in\N}$ is a weak estimate sequence at $x_*$ for $\lambda_{n}\coloneqq\Pi_{i=0}^{n-1}(1- \alpha_i)$, and $F(x_n)-f_* \le \lambda_n (\phi(x_*,y_0)-f_*)$.
\end{lemma}
\begin{proof} The \eqref{sfpp} property gives us a lower bound on $F$  
    \[F_n(x)\coloneqq \phi(x,y_{n+1})+ (1-\lambda)(\phi(x_{n},y_{n}) - \phi(x,y_{n})).\]
    Using the definition of $\phi_n$ and doing some rearrangements, \eqref{eq:est-seq_proof} becomes
    \begin{equation}\label{eq:est-seq_proof2}
        \phi(x_*,y_{n+2}) \le \phi(x_*,y_{n+1}) + \alpha_n (1-\lambda)(\phi(x_{n},y_{n}) - \phi(x_*,y_{n})).
    \end{equation}
    However \eqref{sfpp} applied at $x_*$ gives
    \begin{multline}
        \phi(x_*,y_{n+1}) - \phi(x_*,y_{n+2}) \ge \lambda \phi(x_*,y_{n+1})-f_*+(1-\lambda)\phi(x_{n+1},y_{n+1})\\
        \ge \max\left(\phi(x_{n+1},y_{n+1})-f_*,\lambda (\phi(x_*,y_{n+1})-\phi(x_{n+1},y_{n+1}))\right)\ge 0.
    \end{multline}
    By definition of $\alpha_n$, \eqref{eq:est-seq_proof2} thus holds. Moreover, applying again \eqref{sfpp}, we have
    \begin{equation*}
        \phi(x_*,y_{n+1})-\phi(x_{n+1},y_{n+1})\le \phi(x_*,y_{n+1}) -f_* \le (1-\lambda)(\phi(x_*,y_{n})-\phi(x_{n},y_{n}))
    \end{equation*}
    and, by definition of the iterates, we have $f_*\le\phi(x_{n+1},y_{n+1})\le \phi(x_{n},y_{n})\le \phi(x_*,y_{n})$. So $\alpha_n\in[0, 1]$.

    If $C=\sum\alpha_n<\infty$, then, using the descent property of the iterates, we obtain that
    \begin{equation*}
        C \ge \sum^{n-1}_{i=0}\alpha_i\ge \sum^{n-1}_{i=0} \frac{\phi(x_{i+1},y_{i+1})-f_*}{\phi(x_*,y_{i})-\phi(x_{i},y_{i})}\ge n\frac{\phi(x_{n},y_{n})-f_*}{\phi(x_*,y_{0})-\phi(x_{0},y_{0})}.
    \end{equation*}
    Otherwise, if $\sum\alpha_n=\infty$, we conclude by applying \eqref{eq:est-seq_cv }.
\end{proof}

\bibliographystyle{amsalphaurl}
\bibliography{general_gradient_descent}

\end{document}